\documentclass[a4paper,12pt,reqno]{amsart}

\usepackage[T1]{fontenc}
\usepackage[utf8]{inputenc}
\usepackage{lmodern}
\usepackage[english]{babel}

\usepackage[%
	left=2.5cm,       
	right=2.5cm,      
	top=3.5cm,        
	bottom=3.5cm,     
	heightrounded,    
	bindingoffset=0mm 
]{geometry}

\usepackage{amsmath}
\usepackage{amssymb}
\usepackage{mathtools}
\usepackage{mathrsfs}
\usepackage[normalem]{ulem}
\numberwithin{equation}{section}

\usepackage{hyperref} 
\usepackage[noabbrev,capitalize]{cleveref}

\usepackage{bookmark}

\theoremstyle{plain}
	\newtheorem{theorem}{Theorem}[section]
	\newtheorem{lemma}[theorem]{Lemma}
	\newtheorem{proposition}[theorem]{Proposition}
	\newtheorem{corollary}[theorem]{Corollary}

\theoremstyle{definition}
	\newtheorem{definition}[theorem]{Definition}
	
	\newtheorem{remark}[theorem]{Remark}
	\newtheorem{open.problem}[theorem]{Open Problem}

\newcommand{\N}{\mathbb{N}}
\newcommand{\R}{\mathbb{R}}

\newcommand{\eps}{\varepsilon}

\usepackage[initials]{amsrefs}

\usepackage[dvipsnames]{xcolor}
\usepackage{tikz}
\usepackage{graphicx}

\usepackage[font=small]{caption}

\newcommand{\closure}[2][3]{%
  {}\mkern#1mu\overline{\mkern-#1mu#2}}
  
\usepackage{enumerate}

\usepackage{url}

\newcommand{\de}{\partial}

\newcommand{\longto}{\longrightarrow}
\newcommand{\weakto}{\rightharpoonup}

\renewcommand{\phi}{\varphi}
\renewcommand{\rho}{\varrho}
\renewcommand{\theta}{\vartheta}

\DeclareMathOperator{\dist}{dist}

\DeclareMathOperator{\supp}{supp}
\DeclareMathOperator{\diverg}{div}
\DeclareMathOperator{\Lip}{Lip}
\DeclareMathOperator{\loc}{loc}
					    
\DeclarePairedDelimiter{\scalar}{<}{>}                                     
\DeclarePairedDelimiter{\set}{\{}{\}}
\DeclarePairedDelimiter{\abs}{|}{|}

\mathchardef\ordinarycolon\mathcode`\:
\mathcode`\:=\string"8000
\begingroup \catcode`\:=\active
  \gdef:{\mathrel{\mathop\ordinarycolon}}
\endgroup

\newcommand{\Haus}[1]{\mathscr{H}^{#1}} 
\newcommand{\Leb}[1]{\mathscr{L}^{#1}} 
\renewcommand{\div}{\mathrm{div}} 

\allowdisplaybreaks

\begin{document}

\title[A distributional approach to fractional variation: asymptotics II]{A distributional approach to fractional Sobolev spaces and fractional variation: asymptotics II}

\author[E.~Bruè]{Elia Bruè}
\address[E.~Bruè]{School of Mathematics, Institute for Advanced Study, 1 Einstein Dr., Princeton NJ 05840, USA}
\email{elia.brue@math.ias.edu}

\author[M.~Calzi]{Mattia Calzi}
\address[M.~Calzi]{Dipartimento di Matematica, Università degli Studi di Milano, Via C. Saldini~50, 20133 Milano, Italy}
\email{mattia.calzi@unimi.it}

\author[G.~E.~Comi]{Giovanni E. Comi}
\address[G.~E.~Comi]{Dipartimento di Matematica, Università di Pisa, Largo Bruno Pontecorvo 5, 56127 Pisa, Italy}
\email{giovanni.comi@dm.unipi.it}

\author[G.~Stefani]{Giorgio Stefani}
\address[G.~Stefani]{Department Mathematik und Informatik, Universit\"at Basel, Spiegelgasse 1, CH-4051 Basel, Switzerland}
\email{giorgio.stefani@unibas.ch}

\date{\today}



\thanks{
\textit{Acknowledgements}. 
The authors thank Daniel Spector for many valuable observations on a preliminary version of the present work. 
The first author is supported by the \textit{Giorgio and Elena Petronio Fellowship} at the Institute for Advanced Study. 
The second, the third and the fourth authors are members of the \textit{Gruppo Nazionale per l'Analisi Matematica, la Probabilità e le loro Applicazioni} (GNAMPA) of the \textit{Istituto Nazionale di Alta Matematica} (INdAM).
The fourth author is partially supported by the ERC Starting Grant 676675 FLIRT -- \textit{Fluid Flows and Irregular Transport} and by the INdAM--GNAMPA Project 2020 \textit{Problemi isoperimetrici con anisotropie} (n.\ prot.\ U-UFMBAZ-2020-000798 15-04-2020).
This work was started while the authors were Ph.D.\ students at the Scuola Normale Superiore of Pisa (Italy) and mostly developed while the first two authors and the last one were still employed there. 
The authors wish to express their gratitude to this institution for the excellent working conditions and the stimulating atmosphere.
The third author worked on this paper during his PostDoc at the Department of Mathematics of the University of Hamburg (Germany), and he is grateful for the support received.
}

\begin{abstract}
We continue the study of the space $BV^\alpha(\R^n)$ of functions with bounded fractional variation in~$\R^n$ and of the distributional fractional Sobolev space $S^{\alpha,p}(\R^n)$, with $p\in [1,+\infty]$ and $\alpha\in(0,1)$, considered in the previous works~\cites{CS19,CS19-2}. 
We first define the space $BV^0(\R^n)$ and establish the identifications $BV^0(\R^n)=H^1(\R^n)$ and $S^{\alpha,p}(\R^n)=L^{\alpha,p}(\R^n)$, where $H^1(\R^n)$ and $L^{\alpha,p}(\R^n)$ are the (real) Hardy space and the Bessel potential space, respectively. 
We then prove that the fractional gradient $\nabla^\alpha$ strongly converges to the Riesz transform as~$\alpha\to0^+$ for $H^1\cap W^{\alpha,1}$ and $S^{\alpha,p}$ functions.
We also study the convergence of the $L^1$-norm of the $\alpha$-rescaled fractional gradient of $W^{\alpha,1}$ functions.
To achieve the strong limiting behavior of~$\nabla^\alpha$ as~$\alpha\to0^+$, we prove some new fractional interpolation inequalities which are stable with respect to the interpolating parameter.
\end{abstract}

\maketitle

\tableofcontents

\section{Introduction}

\subsection{Fractional operators and related spaces}

In~\cites{CS19,CS19-2}, for a parameter $\alpha\in(0,1)$, the third and fourth authors introduced the \emph{space of functions with bounded fractional variation}
\begin{equation*}
BV^\alpha(\R^n)
:=
\set*{f\in L^1(\R^n) : |D^\alpha f|(\R^n)<+\infty},
\end{equation*}
where
\begin{equation}\label{intro_eq:frac_variation}
|D^\alpha f|(\R^n):=
\sup\set*{\int_{\R^n} f\,\div^\alpha\phi\,dx : \phi\in C^\infty_c(\R^n;\R^n),\ \|\phi\|_{L^\infty(\R^n;\,\R^n)}\le1}
\end{equation}
for all $f\in L^1(\R^n)$, and the \emph{distributional fractional Sobolev space}
\begin{equation}
\label{intro_eq:S_alpha_p}
S^{\alpha,p}(\R^n):=\set*{f\in L^p(\R^n) : \exists \nabla^\alpha f \in L^p(\R^n;\R^n)}
\end{equation}
for all $p\in[1,+\infty]$ (see Section \ref{sect:overview_fract_grad} for a precise definition). 
Here and in the following,
\begin{equation}\label{intro_eq:nabla_alpha}
\nabla^\alpha f(x):=\mu_{n,\alpha}\int_{\R^n}\frac{(y-x)(f(y)-f(x))}{|y-x|^{n+\alpha+1}}\,dy,
\quad
x\in\R^n,	
\end{equation}
and
\begin{equation}\label{intro_eq:div_alpha}
\div^\alpha\phi(x):=\mu_{n,\alpha}\int_{\R^n}\frac{(y-x)\cdot(\phi(y)-\phi(x))}{|y-x|^{n+\alpha+1}}\,dy,
\quad
x\in\R^n,	
\end{equation}
are respectively the \emph{fractional gradient} and the \emph{fractional divergence} operators, where
\begin{equation}\label{intro_eq:mu_n_alpha}
\mu_{n, \alpha} := 2^{\alpha} \pi^{- \frac{n}{2}} \frac{\Gamma\left ( \frac{n + \alpha + 1}{2} \right )}{\Gamma\left ( \frac{1 - \alpha}{2} \right )}.
\end{equation}
These two operators are \emph{dual}, in the sense that
\begin{equation*}
\int_{\R^n}f\,\div^\alpha\phi \,dx=-\int_{\R^n}\phi\cdot\nabla^\alpha f\,dx
\end{equation*}
for all sufficiently regular functions~$f$ and vector fields~$\phi$.
For an account on the existing literature related to these operators, we refer the reader to~\cites{BCM20,BCM21,CS19,CS19-2,H59,P16,SSS15,SSS18,SSS17,SS15,SS18,Sil19,S18,S19} and to the references therein.

While the first paper~\cite{CS19} was focused on some geometric aspects of $BV^\alpha$ functions, the subsequent work~\cite{CS19-2} was inspired by the celebrated Bourgain--Brezis--Mironescu formula~\cite{BBM01} and the $\Gamma$-convergence result of Ambrosio--De Philippis--Martinazzi~\cite{ADM11} and dealt with the asymptotic behavior of the fractional $\alpha$-variation as~$\alpha\to1^-$. 
As already announced in~\cite{CS19-2}, the main aim of present paper is to study the asymptotic behavior of the fractional $\alpha$-variation as~$\alpha\to0^+$, in analogy with the asymptotic result of Maz\cprime ya--Shaposhnikova~\cites{MS02,MS03}.

\subsection{Asymptotic behavior of fractional operators}

The asymptotic behavior of the standard fractional seminorm~$[\,\cdot\,]_{W^{\alpha,p}(\R^n)}$ was completely understood since the groundbreaking work of Bour\-gain--Bre\-zis--Mi\-ro\-ne\-scu~\cite{BBM01} and the subsequent developments of D\'avila~\cite{D02} and Maz\cprime ya--Sha\-po\-shni\-ko\-va~\cites{MS02,MS03}.
Here and in the following,
\begin{equation*}
W^{\alpha,p}(\R^n)=\set*{f\in L^p(\R^n) : [f]_{W^{\alpha,p}(\R^n)}^p=\int_{\R^{n}} \int_{\R^{n}} \frac{|f(x)-f(y)|^p}{|x-y|^{n+p\alpha}}\,dx\,dy<+\infty}
\end{equation*}
is the well-known Sobolev--Slobodeckij space of parameters $\alpha\in(0,1)$ and $p\in[1,+\infty)$ (see~\cite{DiNPV12} for an introduction and the related literature).
Precisely, for $p\in[1,+\infty)$,
\begin{equation}\label{intro_eq:MS_1}
\lim_{\alpha\to1^-}(1-\alpha)\,[f]_{W^{\alpha,p}(\R^n)}^p
=A_{n,p}\,\|\nabla f\|_{L^p(\R^n; \R^{n})}^p
\end{equation}   
for all $f\in W^{1,p}(\R^n)$, while
\begin{equation}\label{intro_eq:MS_0}
\lim_{\alpha\to0^+}\alpha\,[f]_{W^{\alpha,p}(\R^n)}^p
=B_{n,p}\,\|f\|_{L^p(\R^n)}^p
\end{equation}   
for all $f\in\bigcup_{\alpha\in(0,1)}W^{\alpha,p}(\R^n)$. 
Here $A_{n,p},B_{n,p}>0$ are two constants depending uniquely on~$n$ and~$p$. 
When $p=1$, the limit in~\eqref{intro_eq:MS_1} holds for the more general class of $BV$ functions, that is,
\begin{equation}\label{intro_eq:Davila_limit}
\lim_{\alpha\to1^-}(1-\alpha)\,[f]_{W^{\alpha,1}(\R^n)}
=A_{n,1}\,|Df|(\R^n)
\end{equation}   
for all $f\in BV(\R^n)$.

The limits in~\eqref{intro_eq:MS_1} and in~\eqref{intro_eq:Davila_limit} can be recognized as special consequences of the celebrated Bourgain--Brezis--Mironescu (BBM, for short) formula
\begin{equation}
\label{intro_eq:BBM}
\lim_{k\to+\infty}
\int_{\R^n}\int_{\R^n}
\frac{|f(x)-f(y)|^p}{|x-y|^p}\,\rho_k(|x-y|)\,dx\,dy
=
\begin{cases}
C_{n,p}\,\|\nabla f\|_{L^p(\R^n)}^p 
& \text{for}\ p\in(1,+\infty),\\[3mm]
C_{n,1}\,|Df|(\R^n)
& \text{for}\ p=1,
\end{cases}
\end{equation}
where $C_{n,p}>0$ is a constant depending only on~$n$ and~$p$, and $(\rho_k)_{k\in\N}\subset L^1_{\loc}([0,+\infty))$ is a sequence of non-negative radial mollifiers such that
\begin{equation*}
\int_{\R^n}\rho_k(|x|)\,dx=1
\quad
\text{for all $k\in\N$}
\end{equation*}
and
\begin{equation*}
\lim_{k\to+\infty}
\int_\delta^{+\infty}
\rho_k(r)\,r^{n-1}\,dr=0
\quad
\text{for all $\delta>0$.}
\end{equation*}
Since its appearance, the BBM formula~\eqref{intro_eq:BBM} has deeply influenced the development of the asymptotic analysis in the fractional framework.
On the one hand, the limit in~\eqref{intro_eq:BBM} has led to several important applications, such as Brezis' celebrated work~\cite{B02} on how to recognize constant functions, new characterizations of Sobolev and $BV$ functions and $\Gamma$-convergence results~\cites{AGMP18,AGMP20,AGP20,BMR20,BN06,LS11,LS14,N07,N08,N11,P04-2}, approximation of Sobolev norms and image processing~\cites{B15,BN16-bis,BN18,BN20}, and last but not least fractional Hardy and Poincaré inequalities~\cites{BBM02-bis,FS08,P04-1}.
On the other hand, the BBM formula~\eqref{intro_eq:BBM} has inspired an alternative route to fractional asymptotic analysis by means of interpolation techniques~\cites{M05,PS17}.
Recently, the BBM formula in~\eqref{intro_eq:BBM} has been revisited in terms of a.e.\ pointwise convergence by Brezis--Nguyen~\cite{BN16} and in connection with weak $L^p$ quasi-norms~\cite{BVY20}, where the now-called Brezis--Van Schaftingen--Yung space
\begin{equation*}
BSY^{\alpha,p}(\R^n)
=
\set*{f\in L_{\loc}^1(\R^n) : \left\|\frac{|f(x)-f(y)|}{|x-y|^{\frac np+\alpha}}\right\|_{L^p_w(\R^n\times\R^n)}<+\infty},
\end{equation*}
defined for $\alpha\in(0,1]$ and $p\in[1,+\infty)$, has offered a completely new and promising perspective in the field~\cite{DM20}. 

The limits \eqref{intro_eq:MS_1} -- \eqref{intro_eq:BBM} have been linked to variational problems~\cite{AK09}, generalized to various function spaces,  such as Besov spaces~\cites{KL05,T11}, Orlicz spaces~\cites{ACPS20,FHR20,FS19} and magnetic and anisotropic Sobolev spaces~\cites{LMP19,NS19,PSV17,PSV19,SV16}, and extended to several ambient spaces, such as compact connected Riemannian manifolds~\cite{KM19}, the flat torus~\cite{A20}, Carnot groups~\cites{B11,MP19} and complete doubling metric-measure spaces supporting a local Poincaré inequality~\cite{DiMS19}.

The asymptotic behavior of the fractional gradient $\nabla^\alpha$ as $\alpha\to1^-$ was fully discussed in~\cite{CS19-2} (see also~\cite{BCM21}*{Theorem~3.2} for a different proof of~\eqref{intro_eq:CS_limit_p} below for the case $p\in(1,+\infty)$ via Fourier transform).
Precisely, if $f\in W^{1,p}(\R^n)$ for some $p\in[1,+\infty)$, then $f\in S^{\alpha,p}(\R^n)$ for all $\alpha\in(0,1)$ with
\begin{equation}\label{intro_eq:CS_limit_p}
\lim_{\alpha\to1^-}\|\nabla^\alpha f-\nabla f\|_{L^p(\R^n;\,\R^n)}=0.
\end{equation}
If $f\in BV(\R^n)$ instead, then $f\in BV^\alpha(\R^n)$ for all $\alpha\in(0,1)$ with
\begin{equation*}
D^\alpha f\weakto Df\ 
\text{in $\mathscr{M}(\R^n;\R^n)$ and}\ 
|D^\alpha f|\weakto |Df|\
\text{in $\mathscr{M}(\R^n)$ as $\alpha\to1^-$}
\end{equation*}
and 
\begin{equation}\label{intro_eq:CS_limit_BV}
\lim_{\alpha\to1^-}|D^\alpha f|(\R^n)=|Df|(\R^n).
\end{equation}
We underline that, differently from the limits~\eqref{intro_eq:MS_1} and~\eqref{intro_eq:Davila_limit}, the renormalizing factor $(1-\alpha)^\frac{1}{p}$ does not appear in~\eqref{intro_eq:CS_limit_p} and~\eqref{intro_eq:CS_limit_BV}.
This is motivated by the fact that the constant~$\mu_{n,\alpha}$ encoded in the definition~\eqref{intro_eq:nabla_alpha} of the operator~$\nabla^\alpha$ satisfies
\begin{equation*}
\mu_{n,\alpha}\sim\frac{1-\alpha}{\omega_n}
\quad
\text{as $\alpha\to1^-$}.
\end{equation*}

Concerning the asymptotic behavior of~$\nabla^\alpha$ as~$\alpha\to0^+$, at least for sufficiently regular functions, the fractional gradient in~\eqref{intro_eq:nabla_alpha} is converging to the operator
\begin{equation}\label{intro_eq:def_nabla_0}
\nabla^0 f(x)=\mu_{n,0}\int_{\R^n}\frac{(y-x)(f(y)-f(x))}{|y-x|^{n+1}}\,dy,
\quad
x\in\R^n.	
\end{equation}
Here and in the following, $\mu_{n,0}$ is simply the limit of the constant~$\mu_{n,\alpha}$ defined in~\eqref{intro_eq:mu_n_alpha} as~$\alpha\to0^+$ (thus, in this case, no renormalization factor has to be taken into account).
The operator in~\eqref{intro_eq:def_nabla_0} is well defined (in the principal value sense) at least for all $f\in C^\infty_c(\R^n)$ and, actually, coincides (possibly up to a minus sign, see \cref{subsec:notation} below) with the well-known vector-valued \emph{Riesz transform}~$Rf$, see~\cites{G14-C,S70,S93}.
The formal limit $\nabla^\alpha\to R$ as~$\alpha\to0^+$ can be also motivated either by the asymptotic behavior of the Fourier transform of~$\nabla^\alpha$ as~$\alpha\to0^+$ or by the fact that $\nabla^\alpha=\nabla I_{1-\alpha}\to\nabla I_1=R$ for~$\alpha\to0^+$, where
\begin{equation*}
I_{\alpha} f(x) := 
2^{-\alpha} \pi^{- \frac{n}{2}} \frac{\Gamma\left(\frac{n-\alpha}2\right)}{\Gamma\left(\frac\alpha2\right)}
\int_{\R^{n}} \frac{f(y)}{|x - y|^{n - \alpha}} \, dy, 
\quad
x\in\R^n,
\end{equation*}
stands for the Riesz potential of order $\alpha\in(0,n)$.
In a similar fashion, the fractional $\alpha$-divergence in~\eqref{intro_eq:div_alpha} is converging as $\alpha\to0^+$ to the operator
\begin{equation*}
\div^0\phi(x)=\mu_{n,0}\int_{\R^n}\frac{(y-x)\cdot(\phi(y)-\phi(x))}{|y-x|^{n+1}}\,dy,
\quad
x\in\R^n,	
\end{equation*} 
which is well defined (in the principal value sense) at least for all $\phi\in C^\infty_c(\R^n;\R^n)$. 

As a natural target space for the study of the limiting behavior of~$\nabla^\alpha$ as~$\alpha\to0^+$, in a\-na\-lo\-gy with the fractional variation~\eqref{intro_eq:frac_variation}, we introduce the space $BV^0(\R^n)$ of functions $f\in L^1(\R^n)$ such that the quantity
\begin{equation*}
|D^0 f|(\R^n):=\sup\set*{\int_{\R^n}f\,\div^0\phi\,dx : \phi\in C^\infty_c(\R^n;\R^n),\ \|\phi\|_{L^\infty(\R^n;\,\R^n)}\le1}
\end{equation*}
is finite.
As for the $BV^\alpha$ space, it is not difficult to see that a function $f\in L^1(\R^n)$ belongs to $BV^0(\R^n)$ if and only if there exists a vector-valued Radon measure $D^0 f\in\mathscr M(\R^n;\R^n)$ with finite total variation such that
\begin{equation*}
\int_{\R^n}f\,\div^0\phi\,dx
=
-\int_{\R^n} \phi\cdot\,d D^0 f
\quad
\text{for all $\phi\in C_c^\infty(\R^n;\R^n)$}.
\end{equation*}
Surprisingly,  it turns out that $D^0f\ll\Leb{n}$ for all $f\in BV^0(\R^n)$, in contrast with what is known for the fractional $\alpha$-variation in the case $\alpha\in(0,1]$, see~\cite{CS19}*{Theorem~3.30}. 
More precisely, we prove that
\begin{equation}\label{intro_eq:H_1=BV_0}
f\in BV^0(\R^n)
\iff
f\in H^1(\R^n),\ 
\text{with $D^0f=Rf\Leb{n}$ in~$\mathscr M(\R^n;\R^n)$},
\end{equation}
where 
\begin{equation*}
H^1(\R^n)=\set*{f\in L^1(\R^n) : Rf\in L^1(\R^n;\R^n)}
\end{equation*}
is the well-known (real) \emph{Hardy space}.

Having the identification~\eqref{intro_eq:H_1=BV_0} at disposal, we can rigorously establish the validity of the convergence $\nabla^\alpha\to R$ as~$\alpha\to0^+$. 
For $p=1$, we prove that 
\begin{equation}\label{intro_eq:frac_limit_1}
\lim_{\alpha\to0^+}\|\nabla^\alpha f-Rf\|_{L^1(\R^n;\,\R^n)}=0
\end{equation}
for all $f\in H^1(\R^n)\cap\bigcup_{\alpha\in(0,1)}W^{\alpha,1}(\R^n)$.
For $p\in(1,+\infty)$ instead, since the Riesz transform~\eqref{intro_eq:def_nabla_0} extends to a linear continuous operator $R\colon L^p(\R^n)\to L^p(\R^n;\R^n)$, the natural target space for the study of the limiting behavior of the fractional gradient is simply $L^p(\R^n;\R^n)$. 
In this case, we prove that
\begin{equation}\label{intro_eq:frac_limit_p}
\lim_{\alpha\to0^+}\|\nabla^\alpha f-Rf\|_{L^p(\R^n;\,\R^n)}=0
\end{equation}
for all $f\in\bigcup_{\alpha\in(0,1)}S^{\alpha,p}(\R^n)$. 

The limits in~\eqref{intro_eq:frac_limit_1} and~\eqref{intro_eq:frac_limit_p} can be considered as the counterparts of~\eqref{intro_eq:MS_0} in our fractional setting. 
However, differently from~\eqref{intro_eq:MS_0}, in~\eqref{intro_eq:frac_limit_1} and in~\eqref{intro_eq:frac_limit_p} we obtain strong convergence.
This improvement can be interpreted as a natural consequence of the fact that, generally speaking, the $L^p$-norm of the fractional gradient~$\nabla^\alpha$ allows for more cancellations than the $W^{\alpha,p}$-seminorm.

Since the Riesz transform~\eqref{intro_eq:def_nabla_0} extends to a linear continuous operator $R\colon H^1(\R^n)\to H^1(\R^n;\R^n)$, the limit in~\eqref{intro_eq:frac_limit_1} can be improved. 
Precisely, we prove that
\begin{equation}\label{intro_eq:frac_limit_H1}
\lim_{\alpha\to0^+}\|\nabla^\alpha f-Rf\|_{H^1(\R^n;\,\R^n)}=0
\end{equation}
for all $f\in\bigcup_{\alpha\in(0,1)} HS^{1,\alpha}(\R^n)$.
Here 
\begin{equation*}
HS^{\alpha,1}(\R^n)=
\set*{f\in H^1(\R^n) : \nabla^\alpha f\in H^1(\R^n;\R^n)}
\end{equation*}
is (an equivalent definition of) the fractional Hardy--Sobolev space, see~\cite{Str90} and below for a more detailed presentation. 
One can recognize that
\begin{equation*}
H^1(\R^n)\cap\bigcup_{\alpha\in(0,1)}W^{\alpha,1}(\R^n)= \bigcup_{\alpha\in(0,1)} HS^{\alpha,1}(\R^n),	
\end{equation*}
so that~\eqref{intro_eq:frac_limit_H1} is indeed a reinforcement of~\eqref{intro_eq:frac_limit_1}.

Naturally, if $f\notin H^1(\R^n)$, then we cannot expect that $\nabla^\alpha f\to Rf$ in~$L^1(\R^n;\R^n)$ as~$\alpha\to0^+$.
Instead, as suggested by the limit in~\eqref{intro_eq:MS_0}, we have to consider the asymptotic behavior of the rescaled fractional gradient $\alpha\,\nabla^\alpha f$ as~$\alpha\to0^+$. 
In this case, we prove that
\begin{equation}\label{intro_eq:mattia_limit}
\lim_{\alpha\to0^+}
\alpha\int_{\R^n}|\nabla^\alpha f(x)|\,dx
=n\omega_n\mu_{n,0}\,\bigg|\int_{\R^n} f(x)\,dx\,\bigg|.
\end{equation} 
for all $f\in\bigcup_{\alpha\in(0,1)}W^{\alpha,1}(\R^n)$.
Note that~\eqref{intro_eq:mattia_limit} is consistent with both~\eqref{intro_eq:MS_0} and~\eqref{intro_eq:frac_limit_1}. 
Indeed, on the one side, by simply bringing the modulus inside the integral in the definition~\eqref{intro_eq:nabla_alpha} of~$\nabla^\alpha$, we can estimate 
\begin{equation*}
\int_{\R^n}|\nabla^\alpha f(x)|\,dx
\le
\mu_{n,\alpha}[f]_{W^{\alpha,1}(\R^n)}
\end{equation*} 
for all $f\in W^{\alpha,1}(\R^n)$ (see also~\cite{CS19}*{Theorem~3.18}), so that, by~\eqref{intro_eq:MS_0}, we can infer 
\begin{equation*}
\limsup_{\alpha\to0^+}
\alpha\int_{\R^n}|\nabla^\alpha f(x)|\,dx
\le
\mu_{n,0}
\lim_{\alpha\to0^+}\alpha\,[f]_{W^{\alpha,1}(\R^n)}
=
\mu_{n,0}B_{n,1}\|f\|_{L^1(\R^n)}
\end{equation*}
for all $f\in\bigcup_{\alpha\in(0,1)}W^{\alpha,1}(\R^n)$.
On the other side, if $f\in H^1(\R^n)$, then 
\begin{equation*}
\int_{\R^n}f(x)\,dx=0
\end{equation*}
(see~\cite{S93}*{Chapter~III, Section~5.4(c)} for example), and thus for all $f\in H^1(\R^n)\cap\bigcup_{\alpha\in(0,1)}W^{\alpha,1}(\R^n)$ the limit in~\eqref{intro_eq:mattia_limit} reduces to
\begin{equation*}
\lim_{\alpha\to0^+}
\alpha\int_{\R^n}|\nabla^\alpha f(x)|\,dx
=0,
\end{equation*}
in accordance with the strong convergence~\eqref{intro_eq:frac_limit_1}.

\subsection{Fractional interpolation inequalities}
\label{subsec:intro_frac_interp}

While~\eqref{intro_eq:mattia_limit} is proved by a direct computation, the limits~\eqref{intro_eq:frac_limit_1}, \eqref{intro_eq:frac_limit_p} and~\eqref{intro_eq:frac_limit_H1} follow from some new \emph{fractional interpolation inequalities}. 

Let $\alpha\in(0,1)$ be fixed.
In the standard fractional framework, by a simple splitting argument, it is not difficult to estimate the $W^{\beta,1}$-seminorm of a function $f\in W^{\alpha,1}(\R^n)$ as
\begin{equation}\label{intro_eq:interp_explain_R}
[f]_{W^{\beta,1}(\R^n)}
\le
R^{\alpha-\beta}\,[f]_{W^{\alpha,1}(\R^n)}
+c_n\frac{R^{-\beta}}\beta\,\|f\|_{L^1(\R^n)}
\end{equation}
for all $R>0$ and $\beta\in(0,\alpha)$,
where $c_n>0$ is a dimensional constant. 
If we choose 
$R=\|f\|_{L^1(\R^n)}^{1/\alpha}\,[f]_{W^{\alpha,1}(\R^n)}^{-1/\alpha}$, then~\eqref{intro_eq:interp_explain_R} gives
\begin{equation}\label{intro_eq:interp_explain}
[f]_{W^{\beta,1}(\R^n)}
\le
\left(1+\tfrac{c_n}\beta\right)\,\|f\|_{L^1(\R^n)}^{1-\frac\beta\alpha}\,[f]_{W^{\alpha,1}(\R^n)}^{\frac\beta\alpha}
\end{equation} 
for all $\beta\in(0,\alpha)$.
Inequality~\eqref{intro_eq:interp_explain} implies the bound 
\begin{equation}\label{intro_eq:big_O_W}
[f]_{W^{\beta,1}(\R^n)}
=
O\left(\tfrac1\beta\right)
\quad
\text{for}\ \beta\to0^+,
\end{equation}
in agreement with~\eqref{intro_eq:MS_0}. 

In a similar fashion (but with a more delicate analysis), an interpolation inequality of the form~\eqref{intro_eq:interp_explain} has been recently obtained by the third and the fourth author for the fractional gradient~$\nabla^\alpha$.
Precisely, if $f\in BV^\alpha(\R^n)$, then
\begin{equation}\label{intro_eq:frac_interp_explain}
[f]_{BV^\beta(\R^n)}
\le
c_{n,\alpha,\beta}\,
\|f\|_{L^1(\R^n)}^{1-\frac\beta\alpha}\,[f]_{BV^\alpha(\R^n)}^{\frac\beta\alpha}
\end{equation} 
for all $\beta\in(0,\alpha)$, where $c_{n,\alpha,\beta}>0$ is a constant such that
\begin{equation}\label{intro_eq:good_asymp_1}
c_{n,\alpha,\beta}\sim 1
\quad
\text{for}\ \beta\to\alpha^-
\end{equation}
and
\begin{equation}\label{intro_eq:big_O_constant}
c_{n,\alpha,\beta}=O\left(\tfrac1\beta\right)
\quad
\text{for}\ \beta\to0^+,
\end{equation}
see~\cite{CS19-2}*{Proposition~3.12} (see~\cite{CS19-2}*{Proposition~3.2} also for the case~$\alpha=1$).
Here and in the following, we let $[f]_{BV^\alpha(\R^n)}$ be the total fractional variation~\eqref{intro_eq:frac_variation} of $f\in BV^\alpha(\R^n)$.
Thanks to~\eqref{intro_eq:big_O_constant}, inequality~\eqref{intro_eq:frac_interp_explain} implies the bound
\begin{equation}\label{intro_eq:big_O_BV_beta}
[f]_{BV^\beta(\R^n)}=O\left(\tfrac1\beta\right)
\quad
\text{for}\ \beta\to0^+,
\end{equation} 
coherently with~\eqref{intro_eq:mattia_limit}.

Although strong enough to settle the asymptotic behavior of the fractional gradient~$\nabla^\beta$ when $\beta\to\alpha^-$ thanks to~\eqref{intro_eq:good_asymp_1}, because of~\eqref{intro_eq:big_O_BV_beta} inequality~\eqref{intro_eq:frac_interp_explain} is of no use for the study of the strong $L^1$-limit $\nabla^\beta\to R$ as~$\beta\to0^+$.
To achieve this convergence, we thus have to control the interpolation constant $c_{n,\alpha,\beta}$ in~\eqref{intro_eq:frac_interp_explain} with a new interpolation constant $c_{n,\alpha}>0$ independent of $\beta\in(0,\alpha)$, at the price of weakening~\eqref{intro_eq:frac_interp_explain} by replacing the $L^1$-norm with a bigger norm.

This strategy is in fact motivated by the non-optimality of the bound~\eqref{intro_eq:big_O_BV_beta} since, in view of the limit in~\eqref{intro_eq:mattia_limit}, we can still expect some cancellation effect of the fractional gradient for a subclass of $L^1$-functions having zero average. 
Note that this approach cannot be implemented to stabilize the standard interpolation inequality~\eqref{intro_eq:interp_explain}, since the bound in~\eqref{intro_eq:big_O_W} is in fact optimal due to~\eqref{intro_eq:MS_0}.

At this point, our idea is to exploit the cancellation properties of the fractional gradient~$\nabla^\beta$ by rewriting its non-local part in terms of a convolution kernel.
In more precise terms, recalling the definition in~\eqref{intro_eq:nabla_alpha}, for $R>0$ we can split 
\begin{equation}\label{intro_eq:split_frac_nabla}
\nabla^\beta f
=
\nabla^\beta_{<R} f+\nabla^\beta_{\ge R} f	
\end{equation}
with
\begin{equation}\label{intro_eq:nabla_non-local_operator}
\nabla^\beta_{\ge R} f(x)
=
\mu_{n,\beta}
\int_{\R^n}f(y)\,K_{\beta,R}(y-x)\,dy,
\quad
x\in\R^n,
\end{equation}
for all Schwartz functions $f\in\mathcal S(\R^n)$, where the convolution kernel $K_{\beta,R}$ is a smoothing of the function
\begin{equation*}
y
\mapsto
\frac{y}{|y|^{n+\beta+1}}\,\chi_{[R,+\infty)}(|y|).
\end{equation*}
By the Calder\'on--Zygmund Theorem, we can extend the functional defined in~\eqref{intro_eq:nabla_non-local_operator} to a linear continuous mapping $\nabla^\beta_{\ge R}\colon H^1(\R^n)\to L^1(\R^n;\R^n)$ whose operator norm can be estimated as
\begin{equation}\label{intro_eq:operator_norm_bound}
\|\nabla^\beta_{\ge R}\|_{H^1\to L^1}\le c_nR^{-\beta}
\quad
\text{for all}\ R>0,
\end{equation}
for some dimensional constant $c_n>0$.
By combining the splitting~\eqref{intro_eq:split_frac_nabla} with the bound~\eqref{intro_eq:operator_norm_bound} and arguing as in~\cite{CS19-2}, we get that
\begin{equation}\label{intro_eq:interp_inequ_H1_BV_alpha}
[f]_{BV^\beta(\R^n)}\le c_{n,\alpha}\,\|f\|_{H^1(\R^n)}^{\frac{\alpha-\beta}\alpha}\,[f]_{BV^\alpha(\R^n)}^{\frac\beta\alpha}
\end{equation}
for all $\beta\in[0,\alpha)$ and all $f\in H^1(\R^n)\cap BV^\alpha(\R^n)$, whenever $\alpha\in(0,1]$.
Exploiting~\eqref{intro_eq:interp_inequ_H1_BV_alpha} together with an approximation argument, we thus just need to establish~\eqref{intro_eq:frac_limit_1} for all sufficiently regular functions, in which case we can easily conclude by a direct computation.  

To achieve the limit in~\eqref{intro_eq:frac_limit_p} for $p\in(1,+\infty)$ and the stronger convergence in~\eqref{intro_eq:frac_limit_H1} for the case~$p=1$, we adopt a slightly different strategy.
Instead of splitting the fractional gradient as in~\eqref{intro_eq:split_frac_nabla}, we rewrite it as 
\begin{equation}
\label{intro_eq:frac_nabla_frac_Laplacian_decomposition}
\nabla^\beta=R\,(-\Delta)^{\frac\beta2},
\end{equation}
where
\begin{equation*}
(- \Delta)^{\frac{\beta}{2}} f(x) 
:= 
\nu_{n,\beta}
\int_{\R^n} \frac{f(x + y) - f(x)}{|y|^{n + \beta}}\,dy,
\quad
x\in\R^n,
\end{equation*}
is the usual fractional Laplacian with renormalizing constant given by
\begin{equation*}
\nu_{n,\beta}
:=
2^{\beta} \pi^{- \frac{n}{2}} \,\frac{\Gamma \left ( \frac{n + \beta}{2} \right )}{\Gamma \left ( - \frac{ \beta}{2} \right )}.
\end{equation*}
Since the Riesz transform extends to a linear continuous operator on $L^p(\R^n)$ and $H^1(\R^n)$ as mentioned above, to achieve~\eqref{intro_eq:frac_limit_p} and~\eqref{intro_eq:frac_limit_H1} we just have to study the continuity properties of $(-\Delta)^{\frac\beta2}$.
To this aim, we rewrite $(-\Delta)^{\frac\beta2}$ as 
\begin{equation}
\label{intro_eq:frac_Laplacian_decomposition}
(-\Delta)^{\frac\beta2}
=
T_{m_{\alpha,\beta}}
\circ
(\mathrm{Id}+(-\Delta)^{\frac\alpha2})
\end{equation} 
where
\begin{equation}\label{intro_eq:mattia_operator}
T_{m_{\alpha,\beta}}f
:=
f*\mathcal F^{-1}(m_{\alpha,\beta}),
\quad
f\in\mathcal S(\R^n),
\end{equation}
$\mathcal{F}$ is the Fourier transform and 
\begin{equation*}
m_{\alpha,\beta}(\xi):=\frac{|\xi|^\beta}{1+|\xi|^\alpha},
\quad
\xi\in\R^n.
\end{equation*}
Exploiting the good decay properties of the derivatives of $m_{\alpha,\beta}$ (uniform with respect to the parameters $0\le\beta\le\alpha\le1$), by the Mihlin--H\"ormander Multiplier Theorem the convolution operator in~\eqref{intro_eq:mattia_operator} can be extended to two linear operators continuous from $L^p(\R^n)$ to itself and from $H^1(\R^n)$ to itself, respectively.
Going back to~\eqref{intro_eq:frac_nabla_frac_Laplacian_decomposition} and~\eqref{intro_eq:frac_Laplacian_decomposition}, we can exploit the continuity properties of the (extensions of) the operator $T_{m_{\alpha,\beta}}$ to deduce two new interpolation inequalities.
On the one hand, given $p\in(1,+\infty)$, there exists a constant $c_{n,p}>0$ such that
\begin{equation}\label{intro_eq:interpolation_MH_p}
\|\nabla^\beta f\|_{L^p(\R^n;\,\R^n)}
\le
c_{n,p}
\,
\|\nabla^\gamma f\|_{L^p(\R^n;\,\R^n)}^{\frac{\alpha-\beta}{\alpha-\gamma}}
\,
\|\nabla^\alpha f\|_{L^p(\R^n;\,\R^n)}^{\frac{\beta-\gamma}{\alpha-\gamma}}
\end{equation}
for all $0\le\gamma\le\beta\le\alpha\le1$ and all $f\in S^{\alpha,p}(\R^n)$. 
In the particular case $\gamma=0$, thanks to the $L^p$-continuity of the Riesz transform, we also have
\begin{equation}\label{intro_eq:interpolation_MH_p_gamma=0}
\|\nabla^\beta f\|_{L^p(\R^n;\,\R^n)}
\le
c_{n,p}
\,
\|f\|_{L^p(\R^n)}^{\frac{\alpha-\beta}{\alpha}}
\,
\|\nabla^\alpha f\|_{L^p(\R^n;\,\R^n)}^{\frac{\beta}{\alpha}}
\end{equation}
for all $0\le\beta\le\alpha\le1$ and all $f\in S^{\alpha,p}(\R^n)$. 
On the other hand, there exists a dimensional constant $c_n>0$ such that 
\begin{equation}
\label{intro_eq:interpolation_MH_1}
\|\nabla^\beta f\|_{H^1(\R^n;\,\R^n)}
\le
c_n
\,
\|\nabla^\gamma f\|_{H^1(\R^n;\,\R^n)}^{\frac{\alpha-\beta}{\alpha-\gamma}}
\,
\|\nabla^\alpha f\|_{H^1(\R^n;\,\R^n)}^{\frac{\beta-\gamma}{\alpha-\gamma}}
\end{equation}
for all $0\le\gamma\le\beta\le\alpha\le1$ and all $f\in HS^{\alpha,1}(\R^n)$. 
Again, in the particular case $\gamma=0$, thanks to the $H^1$-continuity of the Riesz transform, we also have
\begin{equation}
\label{intro_eq:interpolation_MH_1_gamma=0}
\|\nabla^\beta f\|_{H^1(\R^n;\,\R^n)}
\le
c_n
\,
\|f\|_{H^1(\R^n)}^{\frac{\alpha-\beta}{\alpha}}
\,
\|\nabla^\alpha f\|_{H^1(\R^n;\,\R^n)}^{\frac{\beta}{\alpha}}
\end{equation}
for all $0\le\beta\le\alpha\le1$ and all $f\in HS^{\alpha,1}(\R^n)$.
Having the interpolation inequalities~\eqref{intro_eq:interpolation_MH_p_gamma=0} and~\eqref{intro_eq:interpolation_MH_1_gamma=0} at disposal, as before we just need to establish~\eqref{intro_eq:frac_limit_p} and~\eqref{intro_eq:frac_limit_H1} for all sufficiently regular functions, in which case we can again conclude by a direct computation.

As the reader may have noticed, in the above line of reasoning we can infer the validity of~\eqref{intro_eq:interpolation_MH_p} and~\eqref{intro_eq:interpolation_MH_1} only if we are able to prove the identifications
\begin{equation}\label{intro_eq:identification_p}
f\in S^{\alpha,p}(\R^n)
\iff
f\in(\mathrm{Id}-\Delta)^{-\frac\alpha2}(L^p(\R^n))
\iff
f\in L^p(\R^n)\cap I_\alpha(L^p(\R^n)),
\end{equation}
for $p\in(1,+\infty)$, and 
\begin{equation}\label{intro_eq:identification_H1}
f\in HS^{\alpha,1}(\R^n)
\iff
f\in(\mathrm{Id}-\Delta)^{-\frac\alpha2}(H^1(\R^n))
\iff
f\in H^1(\R^n)\cap I_\alpha(H^1(\R^n)),
\end{equation}
respectively, with equivalence of the naturally associated norms, where $(\mathrm{Id}-\Delta)^{-\frac\alpha2}$ is the standard Bessel potential.
While~\eqref{intro_eq:identification_H1} follows by a plain approximation argument building upon the results of~\cite{Str90}, the identification in~\eqref{intro_eq:identification_p} is more delicate and, actually, answers an equivalent question left open in~\cite{CS19}, that is, the density of $C^\infty_c(\R^n)$ functions in $S^{\alpha,p}(\R^n)$, see \cref{sec:identificaiton_Bessel} for the proof.
In other words, the equivalence~\eqref{intro_eq:identification_p} allows to identify the Bessel potential space 
\begin{equation*}
L^{\alpha,p}(\R^n):
=(\mathrm{Id}-\Delta)^{-\frac\alpha2}(L^p(\R^n))
=\set*{f\in\mathcal S'(\R^n) : (\mathrm{Id}-\Delta)^{\frac\alpha2}f\in L^p(\R^n)}
\end{equation*}
with the distributional fractional Sobolev space $S^{\alpha,p}(\R^n)$ in~\eqref{intro_eq:S_alpha_p}.
Thanks to the identification $L^{\alpha,p}(\R^n)=S^{\alpha,p}(\R^n)$, many of the results established in~\cites{BCM20,BCM21} and in~\cites{SS15,SS18} can be proved in a simpler and more direct way. 
See also \cref{sec:props_of_S_alpha_p} for other consequences of this identification.

\subsection{Complex interpolation and open problems}

To achieve the interpolation inequalities~\eqref{intro_eq:interp_inequ_H1_BV_alpha} and \eqref{intro_eq:interpolation_MH_p} -- \eqref{intro_eq:interpolation_MH_1_gamma=0}, we essentially relied on a direct approach exploiting the precise structure of the fractional gradient in~\eqref{intro_eq:nabla_alpha}.
Adopting the point of view of~\cites{M05,PS17}, a possible alternative route to the above fractional inequalities may follow from complex interpolation techniques. 

According to~\cite{BL76}*{Theorem~6.4.5(7)} and thanks to the aforementioned identification $L^{\alpha,p}(\R^n)=S^{\alpha,p}(\R^n)$, for all $\alpha,\theta\in(0,1)$ and $p\in(1,+\infty)$  we have the complex interpolation
\begin{equation}
\label{intro_eq:interpolation_Bessel}
(L^p(\R^n),S^{\alpha,p}(\R^n))_{ [\theta]}
\cong
S^{\theta\alpha,p}(\R^n).
\end{equation}
Here and in the following, we write $A\cong B$ to  emphasize the fact that the spaces~$A$ and~$B$ are the same with equivalence (and thus, possibly, not equality) of the relative norms.
As a consequence, \eqref{intro_eq:interpolation_Bessel} implies that, for all $0<\beta<\alpha<1$ and $p\in(1,+\infty)$, there exists a constant $c_{n,\alpha,\beta,p}>0$ such that
\begin{equation}
\label{intro_eq:interpolation_Bessel_ineq}
\|f\|_{S^{\beta,p}(\R^n)}
\le
c_{n,\alpha,\beta,p}
\,
\|f\|_{L^p(\R^n)}^{\frac{\alpha-\beta}{\alpha}}
\,
\|f\|_{S^{\alpha,p}(\R^n)}^{\frac{\beta}{\alpha}}
\end{equation}
for all $f\in S^{\alpha,p}(\R^n)$. 
In a similar way (we omit the proof because beyond the scopes of the present paper), for all $\alpha,\theta\in(0,1)$ one can also establish the complex interpolation
\begin{equation}
\label{intro_eq:interpolation_HS}
(H^1(\R^n),HS^{\alpha,1}(\R^n))_{ [\theta]}\cong HS^{\theta\alpha,1}(\R^n),
\end{equation}
and thus, for some constant $c_{n,\alpha,\beta}>0$,
\begin{equation}\label{intro_eq:interpolation_HS_ineq}
\|f\|_{HS^{\beta,1}(\R^n)}
\le
c_{n,\alpha,\beta}
\,
\|f\|_{H^1(\R^n)}^{\frac{\alpha-\beta}{\alpha}}
\,
\|f\|_{HS^{\alpha,1}(\R^n)}^{\frac{\beta}{\alpha}}
\end{equation}
for all $f\in HS^{\alpha,1}(\R^n)$. 

Inequalities~\eqref{intro_eq:interpolation_Bessel_ineq} and \eqref{intro_eq:interpolation_HS_ineq} suggest that, in order to obtain \eqref{intro_eq:interpolation_MH_p_gamma=0} and \eqref{intro_eq:interpolation_MH_1_gamma=0} with complex interpolation methods, one essentially should prove that the identifications~\eqref{intro_eq:interpolation_Bessel} and~\eqref{intro_eq:interpolation_HS} hold uniformly with respect to the interpolating parameter. We believe that this result may be achieved but, since we do not need this level of generality for our aims, we preferred to prove~\eqref{intro_eq:interpolation_MH_p} -- \eqref{intro_eq:interpolation_MH_1_gamma=0} in a more direct and explicit way.

We do not know if also inequality~\eqref{intro_eq:interp_inequ_H1_BV_alpha} can be achieved by complex interpolation methods. 
In fact, we do not even know if the spaces $(H^1(\R^n),BV(\R^n))_{ [\theta]}$ and $BV^\theta(\R^n)$ are somehow linked for $\theta\in(0,1)$ (for a related discussion, see also~\cite{SS15}*{Section~1.1}). 
By~\cite{BL76}*{Theorems~3.5.3 and~6.4.5(1)}, we have the real interpolations
\begin{equation*}
(L^1(\R^n),W^{1,1}(\R^n))_{\theta,p}
\cong
(L^1(\R^n),BV(\R^n))_{\theta,p}
\cong 
B^\theta_{1,p}(\R^n)
\end{equation*} 
for all $\theta\in(0,1)$ and $p\in[1,+\infty]$, where $B^\theta_{p,q}(\R^n)$ denotes the Besov space as usual (see~\cite{BL76}*{Section~6.2} or~\cite{L09}*{Chapter~14} for the definition). By~\cite{BL76}*{Theorem~4.7.1}, we know that
\begin{equation*}
(H^1(\R^n),BV(\R^n))_{\theta,1}
\subset
(H^1(\R^n),BV(\R^n))_{ [\theta]}
\subset
(H^1(\R^n),BV(\R^n))_{\theta,\infty}
\end{equation*}  
for all $\theta\in(0,1)$. Since $H^1(\R^n)\subset L^1(\R^n)$ continuously, on the one side we have
\begin{equation*}
(H^1(\R^n),BV(\R^n))_{\theta,1}
\subset
(L^1(\R^n),BV(\R^n))_{\theta,1}
\cong
B^\theta_{1,1}(\R^n)
\cong
W^{\theta,1}(\R^n)
\end{equation*}
and, on the other side, 
\begin{equation*}
(H^1(\R^n),BV(\R^n))_{\theta,\infty}
\subset
(L^1(\R^n),BV(\R^n))_{\theta,\infty}
\cong
B^\theta_{1,\infty}(\R^n),
\end{equation*} 
for all $\theta\in(0,1)$. 
On the one hand, the continuous inclusion $W^{\alpha,1}(\R^n)\subset BV^\alpha(\R^n)$ is strict for all $\alpha\in(0,1)$ by~\cite{CS19}*{Theorem~3.31}. 
On the other hand, the inclusion $BV^\alpha(\R^n)\subset B^\alpha_{1,\infty}(\R^n)$ holds continuously for all $\alpha\in(0,1)$ as a consequence of~\cite{CS19}*{Proposition~3.14}, but it also holds strictly when $n\ge2$, see \cref{thm:strict_inclusion}. 

\subsection{Organization of the paper}

We conclude this introduction by briefly presenting the organization of the present paper.
\cref{sec:preliminaires} provides the main notation, recalls the needed properties of the fractional operators~$\nabla^\alpha$ and~$\div^\alpha$ and, finally, deals with the properties of the space $HS^{\alpha,1}(\R^n)$.
\cref{sec:BV_0} is devoted to the proof of the identification $BV^0(\R^n)=H^1(\R^n)$, together with some useful consequences about the relation between $H^1(\R^n)$ and $W^{\alpha,1}(\R^n)$.
In Sections~\ref{sec:interpolation_inequalities} and~\ref{sec:asymptotic_to_zero}, the core of our work, we detail the proof of the interpolation inequalities~\eqref{intro_eq:interp_inequ_H1_BV_alpha}, \eqref{intro_eq:interpolation_MH_p} and~\eqref{intro_eq:interpolation_MH_1} and, consequently, we prove both the strong convergence of the fractional gradient~$\nabla^\alpha$ as $\alpha\to0^+$ given by~\eqref{intro_eq:frac_limit_p}, \eqref{intro_eq:frac_limit_H1} and the limit~\eqref{intro_eq:mattia_limit}.
We close our work with three appendices: 
in \cref{sec:identificaiton_Bessel} we prove the density of $C^\infty_c(\R^n)$ functions in $S^{\alpha,p}(\R^n)$; 
in \cref{sec:props_of_S_alpha_p} we state some properties of $S^{\alpha,p}$-functions; 
in \cref{sec:continuity_props_nabla_alpha} we establish some continuity properties of the map $\alpha\mapsto\nabla^\alpha$.

\section{Preliminaries}
\label{sec:preliminaires}

We start with a brief description of the main notation used in this paper. In order to keep the exposition as reader-friendly as possible, we retain the same notation adopted in the previous works~\cites{CS19,CS19-2}.

\subsection{General notation}
\label{subsec:notation}
We let $\Leb{n}$ and $\Haus{\alpha}$ be the $n$-dimensional Lebesgue measure and the $\alpha$-dimensional Hausdorff measure on $\R^n$ respectively, with $\alpha \ge 0$. A measurable set is a $\Leb{n}$-measurable set. We also use the notation $|E|=\Leb{n}(E)$. All functions we consider in this paper are Lebesgue measurable. We let $B_r(x)$ be the standard open Euclidean ball with center $x\in\R^n$ and radius $r>0$. We set $B_r=B_r(0)$. Recall that $\omega_{n} := |B_1|=\pi^{\frac{n}{2}}/\Gamma\left(\frac{n+2}{2}\right)$ and $\Haus{n-1}(\partial B_{1}) = n \omega_n$, where $\Gamma$ is the Euler's \emph{Gamma function}, see~\cite{A64}. 

For $m\in\N$, the total variation on~$\Omega$ of the $m$-vector-valued Radon measure $\mu$ is defined as
\begin{equation*}
	|\mu|(\Omega):=\sup\set*{\int_\Omega\phi\cdot d\mu : \phi\in C^\infty_c(\Omega;\R^m),\ \|\phi\|_{L^\infty(\Omega;\,\R^m)}\le1}.
\end{equation*}
We thus let $\mathscr{M}(\Omega;\R^m)$ be the space of $m$-vector-valued Radon measure  with finite total variation on $\Omega$.

For $k\in\N_{0}\cup\set{+\infty}$ and $m \in \N$, we let $C^{k}_{c}(\Omega;\R^{m})$ and $\Lip_c(\Omega;\R^{m})$ be the spaces of $C^{k}$-regular and, respectively, Lipschitz-regular, $m$-vector-valued functions defined on~$\R^n$ with compact support in the open set~$\Omega\subset\R^n$.

For $m\in\N$, we let $\mathcal{S}(\R^n;\R^m)$ be the space of $m$-vector-valued Schwartz functions on~$\R^n$. 
For $k\in\N_{0}\cup\set{+\infty}$ and $m\in\N$, let
\begin{equation*}
\mathcal{S}_k(\R^n;\R^m):=\set*{f\in\mathcal{S}(\R^n;\R^m) : \int_{\R^n} x^\mathsf{a}f(x)\,dx=0\ \text{for all}\ \mathsf{a}\in\N^n_0\ \text{with}\ |\mathsf{a}|\le k},	
\end{equation*}
where $x^\mathsf{a}:=x_1^{\mathsf{a}_1}\cdot\ldots\cdot x_n^{\mathsf{a}_n}$ for all multi-indices $\mathsf{a}\in\N^n_0$. 
We let $\mathcal{S}'(\R^n;\R^m)$ be the dual of $\mathcal{S}(\R^n;\R^m)$ and we call it the space of tempered distributions. 
See~\cite{G14-C}*{Section~2.2 and 2.3} for instance.

For any exponent $p\in[1,+\infty]$, we let $L^p(\Omega;\R^m)$ be the space of $m$-vector-valued Lebesgue $p$-integrable functions on~$\Omega$. 

We let 
\begin{equation*}
\mathcal{F}(f)(\xi) := \int_{\R^n} f(x) \, e^{- i \xi \cdot x} \, dx, \quad \xi \in \R^n,
\end{equation*}
be the Fourier transform of the function $f \in L^1(\R^n;\R^m)$.
As it is well known, the Fourier transform maps $\mathcal{S}(\R^n;\R^m)$ onto itself and may be extended to $\mathcal{S}'(\R^n;\R^m)$  (see \cite{G14-C}*{Sections~2.2 and 2.3} for instance).

We let
\begin{equation*}
W^{1,p}(\Omega;\R^m):=\set*{u\in L^p(\Omega;\R^m) : [u]_{W^{1,p}(\Omega;\R^m)}:=\|\nabla u\|_{L^p(\Omega;\R^{n m})}<+\infty}
\end{equation*}
be the space of $m$-vector-valued Sobolev functions on~$\Omega$, see for instance~\cite{L09}*{Chapter~10} for its precise definition and main properties. 
We let
\begin{equation*}
BV(\Omega;\R^m):=\set*{u\in L^1(\Omega;\R^m) : [u]_{BV(\Omega;\R^m)}:=|Du|(\Omega)<+\infty}
\end{equation*}
be the space of $m$-vector-valued functions of bounded variation on~$\Omega$, see for instance~\cite{AFP00}*{Chapter~3} or~\cite{EG15}*{Chapter~5} for its precise definition and main properties. 

For $\alpha\in(0,1)$ and $p\in[1,+\infty)$, we let 
\begin{equation*}
W^{\alpha,p}(\Omega;\R^m):=\set*{u\in L^p(\Omega;\R^m) : [u]_{W^{\alpha,p}(\Omega;\R^m)}\!:=\left(\int_\Omega\int_\Omega\frac{|u(x)-u(y)|^p}{|x-y|^{n+p\alpha}}\,dx\,dy\right)^{\frac{1}{p}}\!<+\infty}
\end{equation*}
be the space of $m$-vector-valued fractional Sobolev functions on~$\Omega$, see~\cite{DiNPV12} for its precise definition and main properties. 
For $\alpha\in(0,1)$ and $p=+\infty$, we simply let
\begin{equation*}
W^{\alpha,\infty}(\Omega;\R^m):=\set*{u\in L^\infty(\Omega;\R^m) : \sup_{x,y\in \Omega,\, x\neq y}\frac{|u(x)-u(y)|}{|x-y|^\alpha}<+\infty},
\end{equation*}
so that $W^{\alpha,\infty}(\Omega;\R^m)=C^{0,\alpha}_b(\Omega;\R^m)$, the space of $m$-vector-valued bounded $\alpha$-H\"older continuous functions on~$\Omega$.

Given $\alpha\in(0,n)$, let
\begin{equation}\label{eq:Riesz_potential_def} 
I_{\alpha} f(x) := 
2^{-\alpha} \pi^{- \frac{n}{2}} \frac{\Gamma\left(\frac{n-\alpha}2\right)}{\Gamma\left(\frac\alpha2\right)}
\int_{\R^{n}} \frac{f(y)}{|x - y|^{n - \alpha}} \, dy, 
\quad
x\in\R^n,
\end{equation}
be the Riesz potential of order $\alpha\in(0,n)$ of $f\in C^\infty_c(\R^n;\R^m)$. We recall that, if $\alpha,\beta\in(0,n)$ satisfy $\alpha+\beta<n$, then we have the following \emph{semigroup property}
\begin{equation}\label{eq:Riesz_potential_semigroup}
I_{\alpha}(I_\beta f)=I_{\alpha+\beta}f
\end{equation}
for all $f\in C^\infty_c(\R^n;\R^m)$. In addition, if $1<p<q<+\infty$ satisfy 
\begin{equation*}
\frac{1}{q}=\frac{1}{p}-\frac{\alpha}{n},	
\end{equation*}
then there exists a constant $C_{n,\alpha,p}>0$ such that the operator in~\eqref{eq:Riesz_potential_def} satisfies
\begin{equation}\label{eq:Riesz_potential_boundedness}
\|I_\alpha f\|_{L^q(\R^n;\,\R^m)}\le C_{n,\alpha,p}\|f\|_{L^p(\R^n;\,\R^m)}
\end{equation}
for all $f\in C^\infty_c(\R^n;\,\R^m)$. As a consequence, the operator in~\eqref{eq:Riesz_potential_def} extends to a linear continuous operator from $L^p(\R^n;\R^m)$ to $L^q(\R^n;\R^m)$, for which we retain the same notation. For a proof of~\eqref{eq:Riesz_potential_semigroup} and~\eqref{eq:Riesz_potential_boundedness}, we refer the reader to~\cite{S70}*{Chapter~V, Section~1} and to~\cite{G14-M}*{Section~1.2.1}.

Given $\alpha\in(0,1)$, we also let
\begin{equation}\label{eq:def_frac_Laplacian}
(- \Delta)^{\frac{\alpha}{2}} f(x) := \nu_{n, \alpha}
\int_{\R^n} \frac{f(x + y) - f(x)}{|y|^{n + \alpha}}\,dy,
\quad
x\in\R^n,
\end{equation}
be the fractional Laplacian (of order~$\alpha$) of $f \in\Lip_b(\R^{n};\R^m)$, where
\begin{equation*}
\nu_{n,\alpha}=2^\alpha\pi^{-\frac n2}\frac{\Gamma\left(\frac{n+\alpha}{2}\right)}{\Gamma\left(-\frac{\alpha}{2}\right)},
\quad
\alpha\in(0,1).
\end{equation*}

For $\alpha\in(0,1)$ and $p\in(1,+\infty)$, let
\begin{equation}\label{eq:def1_Bessel_space}
\begin{split}
L^{\alpha,p}(\R^n;\R^m)
&:=(\mathrm{Id}-\Delta)^{-\frac\alpha2}(L^p(\R^n;\R^m))\\
&:=\set*{f\in\mathcal S'(\R^n;\R^m) : (\mathrm{Id}-\Delta)^{\frac\alpha2}f\in L^p(\R^n;\R^m)}
\end{split}
\end{equation}
be the $m$-vector-valued Bessel potential space with norm
\begin{equation}\label{eq:def_Bessel_norm1}
\|f\|_{L^{\alpha,p}(\R^n;\,\R^m)}=
\|(\mathrm{Id}-\Delta)^{\frac\alpha2}f\|_{L^p(\R^n;\,\R^m)},
\quad
f\in L^{\alpha,p}(\R^n;\R^m),
\end{equation}
see~\cite{A75}*{Sections 7.59-7.65} for its precise definition and main properties. We also refer to~\cite{SKM93}*{Section~27.3}, where the authors prove that the space in~\eqref{eq:def1_Bessel_space} can be equivalently defined as the space
\begin{equation}\label{eq:def2_Bessel_space}
L^p(\R^n;\R^m)\cap I_\alpha(L^p(\R^n;\R^m))
=\set*{f\in L^p(\R^n;\R^m) : (-\Delta)^{\frac\alpha2}f\in L^p(\R^n;\R^m)},
\end{equation}
see~\cite{SKM93}*{Theorem~27.3}. 
In particular, the function
\begin{equation}\label{eq:def_Bessel_norm2}
f\mapsto
\|f\|_{L^p(\R^n;\,\R^m)}
+
\|(-\Delta)^{\frac\alpha2}f\|_{L^p(\R^n;\,\R^m)},
\quad
f\in L^{\alpha,p}(\R^n;\,\R^m),
\end{equation}
defines a norm on $L^{\alpha,p}(\R^n;\,\R^m)$ equivalent to the one in~\eqref{eq:def_Bessel_norm1} (and so, unless otherwise stated, we will use both norms~\eqref{eq:def_Bessel_norm1} and~\eqref{eq:def_Bessel_norm2} with no particular distinction). 
We recall that $C^\infty_c(\R^n)$ is a dense subset of $L^{\alpha,p}(\R^n;\R^m)$, see~\cite{A75}*{Theorem~7.63(a)} and~\cite{SKM93}*{Lemma~27.2}. Note that the space $L^{\alpha,p}(\R^n;\R^m)$ can be defined also for any $\alpha\ge1$ by simply using the composition properties of the Bessel potential (or of the fractional Laplacian), see~\cite{A75}*{Section~7.62}. All the properties stated above remain true also for $\alpha\ge1$ and, moreover, $L^{k,p}(\R^n;\R^m)=W^{k,p}(\R^n;\R^m)$ for all $k\in\N$, see~\cite{A75}*{Theorem~7.63(f)}. 

For $m\in\N$, we let
\begin{equation*}
H^1(\R^n;\R^m):=\set*{f\in L^1(\R^n;\R^m) : Rf\in L^1(\R^n;\R^{mn})}
\end{equation*}
be the $m$-vector-valued (real) Hardy space endowed with the norm
\begin{equation*}
\|f\|_{H^1(\R^n;\,\R^m)}:=\|f\|_{L^1(\R^n;\,\R^m)}+\|Rf\|_{L^1(\R^n;\,\R^{mn})}
\end{equation*}
for all $f\in H^1(\R^n;\R^m)$, where $Rf$ denotes the Riesz transform of~$f\in H^1(\R^n;\R^m)$, componentwise defined by
\begin{equation}\label{eq:def_Riesz_transform}
R f_{i}(x)
:=\pi^{-\frac{n+1}2}\,\Gamma\left(\tfrac{n+1}{2}\right)\,\lim_{\eps\to0^+}\int_{\set*{|y|>\eps}}\frac{y\,f_i(x+y)}{|y|^{n+1}}\,dy,
\quad
x\in\R^n,\
i=1,\dots,m.
\end{equation}
We refer the reader to~\cite{G14-M}*{Sections~2.1 and~2.4.4}, \cite{S70}*{Chapter III, Section 1} and~\cite{S93}*{Chapter~III} for a more detailed exposition. 
We warn the reader that the definition in~\eqref{eq:def_Riesz_transform} agrees with the one in~\cites{S93} and differs from the one in~\cites{G14-M,S70} for a minus sign.
We also recall that the Riesz transform~\eqref{eq:def_Riesz_transform} defines a continuous operator $R\colon L^p(\R^n;\R^m)\to L^p(\R^n;\R^{mn})$ for any given $p\in(1,+\infty)$, see~\cite{G14-C}*{Corollary~5.2.8}, and a continuous operator $R\colon H^1(\R^n;\R^m)\to H^1(\R^n;\R^{mn})$, see~\cite{S93}*{Chapter~III, Section~5.25}.

In the sequel, in order to avoid heavy notation, if the elements of a function space $F(\Omega;\R^m)$ are real-valued (i.e.~$m=1$), then we will drop the target space and simply write~$F(\Omega)$. 

\subsection{Overview of \texorpdfstring{$\nabla^\alpha$}{nablaˆalpha} and \texorpdfstring{$\div^\alpha$}{divˆalpha} and the related function spaces} \label{sect:overview_fract_grad}

We recall the definition (and the main properties) of the non-local operators~$\nabla^\alpha$ and~$\diverg^\alpha$, see~\cites{S19,CS19,CS19-2} and the monograph~\cite{P16}*{Section~15.2}.

Let $\alpha\in(0,1)$ and set 
\begin{equation*}
\mu_{n, \alpha} := 2^{\alpha}\, \pi^{- \frac{n}{2}}\, \frac{\Gamma\left ( \frac{n + \alpha + 1}{2} \right )}{\Gamma\left ( \frac{1 - \alpha}{2} \right )}.
\end{equation*}
We let
\begin{equation*}
\nabla^{\alpha} f(x) := \mu_{n, \alpha} \lim_{\eps \to 0^+} \int_{\{ |y| > \eps \}} \frac{y \, f(x + y)}{|y|^{n + \alpha + 1}} \, dy
\end{equation*}
be the \emph{fractional $\alpha$-gradient} of $f\in\Lip_c(\R^n)$ at $x\in\R^n$. We also let
\begin{equation*}
\div^{\alpha} \varphi(x) := \mu_{n, \alpha} \lim_{\eps \to 0^+} \int_{\{ |y| > \eps \}} \frac{y \cdot \varphi(x + y)}{|y|^{n + \alpha + 1}} \, dy
\end{equation*}
be the \emph{fractional $\alpha$-divergence} of $\phi\in\Lip_c(\R^n;\R^n)$ at $x\in\R^n$. 
The non-local operators~$\nabla^\alpha$ and~$\diverg^\alpha$ are well defined in the sense that the involved integrals converge and the limits exist. Moreover, since 
\begin{equation*}
\int_{\set*{|z| > \eps}} \frac{z}{|z|^{n + \alpha + 1}} \, dz=0,
\quad
\forall\eps>0,
\end{equation*}
it is immediate to check that $\nabla^{\alpha}c=0$ for all $c\in\R$ and
\begin{align*}
\nabla^{\alpha} f(x)
&=\mu_{n, \alpha} \int_{\R^{n}} \frac{(y - x)  (f(y) - f(x)) }{|y - x|^{n + \alpha + 1}} \, dy,
\quad
\forall x\in\R^n,
\end{align*}
for all $f\in\Lip_c(\R^n)$. Analogously, we also have
\begin{align*}
\div^{\alpha} \varphi(x) 
&= \mu_{n, \alpha} \int_{\R^{n}} \frac{(y - x) \cdot (\varphi(y) - \varphi(x)) }{|y - x|^{n + \alpha + 1}} \, dy,
\quad
\forall x\in\R^n,
\end{align*}
for all $\phi\in\Lip_c(\R^n)$.

Thanks to~\cite{CS19}*{Proposition~2.2}, given $\alpha\in(0,1)$ we can equivalently write 
\begin{equation}\label{eq:def_frac_operators_Riesz_potential}
\nabla^\alpha f =\nabla I_{1-\alpha}f=I_{1-\alpha}\nabla f
\quad
\text{and}
\quad
\div^\alpha \phi =\div I_{1-\alpha}\phi=I_{1-\alpha}\div \phi
\end{equation} 
for all $f\in\Lip_c(\R^n;\R^n)$ and $\phi\in\Lip_c(\R^n;\R^n)$, respectively. 

The fractional operators $\nabla^\alpha$ and $\div^\alpha$ are \emph{dual} in the sense that
\begin{equation}\label{eq:duality}
\int_{\R^n}f\,\div^\alpha\phi \,dx=-\int_{\R^n}\phi\cdot\nabla^\alpha f\,dx
\end{equation}
for all $f\in\Lip_c(\R^n)$ and $\phi\in\Lip_c(\R^n;\R^n)$, see~\cite{Sil19}*{Section~6} and~\cite{CS19}*{Lemma~2.5}. In addition, given $f\in\Lip_c(\R^n)$ and $\phi\in\Lip_c(\R^n;\R^n)$, we have 
\begin{equation}\label{eq:integrability}
\nabla^\alpha f\in L^p(\R^n)
\quad\text{and}\quad
\div^\alpha\phi\in L^p(\R^n;\R^n)	
\end{equation}
for all $p\in[1,+\infty]$, see~\cite{CS19}*{Corollary~2.3}. 
The above results and identities hold also for functions $f \in \mathcal{S}(\R^n)$ and $\varphi \in \mathcal{S}(\R^n; \R^n)$.

Given $\alpha\in(0,1)$ and $p\in[1,+\infty]$, inspired by the integration-by-parts formula~\eqref{eq:duality},  we say that a function $f\in L^p(\R^n)$ has bounded \emph{fractional $\alpha$-variation} if
\begin{equation}\label{eq:def_fractional_variation}
|D^\alpha f|(\R^n):=\sup\set*{\int_{\R^n}f\,\div^\alpha\phi\,dx : \phi\in C^\infty_c(\R^n;\R^n),\ \|\phi\|_{L^\infty(\R^n;\,\R^n)}\le1}<+\infty,
\end{equation}
see~\cite{CS19}*{Section~3} for the case $p=1$ and the discussion in~\cite{CS19-2}*{Section~3.3} for the case $p\in(1,+\infty]$. Note that the above notion of fractional $\alpha$-variation is well posed thanks to the integrability property~\eqref{eq:integrability}. Following the strategy outlined in~\cite{CS19}*{Section~3.2}, the reader can verify that the linear space 
\begin{equation*}
BV^{\alpha,p}(\R^n):=
\set*{f\in L^p(\R^n) : |D^\alpha f|(\R^n)<+\infty}
\end{equation*}
endowed with the norm
\begin{equation*}
\|f\|_{BV^{\alpha,p}(\R^n)}:=\|f\|_{L^p(\R^n)}+|D^\alpha f|(\R^n),
\quad
f\in BV^{\alpha,p}(\R^n),
\end{equation*}
is a Banach space and that the fractional variation defined in~\eqref{eq:def_fractional_variation} is lower semicontinuous with respect to $L^p$-convergence. 
In the sequel, we also use the notation $[f]_{BV^{\alpha,p}(\R^n)}=|D^\alpha f|(\R^n)$ for a given $f\in BV^{\alpha,p}(\R^n)$. 

In the case $p=1$, we simply write $BV^{\alpha,1}(\R^n)=BV^\alpha(\R^n)$. The space $BV^\alpha(\R^n)$ resembles the classical space $BV(\R^n)$ from many points of view and we refer the reader to~\cite{CS19}*{Section~3} for a detailed exposition of its main properties. 

Again motivated by~\eqref{eq:duality} and  in analogy with the classical case, given $\alpha\in(0,1)$ and $p\in[1,+\infty]$ we define the \emph{weak fractional $\alpha$-gradient} of a function $f\in L^p(\R^n)$ as the function $\nabla^\alpha f\in L^1_{\loc}(\R^n;\R^n)$ satisfying
\begin{equation*}
\int_{\R^n}f\,\div^\alpha\phi\, dx
=-\int_{\R^n}\nabla^\alpha f\cdot\phi\, dx
\end{equation*} 
for all $\phi\in C^\infty_c(\R^n;\R^n)$. 
We notice that, in the case $f \in \Lip_c(\R^n)$ (or $f \in \mathcal{S}(\R^n)$), the weak fractional $\alpha$-gradient of $f$ coincides with the one defined above, thanks to \eqref{eq:duality}.
As above, the reader can verify that the \emph{distributional fractional Sobolev space} 
\begin{equation}\label{eq:def_distrib_frac_Sobolev}
S^{\alpha,p}(\R^n):=\set*{f\in L^p(\R^n) : \exists\, \nabla^\alpha f \in L^p(\R^n;\R^n)}
\end{equation}
endowed with the norm
\begin{equation}\label{eq:def_distrib_frac_Sobolev_norm}
\|f\|_{S^{\alpha,p}(\R^n)}:=\|f\|_{L^p(\R^n)}+\|\nabla^\alpha f\|_{L^p(\R^n;\,\R^{n})}
\quad
f\in S^{\alpha,p}(\R^n),
\end{equation}
is a Banach space. 

In the case $p=1$, starting from the very definition of the fractional gradient~$\nabla^\alpha$, one can check that $W^{\alpha,1}(\R^n)\subset S^{\alpha,1}(\R^n)\subset BV^\alpha(\R^n)$ with both strict continuous embeddings, see~\cite{CS19}*{Theorems 3.18, 3.25, 3.26, 3.30 and~3.31}, and that $C^\infty_c(\R^n)$ is a dense subset of $S^{\alpha,1}(\R^n)$, see~\cite{CS19}*{Theorem~3.23}.

In the case $p\in(1,+\infty)$, the density of the set of test functions in the space $S^{\alpha,p}(\R^n)$ was left as an open problem in~\cite{CS19}*{Section~3.9}. More precisely, defining
\begin{equation*}
S^{\alpha,p}_0(\R^n):=\closure[-1]{C^\infty_c(\R^n)}^{\,\|\cdot\|_{S^{\alpha,p}(\R^n)}}
\end{equation*}
endowed with the norm in~\eqref{eq:def_distrib_frac_Sobolev_norm}, it is immediate to see that $S^{\alpha,p}_0(\R^n)\subset S^{\alpha,p}(\R^n)$ with continuous embedding. The space $(S^{\alpha,p}_0(\R^n),\|\cdot\|_{S^{\alpha,p}(\R^n)})$ was introduced in~\cite{SS15} (with a different, but equivalent, norm) and, in fact, it satisfies
\begin{equation*}
S^{\alpha,p}_0(\R^n) = L^{\alpha, p}(\R^{n})
\end{equation*}
for all $\alpha \in (0, 1)$ and $p \in (1, + \infty)$, see~\cite{SS15}*{Theorem 1.7}. In~\cref{res:L_p_alpha_equal_S_p_alpha} in the appendix, we positively solve the problem of the density of $C^\infty_c(\R^n)$ in the space $S^{\alpha,p}(\R^n)$. As a consequence, we obtain the following result.

\begin{corollary}[Identification $S^{\alpha,p}=L^{\alpha,p}$]
\label{res:S=L}
Let $\alpha \in (0, 1)$ and $p \in (1, + \infty)$. We have $S^{\alpha,p}(\R^n) = L^{\alpha,p}(\R^n)$.	
\end{corollary}

According to \cref{res:S=L}, in the sequel we will also use the symbol $S^{\alpha,p}$ to denote the Bessel potential space  $L^{\alpha,p}$. In addition, consistently with the asymptotic behavior of the fractional gradient~$\nabla^\alpha$ as $\alpha\to1^-$ established in~\cite{CS19-2}, we will sometimes denote the Sobolev space~$W^{1,p}$ as~$S^{1,p}$ for~$p\in[1,+\infty)$.

Thanks to the identification given by \cref{res:S=L}, we can prove the following result.

\begin{proposition}[$\mathcal S_0$ is dense in $S^{\alpha,p}$]
\label{res:S_0_dense_in_S_p_alpha}
Let $\alpha\in(0,1)$ and $p\in(1,+\infty)$. 
The set $\mathcal S_0(\R^n)$ is dense in $S^{\alpha,p}(\R^n)$. 
\end{proposition}

\begin{proof}
By \cref{res:S=L}, we equivalently have to prove that the set $\mathcal S_0(\R^n)$ is dense in $L^{\alpha,p}(\R^n)$. To this aim, let us consider the functional $M\colon (\mathcal S(\R^n),\|\cdot\|_{L^p(\R^n)})\to\R$ defined as
\begin{equation*}
M(f)=\int_{\R^n} f(x)\,dx,
\quad
f\in\mathcal S(\R^n).
\end{equation*} 
Clearly, the linear functional~$M$ cannot be continuous and thus its kernel $\mathcal S_0(\R^n)$ must be dense in $\mathcal S(\R^n)$ with respect to the $L^p$-norm. Since the Bessel potential 
\begin{equation*}
(\mathrm{Id}-\Delta)^{-\frac\alpha2}\colon (\mathcal S(\R^n),\|\cdot\|_{S^{\alpha,p}(\R^n)})\to (\mathcal S(\R^n),\|\cdot\|_{L^p(\R^n)})	
\end{equation*}
is an isomorphism, the conclusion follows. 
\end{proof}

\subsection{The fractional Hardy--Sobolev space \texorpdfstring{$HS^{\alpha,1}(\R^n)$}{HSˆ{alpha,1}(Rˆn)}}
\label{subsec:HS_frac_space}

Following the classical approach of~\cite{Str90}, for $\alpha\in[0,1]$ let
\begin{align*}
HS^{\alpha,1}(\R^n):
&=(I-\Delta)^{-\frac\alpha2}(H^1(\R^n))\\
&=\set*{f\in H^1(\R^n) : (I-\Delta)^{\frac\alpha2}f\in H^1(\R^n)}
\end{align*}
be the (real) \emph{fractional Hardy--Sobolev space} endowed with the norm
\begin{equation}\label{eq:def_HS_norm}
\|f\|_{HS^{\alpha,1}(\R^n)}=\|(I-\Delta)^{\frac\alpha2}f\|_{H^1(\R^n)},
\quad
f\in H^{1,\alpha}(\R^n).
\end{equation}
In particular, $HS^{0,1}(\R^n)=H^1(\R^n)$ coincides with the (real) Hardy space and $H^{1,1}(\R^n)$ is the standard (real) Hardy--Sobolev space. As remarked in~\cite{Str90}*{p.~130}, $HS^{\alpha,1}(\R^n)$ can be equivalently defined as
\begin{align*}
H^1(\R^n)\cap I_\alpha(H^1(\R^n))
=\set*{f\in H^1(\R^n) : (-\Delta)^{\frac\alpha2} f\in H^1(\R^n)}.
\end{align*}
In particular, the function
\begin{equation}\label{eq:def_HS_norm_bis}
f\mapsto
\|f\|_{H^1(\R^n)}
+\|(-\Delta)^{\frac\alpha2} f\|_{H^1(\R^n)},
\quad 
f\in HS^{\alpha,1}(\R^n),
\end{equation}
defines a norm on $HS^{\alpha,1}(\R^n)$ equivalent to the one in~\eqref{eq:def_HS_norm} (and so, unless otherwise stated, we will use both norms~\eqref{eq:def_HS_norm} and~\eqref{eq:def_HS_norm_bis} with no particular distinction). 
In particular, the operator 
\begin{equation*}
(-\Delta)^{\frac\alpha2}\colon HS^{\alpha,1}(\R^n)\to H^1(\R^n)
\end{equation*}
is well defined and continuous.

For the reader's convenience we briefly prove the following density result. 

\begin{lemma}[Approximation by $\mathcal{S}_{\infty}$ functions in $HS^{\alpha,1}$]
\label{res:approx_H_1_alpha}
Let $\alpha\in(0,1)$. 
The set $\mathcal{S}_{\infty}(\R^n)$ is dense in $HS^{\alpha,1}(\R^n)$.
\end{lemma}

\begin{proof}
Since the set $\mathcal S_{\infty}(\R^n)$  is dense in $H^1(\R^n)$ by~\cite{S93}*{Chapter~III, Section~5.2(a)}, the set $(I-\Delta)^{-\frac\alpha2}(\mathcal S_{\infty}(\R^n))$ is dense in $HS^{\alpha,1}(\R^n)$. 
Since clearly $(I-\Delta)^{-\frac\alpha2}(\mathcal S_{\infty}(\R^n))\subset\mathcal S_{\infty}(\R^n)$, the set $\mathcal S_{\infty}(\R^n)$ is dense (and embeds continuously) in $HS^{\alpha,1}(\R^n)$. 
Thus the conclusion follows.
\end{proof}

Exploiting \cref{res:approx_H_1_alpha}, for $\alpha\in(0,1)$, the space $HS^{\alpha,1}(\R^n)$  can be equivalently defined as the space
\begin{align*}
\set*{f\in H^1(\R^n) : \nabla^\alpha f\in H^1(\R^n;\R^n)}
\end{align*}
endowed with the norm
\begin{equation*}
f\mapsto\|f\|_{H^1(\R^n)}
+\|\nabla^\alpha f\|_{H^1(\R^n;\,\R^n)}.
\end{equation*} 
Indeed, if $f\in \mathcal{S}_{\infty}(\R^n)$, then, by exploiting Fourier transform techniques, we can write $\nabla^\alpha f=R\,(-\Delta)^{\frac\alpha2}f$, so that there exists a dimensional constant $c_n>0$ such that
\begin{equation}\label{eq:equivalence_norm_H_1_alpha}
c_n^{-1}\|(-\Delta)^{\frac\alpha2}f\|_{H^1(\R^n)}
\le\|\nabla^\alpha f\|_{H^1(\R^n;\,\R^n)}
\le c_n\|(-\Delta)^{\frac\alpha2}f\|_{H^1(\R^n)}
\end{equation}
for all $f\in \mathcal{S}_{\infty}(\R^n)$, thanks to the $H^1$-continuity property of the Riesz transform and the fact that 
\begin{equation*}
\sum_{j = 1}^n R_j^2 = -I\ 
\quad
\text{on}\ \mathcal{S}(\R^n),
\end{equation*} 
where $R_j$ is the $j$-th component of the Riesz transform $R$.
By \cref{res:approx_H_1_alpha}, the validity of~\eqref{eq:equivalence_norm_H_1_alpha} extends to all $f\in HS^{\alpha,1}(\R^n)$ and the conclusion follows. 
As a consequence, note that $HS^{\alpha,1}(\R^n)\subset S^{\alpha,1}(\R^n)$ for all $\alpha\in(0,1)$ with continuous embedding.

We note that the well-posedness and the equivalence of the definitions of $HS^{\alpha,1}(\R^n)$ given above and the stated results hold for any $\alpha\ge0$ thanks to the composition properties of the operators involved. We leave the standard verifications to the interested reader.

\section{The \texorpdfstring{$BV^0(\R^n)$}{BVˆ0(Rˆn)} space}
\label{sec:BV_0}

\subsection{Definition of \texorpdfstring{$BV^0(\R^n)$}{BVˆ0(Rˆn)} and Structure Theorem}

Somehow naturally extending the definitions given in~\eqref{eq:def_frac_operators_Riesz_potential} to the case $\alpha=0$, for $f\in\Lip_c(\R^n)$ and $\phi\in\Lip_c(\R^n;\R^n)$ we define
\begin{equation*}
\nabla^0 f:=I_1\nabla f
\quad
\text{and}
\quad
\div^0\phi :=I_1\div\phi. 
\end{equation*}
It is immediate to check that the integration-by-parts formula
\begin{equation}\label{eq:int_by_parts_smooth}
\int_{\R^n}f\,\div^0\phi\,dx
=-\int_{\R^n}\phi\cdot \nabla^0f\,dx
\end{equation}
holds for all given $f\in\Lip_c(\R^n)$ and $\phi\in\Lip_c(\R^n;\R^n)$. Hence, in analogy with~\cite{CS19}*{Definition~3.1}, we are led to the following definition (which is well posed, since $\div^0 \phi \in L^{\infty}(\R^n)$ for $\phi\in\Lip_c(\R^n;\R^n)$).

\begin{definition}[The space $BV^0(\R^n)$]
A function $f\in L^1(\R^n)$ belongs to the space $BV^0(\R^n)$ if 
\begin{equation*}
\sup\set*{\int_{\R^n}f\,\div^0\phi\,dx : \phi\in C^\infty_c(\R^n;\R^n),\ \|\phi\|_{L^\infty(\R^n;\,\R^n)}\le1}<+\infty.
\end{equation*}
\end{definition}

The proof of the following result is very similar to the one of~\cite{CS19}*{Theorem~3.2} and is omitted.

\begin{theorem}[Structure Theorem for $BV^0$ functions]
Let $f\in L^1(\R^n)$. Then, $f\in BV^0(\R^n)$ if and only if there exists a finite vector-valued Radon measure $D^0 f\in \!\mathscr M(\R^n;\R^n)$ such that
\begin{equation}\label{eq:int_by_parts_BV_0}
\int_{\R^n}f\,\div^0\phi\,dx
=-\int_{\R^n}\phi\cdot\,dD^0f
\end{equation}
for all $\phi\in C^\infty_c(\R^n;\R^n)$.
In addition, for all open sets $U\subset\R^n$ it holds 
\begin{equation}\label{eq:fractional_variation_0}
|D^0 f|(U)=\sup\set*{\int_{\R^n}f\,\div^0\phi\,dx : \phi\in C^\infty_c(U;\R^n),\ \|\phi\|_{L^\infty(U;\,\R^n)}\le1}.
\end{equation}
\end{theorem}

\subsection{The identification \texorpdfstring{$BV^0(\R^n)=H^1(\R^n)$}{BVˆ0(Rˆn)=Hˆ1(Rˆn)}}

As already announced in~\cite{CS19-2}, the space $BV^0(\R^n)$ actually coincides with the Hardy space $H^1(\R^n)$. More precisely, we have the following result. 

\begin{theorem}[The identification $BV^0=H^1$]
\label{res:H_1=BV_0}
We have $BV^0(\R^n)=H^1(\R^n)$, with
\begin{equation*}
D^0f=Rf\,\Leb{n}\ 
\text{in}\
\mathscr M(\R^n;\R^n)
\end{equation*} 
for every $f\in BV^0(\R^n)$.
\end{theorem}

\begin{proof}
We prove the two inclusions separately.

\smallskip

\textit{Proof of $H^1(\R^n)\subset BV^0(\R^n)$}. 
Let $f\in H^1(\R^n)$ and assume $f\in\Lip_c(\R^n)$.
By~\eqref{eq:int_by_parts_smooth}, we immediately get that $D^0 f=Rf\,\Leb{n}$ in~$\mathscr M(\R^n;\R^n)$ with $Rf=\nabla^0 f$ in $L^1(\R^n;\R^n)$, so that $f\in BV^0(\R^n)$. Now let $f\in H^1(\R^n)$. By~\cite{S93}*{Chapter~III, Section~5.2(b)}, we can find $(f_k)_{k\in\N}\subset H^1(\R^n)\cap C^\infty_c(\R^n)$ such that $f_k\to f$ in $H^1(\R^n)$ as $k\to+\infty$. Hence, given $\phi\in C^\infty_c(\R^n;\R^n)$, we have
\begin{equation*}
\int_{\R^n}f_k\,\div^0\phi\,dx
=-\int_{\R^n}\phi\cdot Rf_k\,dx
\end{equation*}
for all $k\in\N$. Passing to the limit as $k\to+\infty$, we get
\begin{equation*}
\int_{\R^n}f\,\div^0\phi\,dx
=-\int_{\R^n}\phi\cdot Rf\,dx	
\end{equation*}  
so that $f\in BV^0(\R^n)$ with $D^0f=Rf\,\Leb{n}$ in~$\mathscr M(\R^n;\R^n)$ according to~\eqref{eq:fractional_variation_0}.

\smallskip

\textit{Proof of $BV^0(\R^n)\subset H^1(\R^n)$}. Let $f\in BV^0(\R^n)$. Since $f\in L^1(\R^n)$, $Rf$ is well defined as a (vector-valued) distribution, see~\cite{S93}*{Chapter~III, Section~4.3}. Thanks to~\eqref{eq:int_by_parts_BV_0}, we also have that $\scalar*{Rf,\phi}=\scalar*{D^0f,\phi}$ for all $\phi\in C^\infty_c(\R^n;\R^n)$, so that $Rf=D^0f$ in the sense of distributions. Now let $(\rho_\eps)_{\eps>0}\subset C^\infty_c(\R^n)$ be a family of standard mollifiers (see e.g.~\cite{CS19}*{Section~3.2}). We can thus estimate
\begin{equation*}
\|Rf*\rho_\eps\|_{L^1(\R^n;\,\R^n)}
=\|D^0f*\rho_\eps\|_{L^1(\R^n;\,\R^n)}
\le|D^0 f|(\R^n)
\end{equation*}
for all $\eps>0$, so that $f\in H^1(\R^n)$ by \cite{S93}*{Chapter~III, Section~4.3, Proposition~3}, with $D^0f=Rf\Leb{n}$ in~$\mathscr M(\R^n;\R^n)$.
\end{proof}

\subsection{Relation between \texorpdfstring{$W^{\alpha,1}(\R^n)$}{Wˆ{alpha,1}(Rˆn)} and \texorpdfstring{$H^1(\R^n)$}{Hˆ1(Rˆn)}}

Thanks to the identification established in \cref{res:H_1=BV_0}, we can prove the following result. See also~\cite{CS19}*{Lemma~3.28} and~\cite{CS19-2}*{Lemma~3.11}.

\begin{proposition}\label{res:relation_W_alpha_1_and_H_1}
Let $\alpha\in(0,1)$. The following hold.
\begin{enumerate}[(i)]

\item\label{item:relation_W_alpha_1_and_H_1_1} 
If $f\in H^1(\R^n)$, then $u:=I_{\alpha} f\in BV^{\alpha,\frac{n}{n - \alpha}}(\R^n)$ with $D^\alpha u=R f \Leb{n}$ in~$\mathscr{M}(\R^n;\R^n)$.

\item\label{item:relation_W_alpha_1_and_H_1_2} 
If $u\in W^{\alpha,1}(\R^n)$, then $f:=(-\Delta)^{\alpha/2} u\in H^1(\R^n)$ with 
\begin{equation*}
\|f\|_{L^1(\R^n)}\le\mu_{n,-\alpha}[u]_{W^{\alpha,1}(\R^n)}
\quad\text{and}\quad
Rf=\nabla^\alpha u\ 
\text{a.e.\ in $\R^n$}.
\end{equation*} 
\end{enumerate}
\end{proposition}

\begin{proof}
We prove the two statements separately.

\smallskip

\textit{Proof of~\eqref{item:relation_W_alpha_1_and_H_1_1}}.
Let $f\in H^1(\R^n)$. By the Stein--Weiss inequality (see~\cite{SSS17}*{Theorem~2} for instance), we know that $u:=I_{\alpha} f\in L^{\frac{n}{n - \alpha}}(\R^n)$. To prove that $|D^{\alpha} u|(\R^n)<+\infty$, we exploit \cref{res:H_1=BV_0} and argue as in the proof of~\cite{CS19}*{Lemma~3.28}. Indeed, for all $\phi\in C^\infty_c(\R^n;\R^n)$, we can write
\begin{align*}
\int_{\R^n}f\,\div^0\phi\,dx
=\int_{\R^n}f\,I_{\alpha}\div^{\alpha}\phi\,dx
=\int_{\R^n}u\,\div^{\alpha}\phi\,dx
\end{align*}
by Fubini's Theorem, since $f\in L^1(\R^n)$ and $I_\alpha|\div^{\alpha}\phi|\in L^\infty(\R^n)$, being 
\begin{align*}
I_{\alpha}|\div^{\alpha}\phi|
=I_{\alpha}|I_{1 - \alpha}\div\phi|
\le I_{\alpha} I_{1 - \alpha}|\div\phi|
=I_1|\div\phi|\in L^\infty(\R^n)
\end{align*}
thanks to the semigroup property~\eqref{eq:Riesz_potential_semigroup} of the Riesz potentials. 
This proves that $D^{\alpha}u=D^0f = R f \Leb{n}$ in $\mathscr{M}(\R^n;\R^n)$, again thanks to \cref{res:H_1=BV_0}.

\smallskip

\textit{Proof of~\eqref{item:relation_W_alpha_1_and_H_1_2}}.  
Let $u\in W^{\alpha,1}(\R^n)$. Then $f:=(-\Delta)^{\alpha/2} u$ satisfies
\begin{align*}
\|f\|_{L^1(\R^n)}
=\mu_{n,-\alpha}\int_{\R^n}\bigg|\int_{\R^n}\frac{u(y)-u(x)}{|y-x|^{n+\alpha}}\,dy\,\bigg|\,dx
\le\mu_{n,-\alpha}[u]_{W^{\alpha,1}(\R^n)}.
\end{align*}
To prove that $f\in H^1(\R^n)$, we exploit \cref{res:H_1=BV_0} again. For all $\phi\in C^\infty_c(\R^n;\R^n)$, we can write
\begin{equation*}
\int_{\R^n}u\,\div^\alpha\phi\,dx
=\int_{\R^n}u\,(-\Delta)^\frac{\alpha}{2}\div^0\phi\,dx
=\int_{\R^n}f\,\div^0\phi\,dx
\end{equation*}
by Fubini's Theorem, since $u\in L^1(\R^n)$ and $\div^0\phi\in\Lip_b(\R^n;\R^n)$, proving that $D^0 f=D^\alpha u$ in~$\mathscr{M}(\R^n;\R^n)$. Since $D^\alpha u=\nabla^\alpha u\,\Leb{n}$ with $\nabla^{\alpha}u \in L^{1}(\R^{n}; \R^{n})$ by~\cite{CS19}*{Theorem~3.18} and $D^0f=Rf\,\Leb{n}$ by \cref{res:H_1=BV_0}, we see that $f=(-\Delta)^{\alpha/2} u\in H^1(\R^n)$ and $Rf = \nabla^{\alpha} u$ $\Leb{n}$-a.e., concluding the proof.
\end{proof}

We end this section with the following consequence of \cref{res:relation_W_alpha_1_and_H_1}.

\begin{corollary}
\label{res:scatole}
The following statements hold.
\begin{enumerate}[(i)]

\item\label{item:scatola_HS}
$H^1(\R^n)
\cap
\bigcup_{\alpha\in(0,1)}
W^{\alpha,1}(\R^n)
=
\bigcup_{\alpha\in(0,1)} 
HS^{\alpha,1}(\R^n).
$

\item\label{item:scatola_S}
$
\bigcup_{\alpha\in(0,1)}
S^{\alpha,p}(\R^n)
=
\bigcup_{\alpha\in(0,1)} 
W^{\alpha,p}(\R^n)
$
for all $p\in[1,+\infty)$.

\end{enumerate}
\end{corollary}

\begin{proof}
We prove the two statements separately.

\smallskip

\textit{Proof of~\eqref{item:scatola_HS}}.
On the one hand, we have $H^1(\R^n)\cap W^{\alpha,1}(\R^n)\subset HS^{\alpha,1}(\R^n)$ for all $\alpha\in(0,1)$ by \cref{res:relation_W_alpha_1_and_H_1}\eqref{item:relation_W_alpha_1_and_H_1_2} in virtue of the discussion made in \cref{subsec:HS_frac_space}.
On the other hand, $HS^{\alpha,1}(\R^n)\subset H^1(\R^n)\cap S^{\alpha,1}(\R^n)$ for all $\alpha\in(0,1)$ as remarked at the end of \cref{subsec:HS_frac_space}.
Since we already know that $S^{\alpha,1}(\R^n)\subset W^{\alpha',1}$ for all $0<\alpha'<\alpha<1$ by~\cite{CS19}*{Theorems~3.25 and~3.32}, this proves~\eqref{item:scatola_HS}. 

\smallskip

\textit{Proof of~\eqref{item:scatola_S}}.
Since $L^{\alpha+\eps,p}(\R^n)\subset W^{\alpha,p}(\R^n)\subset L^{\alpha-\eps,p}(\R^n)$ for all $\alpha\in(0,1)$, $p\in(1,+\infty)$ and $0<\eps<\min\{\alpha,1-\alpha\}$ by~\cite{A75}*{Theorem~7.63(g)}, thanks to the identification established in \cref{res:S=L} we immediately deduce the validity of~\eqref{item:scatola_S} for all $p\in(1,+\infty)$. 
If $p=1$, then~\eqref{item:scatola_S} is a consequence of~\cite{CS19}*{Proposition~3.24(i) and Theorems~3.25 and~3.32}.
\end{proof}

\section{Interpolation inequalities}
\label{sec:interpolation_inequalities}

\subsection{The case \texorpdfstring{$p=1$}{p=1} via the Calder\'on--Zygmund Theorem}\label{subsec:CZ_p=1}

Here and in the rest of the paper, let $(\eta_R)_{R>0}\subset C^\infty_c(\R^n)$ be a family of cut-off functions defined as 
\begin{equation}\label{eq:def_eta_function}
\eta_R(x)=\eta\left(\tfrac{|x|}{R}\right),
\quad
\text{for all $x\in\R^n$ and $R>0$},	
\end{equation}
where $\eta\in C^\infty_c(\R)$ satisfies
\begin{equation}\label{eq:def_cut_off}
0\le\eta\le1,
\quad
\eta=1\ \text{on}\ \left[-\tfrac12,\tfrac12\right],
\quad
\supp\eta\subset[-1,1]
\quad
\Lip(\eta)\le3.
\end{equation}

For $\alpha\in(0,1)$ and $R>0$, let $T_{\alpha,R}\colon\mathcal{S}(\R^n)\to\mathcal{S}'(\R^n;\R^n)$ be the linear operator defined by
\begin{equation}\label{eq:def_T_alpha_R}
T_{\alpha,R}f(x):=\int_{\R^n}f(y+x)\,\frac{y\,(1-\eta_R(y))}{|y|^{n+\alpha+1}}\,dy,
\quad
x\in\R^n,
\end{equation}
for all $f\in\mathcal{S}(\R^n)$. 
In the following result, we prove that~$T_{\alpha,R}$ is a Calder\'on--Zygmund operator mapping $H^1(\R^n)$ to~$L^1(\R^n;\R^n)$.

\begin{lemma}[Calder\'on--Zygmund estimate for~$T_{\alpha,R}$]\label{res:CZ}
There is a dimensional constant $\tau_n>0$ such that, for any given $\alpha\in(0,1)$ and $R>0$, the operator in~\eqref{eq:def_T_alpha_R} uniquely extends to a bounded linear operator $T_{\alpha,R}\colon H^1(\R^n)\to L^1(\R^n;\R^n)$ with
\begin{equation*}
\|T_{\alpha,R}f\|_{L^1(\R^n;\R^n)}
\le\tau_n R^{-\alpha}\|f\|_{H^1(\R^n)}
\end{equation*} 
for all $f\in H^1(\R^n)$.
\end{lemma}

\begin{proof}
We apply~\cite{G14-M}*{Theorem~2.4.1} to the kernel
\begin{equation*}
K_{\alpha,R}(x):=\frac{x\,(1-\eta_R(x))}{|x|^{n+\alpha+1}},
\quad
x\in\R^n,\ x\ne0.
\end{equation*}
First of all, we have
\begin{align*}
|K_{\alpha,R}(x)|\le\frac{1-\eta_R(x)}{|x|^{n+\alpha}}\le\frac{2^\alpha}{R^\alpha}\,\frac{1}{|x|^n},
\quad
x\in\R^n,\ x\ne0,
\end{align*}
so that we can choose $A_1=2n\omega_nR^{-\alpha}$ in the \emph{size estimate}~(2.4.1) in~\cite{G14-M}. We also have
\begin{equation*}
|\nabla K_{\alpha,R}(x)|\le c_n\bigg(\frac{1}{R}\frac{\left|\eta'\big(\tfrac{|x|}{R}\big)\right|}{|x|^{n+\alpha}}
+\frac{1-\eta_R(x)}{|x|^{n+\alpha+1}}\bigg)
\le 4c_n\,\frac{2^\alpha}{R^\alpha}\,\frac{1}{|x|^{n+1}},
\quad
x\in\R^n,\ x\ne0,
\end{equation*}
where $c_n>0$ is some dimensional constant, so that we can choose $A_2=c_n' R^{-\alpha}$ in the \emph{smoothness condition}~(2.4.2) in~\cite{G14-M}, where $c_n'>c_n$ is another dimensional constant. Finally, since clearly
\begin{equation*}
\int_{\{m<|x|<M\}} K_{\alpha,R}(x)\,dx=0
\end{equation*}
for all $m<M$, we can choose $A_3=0$ in the \emph{cancellation condition}~(2.4.3) in~\cite{G14-M}. Since $A_1+A_2+A_3=c_n''R^{-\alpha}$ for some dimensional constant $c_n''\ge c_n'$, the conclusion follows.
\end{proof}

With \cref{res:CZ} at our disposal, we can prove the following result.

\begin{theorem}[$H^1-BV^\alpha$ interpolation inequality]
\label{res:interpolation_H1_BV_alpha}
Let $\alpha\in(0,1]$. There exists a constant $c_{n,\alpha}>0$ such that
\begin{equation}\label{eq:interpolation_H1_BV_alpha}
[f]_{BV^\beta(\R^n)}\le c_{n,\alpha}\,\|f\|_{H^1(\R^n)}^{(\alpha-\beta)/\alpha}\,[f]_{BV^\alpha(\R^n)}^{\beta/\alpha}
\end{equation}
for all $\beta\in[0,\alpha)$ and all $f\in H^1(\R^n)\cap BV^\alpha(\R^n)$.
\end{theorem}

\begin{proof}
Let $\alpha\in(0,1]$ be fixed. Thanks to \cref{res:H_1=BV_0}, the case $\beta=0$ is trivial, so we assume~$\beta\in(0,\alpha)$. 
We can also assume that $[f]_{BV^\alpha(\R^n)}>0$ without loss of generality, since otherwise $f=0$ $\Leb{n}$-a.e.\ by \cite{CS19}*{Proposition~3.14} (note that the validity of \cite{CS19}*{Proposition~3.14} for all $f\in BV^\alpha(\R^n)$ follows by a simple approximation argument, thanks to~\cite{CS19}*{Theorem~3.8}).
Hence, in particular, we can assume $\|f\|_{L^1(\R^n)}>0$.
We divide the proof in three steps.

\textit{Step~1: stability as $\beta\to0^+$}. 
Let $f\in H^1(\R^n)\cap BV^\alpha(\R^n)$ and assume $f\in\Lip_b(\R^n)$. By~\cite{CS19-2}*{Lemma~2.3}, we can write
\begin{equation}\label{eq:interpolation_H1_BV_alpha_split_beta}
\begin{split}
&|\nabla^\beta f(x)|
=\mu_{n,\beta}\,\bigg|\int_{\R^n}\frac{y (f(y+x)-f(x))}{|y|^{n+\beta+1}}\,dy\,\bigg|\\
&\quad=\mu_{n,\beta}\,\bigg|\int_{\R^n}\eta_R(y)\,\frac{y (f(y+x)-f(x))}{|y|^{n+\beta+1}}\,dy
+\int_{\R^n}(1-\eta_R(y))\,\frac{y (f(y+x)-f(x))}{|y|^{n+\beta+1}}\,dy\,\bigg|	
\end{split}	
\end{equation}
for all $x\in\R^n$ and all $R>0$. On the one hand, for $\alpha<1$, by~\cite{CS19}*{Proposition~3.14} we can estimate
\begin{equation}\label{eq:interpolation_H1_BV_alpha_translation_beta}
\begin{split}
\int_{\R^n}\bigg|\int_{\R^n}\eta_R(y)\,\frac{y (f(y+x)-f(x))}{|y|^{n+\beta+1}}\,dy\,\bigg|\,dx
&\le
\int_{B_R}\frac{1}{|y|^{n+\beta}}\int_{\R^n}|f(y+x)-f(x)|\,dx\,dy\\
&\le
\gamma_{n,\alpha} \,|D^\alpha f|(\R^n)\int_{B_R}\frac{dy}{|y|^{n+\beta-\alpha}}\\
&=
n \omega_n \gamma_{n,\alpha} \,\frac{R^{\alpha-\beta}}{\alpha-\beta}\,|D^\alpha f|(\R^n)
\end{split}
\end{equation}
for all $R>0$, where $\gamma_{n,\alpha}>0$ is a constant depending only on~$n$ and $\alpha$. If $\alpha=1$ instead, we simply have
\begin{equation*}
\int_{\R^n}\bigg|\int_{\R^n}\eta_R(y)\,\frac{y (f(y+x)-f(x))}{|y|^{n+\beta+1}}\,dy\,\bigg|\,dx
\le\, 
n \omega_n \frac{R^{1-\beta}}{1-\beta}\,|D f|(\R^n)
\end{equation*}
for all $R>0$ (by \cite{AFP00}*{Remark 3.25} with $\Omega = \R^{n}$, for instance). On the other hand, by \cref{res:CZ} we have
\begin{equation}\label{eq:interpolation_H1_BV_alpha_CZ_beta}
\begin{split}
\int_{\R^n}\bigg|\int_{\R^n}(1-\eta_R(y))\,\frac{y (f(y+x)-f(x))}{|y|^{n+\beta+1}}\,dy\,\bigg|\,dx
&=\int_{\R^n}\bigg|\int_{\R^n}(1-\eta_R(y))\,\frac{y f(y+x)}{|y|^{n+\beta+1}}\,dy\,\bigg|\,dx\\
&\le\tau_n R^{-\beta}\|f\|_{H^1(\R^n)}	
\end{split}	
\end{equation}
for all $R>0$, where $\tau_n>0$ is the constant of \cref{res:CZ}. Combining the above estimates, we get
\begin{align*}
|D^\beta f|(\R^n)
&\le\mu_{n,\beta}\,\bigg(n \omega_n\gamma_{n,\alpha}\,\frac{R^{\alpha-\beta}}{\alpha-\beta}\,[f]_{BV^\alpha(\R^n)}
+\tau_n R^{-\beta}\,\|f\|_{H^1(\R^n)}\bigg)\\
&\le\mu_{n,\beta}\max\{\tau_n,n \omega_n\gamma_{n,\alpha}\}\,\bigg(\frac{R^{\alpha-\beta}}{\alpha-\beta}\,[f]_{BV^\alpha(\R^n)}
+R^{-\beta}\,\|f\|_{H^1(\R^n)}\bigg)
\end{align*}
for all $R>0$, where we have set $\gamma_{n,1}:=1$ by convention.  
With the choice $R=\|f\|_{H^1(\R^n)}^{1/\alpha}\,[f]_{BV^\alpha(\R^n)}^{-1/\alpha}$, we get
\begin{equation}\label{eq:step1_BV_alpha_beta_interpolation}
|D^\beta f|(\R^n)
\le\frac{2\mu_{n,\beta}\max\{\tau_n,n \omega_n\gamma_{n,\alpha}\}}{\alpha-\beta}
\,\|f\|_{H^1(\R^n)}^{(\alpha-\beta)/\alpha}
\,[f]_{BV^\alpha(\R^n)}^{\beta/\alpha}
\end{equation}
for all $f\in H^1(\R^n)\cap BV^\alpha(\R^n)$ such that $f\in\Lip_b(\R^n)$. Using a standard approximation argument via convolution, thanks to~\cite{CS19}*{Proposition~3.3} inequality~\eqref{eq:step1_BV_alpha_beta_interpolation} follows for all $f\in H^1(\R^n)\cap BV^\alpha(\R^n)$.  

\smallskip

\textit{Step~2: stability as $\beta\to\alpha^-$}.
If $\alpha<1$, then by~\cite{CS19-2}*{Proposition~3.12} we know that
\begin{equation}\label{eq:step2_BV_alpha_beta_interpolation_before}
|D^\beta f|(\R^n)
\le 
d_{n,\alpha}\,\frac{\mu_{n,1+\beta-\alpha}}{n+\beta-\alpha}
\,\bigg(\frac{R^{\alpha-\beta}}{\alpha-\beta}\,[f]_{BV^\alpha(\R^n)}+\frac{R^{-\beta}}{\beta}\|f\|_{L^1(\R^n)}\bigg)
\end{equation}
for all $f\in BV^\alpha(\R^n)$ and all $R>0$, where 
\begin{equation*}
d_{n, \alpha} = \max\left \{ n \omega_n, (n + \alpha) \|\nabla^{\alpha} \chi_{B_1}\|_{L^1(\R^n; \R^n)} \right \},
\end{equation*}
so that \cite{CS19-2}*{Theorem~4.9} implies
\begin{equation*}
d_{n,1}:=\lim_{\alpha\to1^-}d_{n,\alpha} = (n + 1)\, n \omega_n <+\infty.
\end{equation*}
If $\alpha=1$, then by~\cite{CS19-2}*{Proposition~3.2(i)} inequality~\eqref{eq:step2_BV_alpha_beta_interpolation_before} holds with $\alpha=1$ for all $f\in BV(\R^n)$.
Since $\|f\|_{L^1(\R^n)}>0$, choosing $R=[f]_{BV^\alpha(\R^n)}^{1/\alpha}\,\|f\|_{L^1(\R^n)}^{-1/\alpha}$ and using the inequality $\|f\|_{L^1(\R^n)}\le \|f\|_{H^1(\R^n)}$, we can estimate 
\begin{equation}\label{eq:step2_BV_alpha_beta_interpolation}
|D^\beta f|(\R^n)
\le\frac{d_{n,\alpha}}{\beta(\alpha-\beta)}\,\frac{\mu_{n,1+\beta-\alpha}}{n+\beta-\alpha}
\,\|f\|_{H^1(\R^n)}^{(\alpha-\beta)/\alpha}
\,[f]_{BV^\alpha(\R^n)}^{\beta/\alpha}
\end{equation}
for all $f\in H^1(\R^n)\cap BV^\alpha(\R^n)$.

\smallskip

\textit{Step~3: existence of $c_{n,\alpha}$}.
Combining~\eqref{eq:step1_BV_alpha_beta_interpolation} and~\eqref{eq:step2_BV_alpha_beta_interpolation}, we get
\begin{equation*}
|D^\beta f|(\R^n)
\le\phi_n(\alpha,\beta)
\,\|f\|_{H^1(\R^n)}^{(\alpha-\beta)/\alpha}\,[f]_{BV^\alpha(\R^n)}^{\beta/\alpha}
\end{equation*}
for all $f\in H^1(\R^n)\cap BV^\alpha(\R^n)$, where 
\begin{equation*}
\phi_n(\alpha,\beta):=\min\set*{
\frac{2\mu_{n,\beta}\max\{\tau_n,n \omega_n\gamma_{n,\alpha}\}}{\alpha-\beta}, 
\frac{d_{n,\alpha}}{\beta(\alpha-\beta)}\,\frac{\mu_{n,1+\beta-\alpha}}{n+\beta-\alpha}},
\quad
0<\beta<\alpha\le1.
\end{equation*}
We observe that, for all fixed $\alpha \in (0, 1]$, $\varphi_n(\alpha, \beta)$ is continuous in $\beta \in (0, \alpha)$. 
Thanks to \cite{CS19-2}*{Lemma~4.1}, we notice that for all $\alpha \in (0, 1)$ we have
\begin{equation*}
\lim_{\beta\to\alpha^-}\phi_n(\alpha,\beta)
=\frac{d_{n,\alpha}}{\alpha n}
\,\lim_{\beta\to\alpha^-}\frac{\mu_{n,1+\beta-\alpha}}{\alpha-\beta}
=\frac{d_{n,\alpha}}{\alpha n\omega_n},
\end{equation*}
while in the case $\alpha = 1$ we obtain
\begin{align*}
\lim_{\beta\to 1^-}\phi_n(1,\beta) 
& = 
\min\set*{
2\max\{\tau_n, n \omega_n\} \lim_{\beta\to 1^-} \frac{\mu_{n,\beta}}{1-\beta},\  d_{n,1}\lim_{\beta\to 1^-}\frac{\mu_{n,\beta}}{\beta(1-\beta)(n + \beta - 1)}} \\
& = 
\frac{1}{\omega_n} \min\set*{2 \max\{\tau_n, n \omega_n\}, \frac{d_{n,1}}{n}}.
\end{align*}
In addition, for all $\alpha \in (0, 1]$, we get
\begin{equation*}
\lim_{\beta\to0^+}\phi_n(\alpha,\beta)
=\frac{2\mu_{n,0}\max\{\tau_n,n \omega_n\gamma_{n,\alpha}\}}{\alpha}.
\end{equation*}
Thus, for all $\alpha \in (0, 1]$ we have $\varphi_n(\alpha, \cdot) \in C([0, \alpha])$, and the conclusion follows by setting $c_{n, \alpha} := \max_{\beta \in [0, \alpha]} \varphi_n(\alpha, \beta)$.
\end{proof}

\begin{remark}[$H^1-W^{\alpha,1}$ interpolation inequality]
\label{rem:interpolation_W_alpha_1_H_1}
Thanks to~\cite{CS19}*{Theorem~3.18}, by \cref{res:interpolation_H1_BV_alpha} one can replace the $BV^\alpha$-seminorm in the right-hand side of~\eqref{eq:interpolation_H1_BV_alpha} with the $W^{\alpha,1}$-seminorm up to multiply the constant~$c_{n,\alpha}$ by~$\mu_{n,\alpha}$. However, one can prove a slightly finer estimate essentially following the proof of \cref{res:interpolation_H1_BV_alpha}. 
Indeed, for any given $f\in H^1(\R^n)\cap W^{\alpha,1}(\R^n)$ sufficiently regular, one writes $\nabla^\beta f$ as in~\eqref{eq:interpolation_H1_BV_alpha_split_beta} and estimates the second part of it as in~\eqref{eq:interpolation_H1_BV_alpha_CZ_beta}. To estimate the first term, instead of following~\eqref{eq:interpolation_H1_BV_alpha_translation_beta}, one simply notes that
\begin{align*}
\int_{\R^n}\bigg|\int_{\R^n}\eta_R(y)\,\frac{y (f(y+x)-f(x))}{|y|^{n+\beta+1}}\,dy\,\bigg|\,dx
&\le\int_{\R^n}\int_{B_R}\frac{|f(y+x)-f(x)|}{|y|^{n+\beta}}\,dy\,dx\\
&\le R^{\alpha-\beta}\int_{\R^n}\int_{B_R}\frac{|f(y+x)-f(x)|}{|y|^{n+\alpha}}\,dy\,dx\\
&\le R^{\alpha-\beta}\,[f]_{W^{\alpha,1}(\R^n)}
\end{align*}
for all $R>0$. Hence\begin{align*}
|D^\beta f|(\R^n)
&\le\mu_{n,\beta}\big(R^{\alpha-\beta}\,[f]_{W^{\alpha,1}(\R^n)}
+\tau_n R^{-\beta}\,\|f\|_{H^1(\R^n)}\big)
\end{align*}
for all $R>0$, and the desired inequality follows by optimizing the parameter~$R>0$ in the right-hand side.
\end{remark}

\subsection{The cases \texorpdfstring{$p>1$}{p>1} and \texorpdfstring{$H^1$}{Hˆ1} via the Mihlin--H\"ormander Multiplier Theorem}

Let $0\le\beta\le\alpha\le1$ and consider the function
\begin{equation*}
m_{\alpha,\beta}(\xi):=\frac{|\xi|^\beta}{1+|\xi|^\alpha},
\quad
\xi\in\R^n.
\end{equation*}
It is not difficult to see that
\begin{equation*}
\|m_{\alpha,\beta}\|_\star
:=
\sup_{\mathrm{a}\in\N^n_0,\ |\mathrm{a}|\le\left\lfloor\!\frac n2 \!\right\rfloor+1}
\ 
\sup_{\xi\in\R^n\setminus\set{0}}
\ 
\Big|\,\xi^{\mathrm{a}}\,\de^{\mathrm{a}}_\xi \, m_{\alpha,\beta}(\xi)\,\Big|
<+\infty.
\end{equation*} 
We thus define the convolution operator $T_{m_{\alpha,\beta}}\colon\mathcal S(\R^n)\to\mathcal S'(\R^n)$ with convolution kernel given by $\mathcal F^{-1}(m_{\alpha,\beta})$, i.e.,
\begin{equation}\label{eq:def_MH}
T_{m_{\alpha,\beta}}f:=f*\mathcal F^{-1}(m_{\alpha,\beta}),
\quad
f\in\mathcal S(\R^n).
\end{equation}
In the following result, we observe that the multipliers $m_{\alpha,\beta}$ satisfy uniform Mihlin--H\"ormander conditions as $0\le\beta\le\alpha\le1$.

\begin{lemma}[Mihlin--H\"ormander estimates for $T_{m_{\alpha,\beta}}$]
\label{res:MH}
There is a dimensional constant $\sigma_n>0$ such that the following properties hold for all given $0\le\beta\le\alpha\le1$.

\begin{enumerate}[(i)]

\item\label{item:MH_p} For all given $p\in(1,+\infty)$, the operator in~\eqref{eq:def_MH} uniquely extends to a bounded linear operator $T_{m_{\alpha,\beta}}\colon L^p(\R^n) \to L^p(\R^n)$ with
\begin{equation*}
\|T_{m_{\alpha,\beta}}f\|_{L^p(\R^n)}
\le
\sigma_n\max\set*{p,\frac{1}{p-1}}
\,
\|f\|_{L^p(\R^n)}
\end{equation*}
for all $f\in L^p(\R^n)$.

\item\label{item:MH_1} The operator in~\eqref{eq:def_MH} uniquely extends to a bounded linear operator $T_{m_{\alpha,\beta}}\!\colon\! H^1(\R^n) \to H^1(\R^n)$ with
\begin{equation*}
\|T_{m_{\alpha,\beta}}f\|_{H^1(\R^n)}
\le
\sigma_n
\,
\|f\|_{H^1(\R^n)}
\end{equation*}
for all $f\in H^1(\R^n)$.

\end{enumerate}

\end{lemma} 

\begin{proof}
Statements~\eqref{item:MH_p} and~\eqref{item:MH_1} follow from the Mihlin--H\"ormander Multiplier Theorem, see~\cite{G14-C}*{Theorem~6.2.7} for the $L^p$-continuity and~\cite{GR85}*{Chapter~III, Theorem~7.30} for the $H^1$-continuity, where
\begin{equation*}
\sigma_n:=c_n\sup_{0\le\beta\le\alpha\le1}\|m_{\alpha,\beta}\|_\star<+\infty
\end{equation*}
with $c_n>0$ a dimensional constant.
We leave the simple verifications to the interested reader.
\end{proof}

With \cref{res:MH} at our disposal, we can prove the following result.

\begin{theorem}[Bessel and fractional Hardy--Sobolev interpolation inequalities]\label{res:interpolation_MH}
The following statements hold.

\begin{enumerate}[(i)]

\item\label{item:interpolation_MH_p}
Given $p\in(1,+\infty)$, there exists a constant $c_{n,p}>0$ such that, given $0\le\gamma\le\beta\le\alpha\le1$, it holds
\begin{equation}\label{eq:interpolation_MH_p}
\|\nabla^\beta f\|_{L^p(\R^n;\,\R^n)}
\le
c_{n,p}
\,
\|\nabla^\gamma f\|_{L^p(\R^n;\,\R^n)}^{\frac{\alpha-\beta}{\alpha-\gamma}}
\,
\|\nabla^\alpha f\|_{L^p(\R^n;\,\R^n)}^{\frac{\beta-\gamma}{\alpha-\gamma}}
\end{equation}
for all $f\in S^{\alpha,p}(\R^n)$. In the case $\gamma=0$ and $0\le\beta\le\alpha\le1$, we also have
\begin{equation}\label{eq:interpolation_MH_p_gamma=0}
\|\nabla^\beta f\|_{L^p(\R^n;\,\R^n)}
\le
c_{n,p}
\,
\|f\|_{L^p(\R^n)}^{\frac{\alpha-\beta}{\alpha}}
\,
\|\nabla^\alpha f\|_{L^p(\R^n;\,\R^n)}^{\frac{\beta}{\alpha}}
\end{equation}
for all $f\in S^{\alpha,p}(\R^n)$.

\item\label{item:interpolation_MH_1} 
There exists a dimensional constant $c_n>0$ such that, given $0\le\gamma\le\beta\le\alpha\le1$, it holds 
\begin{equation}\label{eq:interpolation_MH_1}
\|\nabla^\beta f\|_{H^1(\R^n;\,\R^n)}
\le
c_n
\,
\|\nabla^\gamma f\|_{H^1(\R^n;\,\R^n)}^{\frac{\alpha-\beta}{\alpha-\gamma}}
\,
\|\nabla^\alpha f\|_{H^1(\R^n;\,\R^n)}^{\frac{\beta-\gamma}{\alpha-\gamma}}
\end{equation}
for all $f\in HS^{\alpha,1}(\R^n)$. In the case $\gamma=0$ and $0\le\beta\le\alpha\le1$, we also have
\begin{equation}\label{eq:interpolation_MH_1_gamma=0}
\|\nabla^\beta f\|_{H^1(\R^n;\,\R^n)}
\le
c_n
\,
\|f\|_{H^1(\R^n)}^{\frac{\alpha-\beta}{\alpha}}
\,
\|\nabla^\alpha f\|_{H^1(\R^n;\,\R^n)}^{\frac{\beta}{\alpha}}
\end{equation}
for all $f\in HS^{\alpha,1}(\R^n)$. 

\end{enumerate}

\end{theorem}

\begin{proof}
Without loss of generality, we can directly assume that $0\le\gamma<\beta<\alpha\le1$. We prove  the two statements separately.

\smallskip

\textit{Proof of \eqref{item:interpolation_MH_p}}. 
Given $f\in S^{\alpha,p}(\R^n)$, we can write
\begin{equation*}
(-\Delta)^{\frac\beta2}f
=T_{m_{\alpha,\beta}}\circ(\mathrm{Id}+(-\Delta)^{\frac\alpha2})f,
\end{equation*}
so that
\begin{align*}
\|(-\Delta)^{\frac\beta2}f\|_{L^p(\R^n)}
&=
\|T_{m_{\alpha,\beta}}\circ(\mathrm{Id}+(-\Delta)^{\frac\alpha2})f\|_{L^p(\R^n)}\\
&\le\sigma_n\max\set*{p,\tfrac{1}{p-1}}
\,
\|f+(-\Delta)^{\frac\alpha2}f\|_{L^p(\R^n)}\\
&\le\sigma_n\max\set*{p,\tfrac{1}{p-1}}
\,
\big(\|f\|_{L^p(\R^n)}+\|(-\Delta)^{\frac\alpha2}f\|_{L^p(\R^n)}\big)
\end{align*}
thanks to \cref{res:MH}\eqref{item:interpolation_MH_p}. By performing a dilation and by optimizing the right-hand side, we find that
\begin{equation*}
\|(-\Delta)^{\frac\beta2}f\|_{L^p(\R^n)}
\le
\sigma_n\max\set*{p,\tfrac{1}{p-1}}
\,
\|f\|_{L^p(\R^n)}^{\frac{\alpha-\beta}{\alpha}}
\,
\|(-\Delta)^{\frac\alpha2}f\|_{L^p(\R^n)}^\frac{\beta}{\alpha}
\end{equation*}
for all $f\in S^{\alpha,p}(\R^n)$. Now let $f\in C^\infty_c(\R^n)$. Since 
\begin{equation*}
(-\Delta)^{\frac\alpha2}\nabla^\gamma f
=R\,(-\Delta)^{\frac{\alpha+\gamma}2}f\in L^p(\R^n;\R^n)
\end{equation*}
because $f\in L^{\alpha+\gamma,p}(\R^n)$ and by the $L^p$-continuity property of the Riesz transform, we get that $\nabla^\gamma f\in S^{\alpha,p}(\R^n;\R^n)$ according to the definition given in~\eqref{eq:def2_Bessel_space} and the identification established in~\cref{res:S=L}. 
Repeating the above computations for (each component of) the function $\nabla^\gamma f\in S^{\alpha,p}(\R^n;\R^n)$ with exponents~$\alpha-\gamma$ and~$\beta-\gamma$ in place of~$\alpha$ and~$\beta$ respectively and then optimizing, we get
\begin{align*}
\|\nabla^\beta f\|_{L^p(\R^n;\,\R^n)}
&=\|(-\Delta)^{\frac{\beta-\gamma}2}\nabla^\gamma f\|_{L^p(\R^n;\R^n)}\\
&\le
c_{n,p}
\,
\|\nabla^\gamma f\|_{L^p(\R^n;\,\R^n)}^{\frac{\alpha-\beta}{\alpha-\gamma}}
\,
\|(-\Delta)^{\frac{\alpha-\gamma}2}\nabla^\gamma f\|_{L^p(\R^n;\,\R^n)}^\frac{\beta-\gamma}{\alpha-\gamma}\\
&=
c_{n,p}
\,
\|\nabla^\gamma f\|_{L^p(\R^n;\,\R^n)}^{\frac{\alpha-\beta}{\alpha-\gamma}}
\,
\|\nabla^\alpha f\|_{L^p(\R^n;\,\R^n)}^\frac{\beta-\gamma}{\alpha-\gamma}
\end{align*} 
for all $f\in C^\infty_c(\R^n)$, where 
\begin{equation*}
c_{n,p}=\sigma_n n^{1/{2p}}\max\set*{p,\tfrac{1}{p-1}}.
\end{equation*}
Thanks to \cref{res:L_p_alpha_equal_S_p_alpha}, \cref{res:lsc_norm_S_alpha_p} and \cref{res:Davila_estimate_S_alpha_p}, inequality~\eqref{eq:interpolation_MH_p} follows by performing a standard approximation argument.

In the case $\gamma=0$, inequality~\eqref{eq:interpolation_MH_p_gamma=0} follows from~\eqref{eq:interpolation_MH_p} by the $L^p$-continuity of the Riesz transform. This concludes the proof of~\eqref{item:interpolation_MH_p}.

\smallskip

\textit{Proof of~\eqref{item:interpolation_MH_1}}. 
Given $f\in HS^{\alpha,1}(\R^n)$, arguing as above, we can write
\begin{equation*}
(-\Delta)^{\frac\beta2}f
=T_{m_{\alpha,\beta}}\circ(\mathrm{Id}+(-\Delta)^{\frac\alpha2})f,
\end{equation*}
so that
\begin{align*}
\|(-\Delta)^{\frac\beta2}f\|_{H^1(\R^n)}
\le\sigma_n
\,
\big(\|f\|_{H^1(\R^n)}+\|(-\Delta)^{\frac\alpha2}f\|_{H^1(\R^n)}\big)
\end{align*}
thanks to \cref{res:MH}\eqref{item:interpolation_MH_1}. By performing a dilation and by optimising the right-hand side, we find that
\begin{equation*}\|(-\Delta)^{\frac\beta2}f\|_{H^1(\R^n)}
\le
\sigma_n
\,
\|f\|_{H^1(\R^n)}^{\frac{\alpha-\beta}{\alpha}}
\,
\|(-\Delta)^{\frac\alpha2}f\|_{H^1(\R^n)}^\frac{\beta}{\alpha}
\end{equation*}
for all $f\in HS^{\alpha,1}(\R^n)$. Now let $f\in C^\infty_c(\R^n)$. Note that $\nabla^\gamma f\in H^1(\R^n;\R^n)$, because $\nabla^\gamma f\in L^1(\R^n;\R^n)$ and
\begin{equation*}
\div^0\nabla^\gamma f
=\div^0R(-\Delta)^{\frac\gamma2}f
=(-\Delta)^{\frac\gamma2}f
\in H^1(\R^n)
\end{equation*}
by \cref{res:relation_W_alpha_1_and_H_1}\eqref{item:relation_W_alpha_1_and_H_1_2}. Moreover,  
\begin{equation*}
(-\Delta)^{\frac\alpha2}\nabla^\gamma f
=R\,(-\Delta)^{\frac{\alpha+\gamma}2}f\in H^1(\R^n;\R^n)
\end{equation*}
because $f\in HS^{\alpha+\gamma,1}(\R^n)$ and by the $H^1$-continuity property of the Riesz transform. Thus $\nabla^\gamma f\in HS^{\alpha,1}(\R^n;\R^n)$. 
Repeating the above computations for (each component of) the function $\nabla^\gamma f\in HS^{\alpha,1}(\R^n;\R^n)$ with exponents~$\alpha-\gamma$ and~$\beta-\gamma$ in place of~$\alpha$ and~$\beta$ respectively and then optimizing, we get
\begin{align*}
\|\nabla^\beta f\|_{H^1(\R^n;\,\R^n)}
\le
c_n
\,
\|\nabla^\gamma f\|_{H^1(\R^n;\,\R^n)}^{\frac{\alpha-\beta}{\alpha-\gamma}}
\,
\|\nabla^\alpha f\|_{H^1(\R^n;\,\R^n)}^\frac{\beta-\gamma}{\alpha-\gamma}
\end{align*} 
for all $f\in C^\infty_c(\R^n)$, where $c_n=\sigma_n n^{1/{2}}$. Thanks to \cref{res:approx_H_1_alpha}, inequality~\eqref{eq:interpolation_MH_1} follows by performing a standard approximation argument.

In the case $\gamma=0$, inequality~\eqref{eq:interpolation_MH_1_gamma=0} follows from~\eqref{eq:interpolation_MH_p} by the $H^1$-continuity of the Riesz transform. This concludes the proof of~\eqref{item:interpolation_MH_1}.
\end{proof}

\section{Asymptotic behavior of fractional \texorpdfstring{$\alpha$}{alpha}-variation as \texorpdfstring{$\alpha\to0^+$}{alpha tends to 0ˆ+}}
\label{sec:asymptotic_to_zero}

In this section, we study the asymptotic behavior of~$\nabla^\alpha$ as $\alpha\to0^+$. 

\subsection{Pointwise convergence of \texorpdfstring{$\nabla^\alpha$}{nablaˆalpha} as \texorpdfstring{$\alpha\to0^+$}{alpha tends to 0ˆ+}}

We start with the pointwise convergence of~$\nabla^\alpha$ to~$\nabla^0$ as~$\alpha\to0^+$ for sufficiently regular functions.

\begin{lemma}[Uniform convergence of~$\nabla^\alpha$ as $\alpha\to0^+$]
\label{res:pointwise_conv_Riesz}
Let $\alpha\in(0,1]$ and $p\in [1,+\infty]$. 
For $\beta\in (0,\alpha)$, the operator
\begin{equation*}
\nabla^\beta\colon C^{0,\alpha}_{\loc}(\R^n)\cap L^p(\R^n)\to C^0(\R^n; \R^{n})
\end{equation*}
is well defined and satisfies
\begin{equation}\label{eq:estimate_nabla_beta_general}
\|\nabla^\beta f\|_{L^\infty(B_R;\,\R^{n})}
\le c_{n,p}\,\mu_{n,\beta}\,
\left(\frac{r^{\alpha-\beta}}{\alpha-\beta}\,[f]_{C^{0,\alpha}(B_{R+r})}
+
\frac{r^{-\frac np-\beta}}{\left(\frac{n}{p} + \beta \right)^{1 - \frac{1}{p}}}\,\|f\|_{L^p(\R^n)}\right)
\end{equation}
for all $r,R>0$ and all $f\in C^{0,\alpha}_{\loc}(\R^n)\cap L^p(\R^n)$, where
\begin{equation}\label{eq:baramba_constant}
c_{n,p}:= 
\begin{cases} 
\max \left \{n \omega_n, (n \omega_n)^{1 - \frac{1}{p}} \left (1 - \frac{1}{p} \right )^{1 - \frac{1}{p}} \right \} 
& \text{if } p \in (1,+\infty), \\
\max \left \{n \omega_n, 1 \right \} 
& \text{if } p = 1, \\
n \omega_n 
& \text{if } p =+\infty.
\end{cases}
\end{equation}
Moreover, for $\beta \in (0, \alpha)$ and $f \in C^{0,\alpha}(\R^n)\cap L^p(\R^n)$, we have $\nabla^\beta f \in C^0_{b}(\R^n;\R^n)$ and
\begin{equation}\label{eq:estimate_nabla_beta_general_global}
\|\nabla^\beta f\|_{L^\infty(\R^n;\,\R^{n})}
\le 
c_{n,p}\, \mu_{n,\beta}\, \frac{\alpha p + n}{(\alpha - \beta)(\beta p + n)}\,  \left ( \tfrac{n}{p} + \beta \right )^{\frac{\alpha - \beta}{\alpha p + n}} \|f\|_{L^p(\R^n)}^{\frac{p(\alpha - \beta)}{\alpha p + n}}\, [f]_{C^{0,\alpha}(\R^n)}^{\frac{\beta p + n}{\alpha p + n}},
\end{equation} 
where $c_{n,p}$ is as in~\eqref{eq:baramba_constant}.

Finally, if $p<+\infty$ and $f\in C^{0,\alpha}_{\loc}(\R^n)\cap L^p(\R^n)$, then $\nabla^{0} f$ is well defined and belongs to $C^{0}(\R^n; \R^n)$, \eqref{eq:estimate_nabla_beta_general_global} holds for $\beta = 0$, for all bounded open sets $U \subset \R^n$ we have
\begin{equation}\label{eq:pointwise_conv_nabla_beta}
\lim_{\beta\to0^+} \|\nabla^\beta f - \nabla^0 f\|_{L^{\infty}(U;\,\R^n)} = 0,
\end{equation}
and \eqref{eq:pointwise_conv_nabla_beta} holds for $U = \R^n$ if $f\in C^{0,\alpha}(\R^n)\cap L^p(\R^n)$ and $p<+\infty$.
\end{lemma}

\begin{proof}
We divide the proof in four steps.

\smallskip

\textit{Step~1: proof of~\eqref{eq:estimate_nabla_beta_general}}.
Let $\alpha \in (0, 1]$, $p \in [1, +\infty]$, $f\in C^{0,\alpha}_{\loc}(\R^n)\cap L^p(\R^n)$, $\beta\in (0,\alpha)$ and $x\in\R^n$. We notice that, for all $\eps \in (0, 1)$, 
\begin{equation*}
\int_{\{|y|>\eps\}}\frac{y f(y+x)}{|y|^{n+\beta+1}}\,dy = \int_{\{\eps < |y| \le 1\}}\frac{y (f(y+x) - f(x))}{|y|^{n+\beta+1}}\,dy + \int_{\{|y|>1\}}\frac{y f(y+x)}{|y|^{n+\beta+1}}\,dy,
\end{equation*}
so that we can pass to the limit in the right hand side as $\eps \to 0^+$ thanks to H\"older's continuity and the fact that $y\mapsto|y|^{-n - \beta} \in L^q(\R^n \setminus B_1)$ for all $q \in [1,+\infty]$. This shows that $\nabla^{\beta}f(x)$ is well defined for all $x \in \R^n$. If $p \in [1, +\infty)$, this argument works also in the case $\beta = 0$.
Now let $\alpha \in (0, 1]$, $\beta\in[0,\alpha)$, $p \in (1, +\infty)$, $f\in C^{0,\alpha}_{\loc}(\R^n)\cap L^p(\R^n)$ and $x\in\R^n$. By H\"older's inequality we can estimate
\begin{align*}
\bigg|\int_{\{|y|>\eps\}}\frac{y f(y+x)}{|y|^{n+\beta+1}}\,dy\bigg|
&\le\int_{\{\eps<|y|<r\}}\frac{|f(y+x)-f(x)|}{|y|^{n+\beta}}\,dy
+\int_{\{|y|\ge r\}}\frac{|f(y+x)|}{|y|^{n+\beta}}\,dy\\
&\le[f]_{C^{0,\alpha}(B_r(x))}\int_{\{|y|<r\}}\frac{dy}{|y|^{n+\beta-\alpha}}
+\|f\|_{L^p(\R^n)}\bigg(\int_{\{|y|\ge r\}}\frac{dy}{|y|^{(n+\beta)q}}\bigg)^{\frac{1}{q}}\\
&\le \frac{n\omega_n r^{\alpha-\beta}}{\alpha-\beta}\,[f]_{C^{0,\alpha}(B_r(x))}
+\bigg(\frac{n\omega_n r^{n-(n+\beta)q}}{(n+\beta)q-n}\bigg)^{\frac1q}\,\|f\|_{L^p(\R^n)}
\end{align*}
for all $r>\eps>0$, where $q = \frac{p}{p - 1}$. 
Moreover, for $p=1$, if $f\in C^{0,\alpha}_{\loc}(\R^n)\cap L^1(\R^n)$, then an analogous calculation shows that 
\begin{equation*}
\bigg|\int_{\{|y|>\eps\}}\frac{y f(y+x)}{|y|^{n+\beta+1}}\,dy\,\bigg|
\le 
\frac{n\omega_n r^{\alpha-\beta}}{\alpha-\beta}\,[f]_{C^{0,\alpha}(B_r(x))} 
+
r^{- n - \beta}\,\|f\|_{L^1(\R^n)}
\end{equation*}
for all $r>\eps>0$.
Finally, for $p = +\infty$, if $\beta\in (0,\alpha)$ and $f\in C^{0,\alpha}_{\loc}(\R^n)\cap L^\infty(\R^n)$, then we similarly obtain 
\begin{equation*}
\bigg|\int_{\{|y|>\eps\}}\frac{y f(y+x)}{|y|^{n+\beta+1}}\,dy\,\bigg|
\le \frac{n\omega_n r^{\alpha-\beta}}{\alpha-\beta}\, [f]_{C^{0,\alpha}(B_r(x))} + \frac{n\omega_n r^{- \beta}}{\beta} \|f\|_{L^\infty(\R^n)}
\end{equation*} 
for all $r>\eps>0$.
Thus we obtain $\nabla^{\beta} f \in L^\infty_{\loc}(\R^n;\R^n)$ for all $f\in C^{0,\alpha}_{\loc}(\R^n)\cap L^p(\R^n)$ with $\beta\in(0,\alpha)$ and $p\in[1,+\infty]$, including $\beta=0$ if $p<+\infty$,  and~\eqref{eq:estimate_nabla_beta_general} readily follows.

\smallskip

\textit{Step~2: proof of $\nabla^\beta f\in C^0(\R^n;\R^n)$}.
Let us now prove that $\nabla^\beta f\in C^0(\R^n;\R^n)$ for any $\beta \in (0,\alpha)$ and $f\in C^{0,\alpha}_{\loc}(\R^n)\cap L^p(\R^n)$, where $\alpha\in (0,1]$ and $p\in [1,+\infty]$. 
Let $R >0$, $r > 1$, $x\in B_R$, $h\in B_1$, $\beta<\alpha'<\alpha$ and $g_h(x):=f(x+h) - f(x)$. 
We notice that
\begin{equation}\label{eq:z}
[g_h]_{C^{0,\alpha'}(B_{R+r})} \le 2 [f]_{C^{0,\alpha}(B_{R+r+|h|})} |h|^{\alpha-\alpha'}.
\end{equation}
Indeed, given $x,x+h'\in B_{R+r}$ with $|h'|\le |h|$ we have
\begin{align*}
	|g_h(x+h') - g_h(x)| & \le |f(x+h+h') - f(x+h)| + |f(x+h') - f(x)|
	\\& \le 2 [f]_{C^{0,\alpha}(B_{R+r+|h|})} |h'|^\alpha
	\\& \le 2 [f]_{C^{0,\alpha}(B_{R+r+|h|})} |h'|^{\alpha'}\, |h|^{\alpha-\alpha'}.
\end{align*}
While, in the case $|h|\le |h'|$, it holds
\begin{align*}
	|g_h(x+h') - g_h(x)| & \le |f(x+h+h') - f(x+h')| + |f(x+h) - f(x)|
	\\& \le 2 [f]_{C^{0,\alpha}(B_{R+r+|h|})} |h|^\alpha
	\\& \le 2 [f]_{C^{0,\alpha}(B_{R+r+|h|})} |h'|^{\alpha'}\, |h|^{\alpha-\alpha'},
\end{align*}
therefore \eqref{eq:z} easily follows.
By plugging $g_h(x)$ in \eqref{eq:estimate_nabla_beta_general} with $\alpha'$ in place of $\alpha$ and $r>0$ we obtain
\begin{align*}
	|\nabla^\beta f(x+h) - \nabla^{\beta} f(x)|  & \le  c_{n,p}\,\mu_{n,\beta}\,
	\left(\frac{r^{\alpha'-\beta}}{\alpha'-\beta}\,[g_h]_{C^{0,\alpha'}(B_{R+r})}
	+
	\frac{r^{-\frac{n}{p}-\beta}}{\left(\frac{n}{p} + \beta \right)^{1 - \frac{1}{p}}}\,\|g_h\|_{L^p(\R^n)}\right)
	\\ & \le C_{n,p,\beta} \left( \frac{r^{\alpha'-\beta}}{\alpha'-\beta}|h|^{\alpha-\alpha'} [f]_{C^{0,\alpha}(B_{R+r+|h|})} + r^{-\frac{n}{p}-\beta}\| f\|_{L^p(\R^n)} \right),
	\end{align*}
where $C_{n,p,\beta}>0$ is a constant depending only on $n$, $p$ and $\beta$. The sought conclusion comes by letting first $h \to 0$ and after $r\to+\infty$.

\smallskip

\textit{Step~3: proof of~\eqref{eq:estimate_nabla_beta_general_global}}.
Let $\alpha \in (0,1]$, $p \in [1,+\infty]$ and $x \in \R^n$.
If $f \in C^{0,\alpha}(\R^n)\cap L^p(\R^n)$, then arguing as in Step~1 we can estimate 
\begin{equation*}
|\nabla^\beta f(x)| 
\le 
c_{n,p}\,\mu_{n,\beta}\,\left(\frac{r^{\alpha-\beta}}{\alpha-\beta}\,[f]_{C^{0,\alpha}(\R^n)}
+\frac{r^{-\frac np-\beta}}{\left (\frac{n}{p} + \beta \right )^{1 - \frac{1}{p}}}\,\|f\|_{L^p(\R^n)}\right),
\end{equation*}
for all $\beta \in (0, \alpha)$, including $\beta = 0$ if $p<+\infty$, so that \eqref{eq:estimate_nabla_beta_general_global} follows by optimizing the parameter $r > 0$ in the right-hand side.
 
\smallskip

\textit{Step~4: proof of~\eqref{eq:pointwise_conv_nabla_beta}}.
Let $\alpha\in(0,1]$, $\beta \in (0, \alpha)$, $U$ be a bounded open set and $x \in U$.
If $p\in(1,+\infty)$, then we can estimate
\begin{align*}
|\nabla^\beta f(x)-\nabla^0 f(x)|
&\le
\bigg|1-\frac{\mu_{n,\beta}}{\mu_{n,0}}\,\bigg|\,|\nabla^0 f(x)|
+
\mu_{n,\beta}\,[f]_{C^{0,\alpha}(B_1(x))}\int_{\{|y|<1\}}\bigg(\frac{1}{|y|^\beta}-1\bigg)\,\frac{dy}{|y|^{n-\alpha}}\\
&\quad
+\mu_{n,\beta}\int_{\{|y| > 1\}}\bigg(1-\frac{1}{|y|^\beta}\bigg)\,\frac{|f(y+x)|}{|y|^{n}}\,dy \\
&\le
\bigg|1-\frac{\mu_{n,\beta}}{\mu_{n,0}}\,\bigg|\, \|\nabla^0 f \|_{L^{\infty}(U;\, \R^n)}
+
\frac{n \omega_n \beta\mu_{n,\beta}}{\alpha(\alpha - \beta)}
\,
[f]_{C^{0,\alpha}(U_1)} \\
&\quad+
\mu_{n,\beta}\,
\|f\|_{L^p(\R^n)}\,\bigg ( \int_{\{|y| >1\}}\bigg(1-\frac{1}{|y|^{\beta}}\bigg)^{q} \frac{1}{|y|^{nq}}\,dy \bigg )^{\frac{1}{q}}, 
\end{align*}
where $q = \frac{p}{p - 1}$ and $U_1 := \{ y \in \R^n : {\rm dist}(y, U) < 1\}$. Since $y\mapsto|y|^{-nq} \in L^1(\R^n \setminus B_1)$ for all $q \in (1,+\infty)$, also the last term vanishes as $\beta \to 0^+$ thanks to Lebesgue's Dominated Convergence Theorem, so that the limit in~\eqref{eq:pointwise_conv_nabla_beta} follows.
If $p = 1$, then we can estimate the last term in the above inequality as 
\begin{equation*}
\int_{\{|y| > 1\}}\bigg(1-\frac{1}{|y|^\beta}\bigg)\,\frac{|f(y+x)|}{|y|^{n}}\,dy 
\le 
\|f\|_{L^1(\R^n)}\,
\sup_{|y| > 1} \frac{1}{|y|^{n}}\, \bigg(1-\frac{1}{|y|^\beta}\bigg).
\end{equation*}
Since 
\begin{equation*}
\sup_{|y| > 1} \frac{1}{|y|^{n}}\, \bigg(1-\frac{1}{|y|^\beta}\bigg) 
= 
\frac{\beta}{n \left ( 1 + \frac{\beta}{n} \right )^{\frac{n}{\beta} + 1}} 
\longto 0 
\quad
\text{as}\ \beta \to 0^+,
\end{equation*}
the limit in~\eqref{eq:pointwise_conv_nabla_beta} follows also in this case. 
Finally, if $f\in C^{0,\alpha}(\R^n)\cap L^p(\R^n)$ and $p<+\infty$, then the above estimates hold for $U = \R^n$, so that we obtain the uniform convergence $\nabla^{\beta} f \to \nabla^0 f$ in $\R^n$.
\end{proof}

\begin{remark} 
\label{rem:div_pointwise_conv_Riesz}
It is easy to see that a result analogous to \cref{res:pointwise_conv_Riesz} can be proved for the fractional divergence operator.
In particular, if $\phi \in C^{0,\alpha}(\R^n; \R^n)\cap L^p(\R^n; \R^n)$ for some $\alpha \in (0, 1]$ and $p \in [1,+\infty]$, then $\div^{\beta} \varphi \in L^{\infty}(\R^n)$ for all $\beta \in (0, \alpha)$ with
\begin{equation*}
\|\div^\beta \varphi\|_{L^\infty(\R^n)}
\le
c_{n,p}\, \mu_{n,\beta}\, \frac{\alpha p + n}{(\alpha - \beta)(\beta p + n)}\,  \left ( \tfrac{n}{p} + \beta \right )^{\frac{\alpha - \beta}{\alpha p + n}} \,\|\varphi\|_{L^p(\R^n; \R^n)}^{\frac{p(\alpha - \beta)}{\alpha p + n}} \,[\varphi]_{C^{0,\alpha}(\R^n;\R^n)}^{\frac{\beta p + n}{\alpha p + n}},
\end{equation*} 
where $c_{n,p}>0$ is the constant defined in~\eqref{eq:baramba_constant}.
If $p<+\infty$, then $\div^{\beta} \varphi \in L^{\infty}(\R^n)$ for all $\beta \in [0, \alpha)$, the above estimate holds also for $\beta = 0$ and we have
\begin{equation*}
\lim_{\beta\to0^+} \|\div^\beta \varphi - \div^0 \varphi\|_{L^{\infty}(\R^n)} = 0.
\end{equation*}
\end{remark}

As an immediate consequence of \cref{res:pointwise_conv_Riesz} and \cref{rem:div_pointwise_conv_Riesz}, we can show that the fractional $\alpha$-variation is lower semicontinuous as~$\alpha\to0^+$.

\begin{corollary}[Lower semicontinuity of $BV^\alpha$-seminorm as~$\alpha\to0^+$]
If $f\in L^1(\R^n)$, then for all open sets $U \subset \R^n$ it holds
\begin{equation}\label{eq:liminf_alpha_0}
|D^0 f|(U)\le\liminf_{\alpha\to0^+}|D^\alpha f|(U).
\end{equation}
\end{corollary}

\begin{proof}
Given $\phi\in C^\infty_c(U;\R^n)$ with $\|\phi\|_{L^\infty(U;\,\R^n)}\le1$, thanks to \cref{res:pointwise_conv_Riesz} and \cref{rem:div_pointwise_conv_Riesz} we have
\begin{align*}
\int_{\R^n}f\,\div^0\phi\,dx
=\lim_{\alpha\to0^+}\int_{\R^n}f\,\div^\alpha\phi\,dx
\le\liminf_{\alpha\to0^+}|D^\alpha f|(U),
\end{align*}
so that~\eqref{eq:liminf_alpha_0} follows by~\eqref{eq:fractional_variation_0}. 
\end{proof}

\subsection{Strong and energy convergence of \texorpdfstring{$\nabla^\alpha$}{nablaˆalpha} as \texorpdfstring{$\alpha\to0^+$}{alpha tends to 0ˆ+}}

We now study the strong and the energy convergence of~$\nabla^\alpha$ as~$\alpha\to0^+$.
For the strong convergence, we have the following result.

\begin{theorem}[Strong convergence of~$\nabla^\alpha$ as~$\alpha\to0^+$]
\label{res:strong_conv_alpha_0}
The following hold.
\begin{enumerate}[(i)]

\item\label{item:strong_conv_Hardy_alpha_0}
 If $f\in\bigcup_{\alpha\in(0,1)} HS^{\alpha,1}(\R^n)$, then
\begin{equation}\label{eq:strong_conv_Hardy_alpha_0}
\lim_{\alpha\to0^+}\|\nabla^\alpha f-Rf\|_{H^1(\R^n;\,\R^n)}=0.
\end{equation}

\item\label{item:strong_conv_alpha_0_Lp}
If $p\in(1,+\infty)$ and $f\in\bigcup_{\alpha\in(0,1)}S^{\alpha,p}(\R^n)$, then
\begin{equation}\label{eq:strong_conv_alpha_0_Lp}
\lim_{\alpha\to0^+}\|\nabla^\alpha f-Rf\|_{L^p(\R^n;\,\R^n)}=0.
\end{equation}

\end{enumerate}
\end{theorem}

\begin{remark}\label{rm:z}
Thanks to \cref{res:scatole}, \cref{res:strong_conv_alpha_0}\eqref{item:strong_conv_Hardy_alpha_0} can be equivalently stated as
\begin{equation}
\lim_{\alpha\to0^+}\|\nabla^\alpha f-Rf\|_{H^1(\R^n;\,\R^n)}=0
\end{equation}
for all $f\in H^1(\R^n)\cap\bigcup_{\alpha\in(0,1)} W^{\alpha,1}(\R^n)$.
\end{remark}

\noindent
We prove \cref{res:strong_conv_alpha_0} in \cref{subsec:proof_strong_conv_alpha_0}. For the convergence of the (rescaled) energy, we instead have the following result. 

\begin{theorem}[Energy convergence of~$\nabla^\alpha$ as~$\alpha\to0^+$]
\label{res:energy_conv_alpha_0}
If $f\in\bigcup_{\alpha\in(0,1)}W^{\alpha,1}(\R^n)$, then
\begin{equation*}
\lim_{\alpha\to0^+}\alpha\int_{\R^n}|\nabla^\alpha f|\,dx
=n\omega_n\mu_{n,0}\,\bigg|\int_{\R^n} f\,dx\,\bigg|.
\end{equation*}
\end{theorem}

\noindent
We prove \cref{res:energy_conv_alpha_0} in \cref{subsec:energy_conv_alpha_0}.

\subsection{Proof of \texorpdfstring{\cref{res:strong_conv_alpha_0}}{4.3}}
\label{subsec:proof_strong_conv_alpha_0}

Before the proof of \cref{res:strong_conv_alpha_0}, we need to recall the following well-known result, see the first part of the proof of~\cite{FS82}*{Lemma~1.60}. For the reader's convenience and to keep the paper as self-contained as possible, we briefly recall its simple proof. 

\begin{lemma}\label{res:derivatives_trick}
Let $m\in\N_0$. If $f\in\mathcal{S}_m(\R^n)$, then $f=\div g$ for some $g\in\mathcal{S}_{m-1}(\R^n;\R^n)$ (with $g\in\mathcal{S}(\R^n;\R^n)$ in the case $m=0$).
\end{lemma}

\begin{proof}
By means of the Fourier transform, the problem can be equivalently restated as follows: if $\phi\in\mathcal{S}(\R^n)$ satisfies $\de^\mathsf{a}\phi(0)=0$ for all $\mathsf{a}\in\N^n_0$ such that $|\mathsf{a}|\le m$, then $\phi(\xi)=\sum_1^n\xi_i\psi_i(\xi)$ for some $\psi_1,\dots,\psi_n\in\mathcal{S}(\R^n)$ with $\de^\mathsf{a}\psi_i(0)=0$ for all $i=1,\dots,n$ and all $\mathsf{a}\in\N^n_0$ such that $|\mathsf{a}|\le m-1$. This can be achieved as follows. Fixed any $\zeta\in C^\infty_c(\R^n)$ such that
\begin{equation*}
\supp\zeta\subset B_2
\quad\text{and}\quad
\zeta \equiv 1\ \text{on}\ B_1,
\end{equation*}
we can define
\begin{equation*}
\psi_i(\xi):=\zeta(\xi)\int_0^1\de_i\phi(t\xi)\,dt
+\frac{1-\zeta(\xi)}{|\xi|^2}\,\xi_i\,\phi(\xi),
\quad \xi\in\R^n,
\end{equation*}
for all $i=1,\dots,n$. It is now easy to prove that such $\psi_i$'s satisfy the required properties and we leave the simple calculations to the reader. 
\end{proof}

Thanks to \cref{res:derivatives_trick}, we can prove the following $L^p$-convergence result of the fractional $\alpha$-Laplacian of suitably regular functions as $\alpha\to0^+$, as well as analogous convergence results for the fractional $\alpha$-gradient.

\begin{lemma}
\label{res:frac_Laplacian_strong_convergence}
Let $p\in [1,+\infty]$. If $f\in\mathcal S_0(\R^n)$, then
\begin{equation}\label{eq:frac_Laplacian_strong_convergence}
\lim_{\alpha\to0^+}
\|(-\Delta)^{\frac\alpha2}f-f\|_{L^p(\R^n)}
=0.
\end{equation}
As a consequence, if $p\in(1,+\infty)$ and $f\in\mathcal S_0(\R^n)$, then
\begin{equation}\label{eq:nabla_alpha_strong_convergence_S}
\lim_{\alpha\to0^+}
\|\nabla^\alpha f-Rf\|_{L^p(\R^n;\,\R^n)}
=0;
\end{equation}
if $p =1$ and $f\in\mathcal S_{\infty}(\R^n)$, then
\begin{equation}\label{eq:nabla_alpha_strong_convergence_S_1}
\lim_{\alpha\to0^+}
\|\nabla^\alpha f-Rf\|_{H^1(\R^n;\,\R^n)}
=0.
\end{equation}
\end{lemma}

\begin{proof}
Let $f\in\mathcal S_0(\R^n)$ be fixed. If $p\in(1,+\infty)$, then
\begin{equation*}
\|\nabla^\alpha f-Rf\|_{L^p(\R^n;\,\R^n)}
=
\|R(-\Delta)^{\frac\alpha2} f-Rf\|_{L^p(\R^n;\,\R^n)}
\le
c_{n,p}
\|(-\Delta)^{\frac\alpha2}f-f\|_{L^p(\R^n)}
\end{equation*}
by the $L^p$-continuity of the Riesz transform, so that~\eqref{eq:nabla_alpha_strong_convergence_S} is a consequence of~\eqref{eq:frac_Laplacian_strong_convergence}. 
To prove~\eqref{eq:frac_Laplacian_strong_convergence}, given $x\in\R^n$ we write
\begin{align*}
(-\Delta)^{\frac\alpha2}f(x)
=\nu_{n,\alpha}
\int_{\set*{|h|>1}}\frac{f(x+h)-f(x)}{|h|^{n+\alpha}}\,dh
+
\nu_{n,\alpha}
\int_{\set*{|h|\le1}}\frac{f(x+h)-f(x)}{|h|^{n+\alpha}}\,dh,
\end{align*}
where
\begin{equation*}
\nu_{n,\alpha}=2^\alpha\pi^{-\frac n2}\frac{\Gamma\left(\frac{n+\alpha}{2}\right)}{\Gamma\left(-\frac{\alpha}{2}\right)},
\quad
\alpha\in(0,1),
\end{equation*}
is the constant appearing in~\eqref{eq:def_frac_Laplacian}.  
One easily sees that
\begin{equation}\label{eq:nu_n_alpha_limit}
\lim_{\alpha\to0^+}\frac{\nu_{n,\alpha}}{\alpha}=-\frac{1}{n\omega_n}.
\end{equation}
On the one hand, we can estimate 
\begin{equation*}
\left\|
\,
\nu_{n,\alpha}
\int_{\set*{|h| \le 1}}\frac{f(\cdot+h)-f(\cdot)}{|h|^{n+\alpha}}\,dh
\,
\right\|_{L^p(\R^n)}
\le
\frac{n\omega_n\nu_{n,\alpha}}{1-\alpha}
\,\|\nabla f\|_{L^p(\R^n;\,\R^n)}
\end{equation*}
(by the Fundamental Theorem of Calculus, see~\cite{Brezis11}*{Proposition 9.3(iii)} for instance), so that
\begin{equation*}
\lim_{\alpha\to0^+}
\left\|
\,
\nu_{n,\alpha}
\int_{\set*{|h| \le 1}}\frac{f(\cdot+h)-f(\cdot)}{|h|^{n+\alpha}}\,dh
\,
\right\|_{L^p(\R^n)}
=0
\end{equation*}
by~\eqref{eq:nu_n_alpha_limit} for all $p\in[1,+\infty]$. On the other hand, by \cref{res:derivatives_trick} there exists $g\in\mathcal S(\R^n;\R^n)$ such that $f=\div g$ and thus we can write
\begin{align*}
\nu_{n,\alpha}
\int_{\set*{|h| > 1}}\frac{f(x+h)-f(x)}{|h|^{n+\alpha}}\,dh
&=
\nu_{n,\alpha}
\int_{\set*{|h| > 1}}\frac{f(x+h)}{|h|^{n+\alpha}}\,dh
-
\frac{n\omega_n\nu_{n,\alpha}}{\alpha}\,
f(x)
\\
&=
\nu_{n,\alpha}
\int_{\set*{|h| > 1}}\frac{\div g(x+h)}{|h|^{n+\alpha}}\,dh
-
\frac{n\omega_n\nu_{n,\alpha}}{\alpha}\,
f(x).
\end{align*}
Integrating by parts, the reader can easily verify that
\begin{equation*}
\lim_{\alpha\to0^+}
\left\|
\,
\nu_{n,\alpha}
\int_{\set*{|h| > 1}}\frac{\div g(\cdot+h)}{|h|^{n+\alpha}}\,dh
\,
\right\|_{L^p(\R^n)}
=0
\end{equation*}
for all $p\in[1,+\infty]$. Hence we get
\begin{equation*}
\lim_{\alpha\to0^+}
\|(-\Delta)^{\frac\alpha2}f-f\|_{L^p(\R^n)}
=
\|f\|_{L^p(\R^n)}
\lim_{\alpha\to0^+}
\left |1+\frac{n\omega_n\nu_{n,\alpha}}{\alpha}\right |=0
\end{equation*}
for all $p\in[1,+\infty]$, so that we obtain \eqref{eq:frac_Laplacian_strong_convergence} and \eqref{eq:nabla_alpha_strong_convergence_S}. Finally, let $f \in \mathcal{S}_{\infty}(\R^n)$, so that $R f \in \mathcal{S}_0(\R^n; \R^n)$, $R(R f) \in \mathcal{S}_0(\R^n; \R^{n^2})$ and $(-\Delta)^{\frac\alpha2} Rf = \nabla^{\alpha} f$. Then, we have 
\begin{align*}
\|\nabla^\alpha f-Rf\|_{H^1(\R^n;\,\R^n)}
= \|(-\Delta)^{\frac\alpha2}Rf - Rf\|_{L^1(\R^n;\,\R^n)} +  \|(-\Delta)^{\frac\alpha2}R(Rf) - R(Rf)\|_{L^1(\R^n;\,\R^{n^2})} 
\end{align*}
and thus
\begin{equation*}
\lim_{\alpha\to0^+}
\|\nabla^\alpha f-Rf\|_{H^1(\R^n;\,\R^n)}
=0
\end{equation*}
thanks \eqref{eq:frac_Laplacian_strong_convergence} (which clearly holds also for vector-valued functions). Thus, we obtain \eqref{eq:nabla_alpha_strong_convergence_S_1}, and the proof is complete.
\end{proof}

We can now prove \cref{res:strong_conv_alpha_0}.

\begin{proof}[Proof of \cref{res:strong_conv_alpha_0}]
We prove the two statements separately.

\smallskip

\textit{Proof of~\eqref{item:strong_conv_Hardy_alpha_0}}.
Let $f\in HS^{\alpha,1}(\R^n)$. By \cref{res:approx_H_1_alpha}, there exists $(f_k)_{k\in\N}\subset\mathcal S_{\infty}(\R^n)$ such that $f_k\to f$ in~$HS^{\alpha,1}(\R^n)$ as $k\to+\infty$. 
If $\beta\in(0,\alpha)$, then we can estimate
\begin{align*}
\|\nabla^\beta f&-Rf\|_{H^1(\R^n;\,\R^n)}
\le
\|\nabla^\beta f_k-Rf_k\|_{H^1(\R^n;\,\R^n)}
+
\|\nabla^\beta f-\nabla^\beta f_k\|_{H^1(\R^n;\,\R^n)}\\
&\quad+
\|R f-R f_k\|_{H^1(\R^n;\,\R^n)}\\
&\le
\|\nabla^\beta f_k-Rf_k\|_{H^1(\R^n;\,\R^n)}
+
c_n\|f-f_k\|_{H^1(\R^n)}^{\frac{\alpha-\beta}\alpha}
\,
\|\nabla^\alpha f-\nabla^\alpha f_k\|_{H^1(\R^n;\,\R^n)}^{\frac\beta\alpha}\\
&\quad+
c_n'\|f-f_k\|_{H^1(\R^n)}
\end{align*} 
for all $k\in\N$ by~\eqref{eq:interpolation_MH_1_gamma=0} in \cref{res:interpolation_MH}\eqref{item:interpolation_MH_1} and the $H^1$-continuity of the Riesz transform, where $c_n,c_n'>0$ are dimensional constants. Thus
\begin{align*}
\limsup_{\beta\to0^+}
\|\nabla^\beta f-Rf\|_{H^1(\R^n;\,\R^n)}
&\le
\limsup_{\beta\to0^+}\|\nabla^\beta f_k-Rf_k\|_{H^1(\R^n;\,\R^n)}
+
c_n''\|f-f_k\|_{H^1(\R^n)}\\
&=
c_n''\|f-f_k\|_{H^1(\R^n)}
\end{align*}
for all $k\in\N$ by~\eqref{eq:nabla_alpha_strong_convergence_S_1} in \cref{res:frac_Laplacian_strong_convergence}, where $c_n''=c_n+c_n'$. Hence~\eqref{eq:strong_conv_Hardy_alpha_0} follows by passing to the limit as $k\to+\infty$ and the proof of~\eqref{item:strong_conv_Hardy_alpha_0} is complete.

\smallskip

\textit{Proof of~\eqref{item:strong_conv_alpha_0_Lp}}.
We argue as in the proof of \eqref{item:strong_conv_Hardy_alpha_0}. Let $f\in S^{\alpha,p}(\R^n)$. By \cref{res:S_0_dense_in_S_p_alpha}, there exists $(f_k)_{k\in\N}\subset\mathcal S_0(\R^n)$ such that $f_k\to f$ in~$S^{\alpha,p}(\R^n)$ as $k\to+\infty$. 
If $\beta\in(0,\alpha)$, then we can estimate
\begin{align*}
\|\nabla^\beta f&-Rf\|_{L^p(\R^n;\,\R^n)}
\le
\|\nabla^\beta f_k-Rf_k\|_{L^p(\R^n;\,\R^n)}
+
\|\nabla^\beta f-\nabla^\beta f_k\|_{L^p(\R^n;\,\R^n)}\\
&\quad+
\|R f-R f_k\|_{L^p(\R^n;\,\R^n)}\\
&\le
\|\nabla^\beta f_k-Rf_k\|_{L^p(\R^n;\,\R^n)}
+
c_{n,p}\|f-f_k\|_{L^p(\R^n)}^{\frac{\alpha-\beta}\alpha}
\,
\|\nabla^\alpha f-\nabla^\alpha f_k\|_{L^p(\R^n;\,\R^n)}^{\frac\beta\alpha}\\
&\quad+
c_{n,p}'\|f-f_k\|_{L^p(\R^n)}
\end{align*} 
for all $k\in\N$ by~\eqref{eq:interpolation_MH_p_gamma=0} in \cref{res:interpolation_MH}\eqref{item:interpolation_MH_p} and the $L^p$-continuity of the Riesz transform, where the constants $c_{n,p},c_{n,p}'>0$ depend only on~$n$ and~$p$. 
Thus
\begin{align*}
\limsup_{\beta\to0^+}
\|\nabla^\beta f-Rf\|_{L^p(\R^n;\,\R^n)}
&\le
\limsup_{\beta\to0^+}\|\nabla^\beta f_k-Rf_k\|_{L^p(\R^n;\,\R^n)}
+
c_{n,p}''\|f-f_k\|_{L^p(\R^n)}\\
&=
c_{n,p}''\|f-f_k\|_{L^p(\R^n)}
\end{align*}
for all $k\in\N$ by~\eqref{eq:nabla_alpha_strong_convergence_S} in \cref{res:frac_Laplacian_strong_convergence}, where $c_{n,p}''=c_{n,p}+c_{n,p}'$. Hence~\eqref{eq:strong_conv_alpha_0_Lp} follows by passing to the limit as $k\to+\infty$ and the proof of~\eqref{item:strong_conv_alpha_0_Lp} is complete.
\end{proof}

\begin{remark}[Direct proof of~\eqref{intro_eq:frac_limit_1}]\label{rm:direct_proof_intro_eq:frac_limit_1}
The proof of \eqref{intro_eq:frac_limit_1}, i.e.,
\begin{equation*}
\lim_{\alpha\to0^+}\|\nabla^\alpha f-Rf\|_{L^1(\R^n;\,\R^n)}=0
\quad
\text{for all}\
f\in H^1(\R^n)\cap\bigcup_{\alpha\in(0,1)}W^{\alpha,1}(\R^n),
\end{equation*}
immediately follows from \cref{res:strong_conv_alpha_0}\eqref{item:strong_conv_Hardy_alpha_0} and \cref{rm:z}.
As briefly discussed in \cref{subsec:intro_frac_interp}, one can directly prove~\eqref{intro_eq:frac_limit_1} by combining the interpolation inequality proven in \cref{res:interpolation_H1_BV_alpha} with an approximation argument as done in the proof of \cref{res:strong_conv_alpha_0}.
We let the interested reader fill the easy details.
\end{remark}

\subsection{Proof of \texorpdfstring{\cref{res:energy_conv_alpha_0}}{4.4}}
\label{subsec:energy_conv_alpha_0}

We now pass to the proof of \cref{res:energy_conv_alpha_0}. We need some preliminaries. 
We begin with the following result.

\begin{lemma}\label{res:unif_lim_alpha_to_0}
Let $f\in L^1(\R^n)$ and let $R\in(0,+\infty)$ be such that $\supp f\subset B_R$. If $\eps>R$, then
\begin{equation*}
\lim_{\alpha\to0^+}\alpha\mu_{n,\alpha}\int_{\R^n}\bigg|\int_{\{|y|>\eps\}}\frac{y f(y+x)}{|y|^{n+\alpha+1}}\,dy\,\bigg|\,dx
=n\omega_n\mu_{n,0}\bigg|\int_{\R^n} f\,dx\,\bigg|.
\end{equation*}
\end{lemma}

\begin{proof}
Since $\mu_{n,\alpha}\to\mu_{n,0}$ as $\alpha\to0^+$, we just need to prove that
\begin{equation}\label{eq:target}
\lim_{\alpha\to0^+}\alpha\int_{\R^n}\bigg|\int_{\{|y|>\eps\}}\frac{y f(y+x)}{|y|^{n+\alpha+1}}\,dy\,\bigg|\,dx
=n\omega_n\bigg|\int_{\R^n} f\,dx\,\bigg|.
\end{equation}
We now divide the proof in two steps.

\smallskip

\textit{Step~1}. 
We claim that
\begin{equation}\label{eq:claim_change_x}
\lim_{\alpha\to0^+}\alpha\int_{\R^n}\bigg|\int_{\{|y|>\eps\}}\frac{x f(y+x)}{|x|^{n+\alpha+1}}\,dy\,\bigg|\,dx
=n\omega_n\bigg|\int_{\R^n} f\,dx\,\bigg|.
\end{equation}
Indeed, since $\supp f\subset B_R$, we have that
\begin{equation*}
\int_{\{|y|>\eps\}}\frac{x f(y+x)}{|x|^{n+\alpha+1}}\,dy=0
\quad 
\text{for all $x\in\R^n$ such that}\  |x+y|\ge R\ \text{for all}\ |y| > \eps.
\end{equation*}
Recalling that $\eps > R$, we see that, for all $|y| > \eps$,  
\begin{equation}\label{eq:balls_domain}
|x|\le\eps-R\implies |x+y|\ge R	
\end{equation}
and thus we can write
\begin{align*}
\alpha\int_{\R^n}\bigg|\int_{\{|y|>\eps\}}\frac{x f(y+x)}{|x|^{n+\alpha+1}}\,dy\,\bigg|\,dx
&=\alpha\int_{\{|x|>\eps-R\}}\bigg|\int_{\{|y|>\eps\}}\frac{x f(y+x)}{|x|^{n+\alpha+1}}\,dy\,\bigg|\,dx\\
&=\alpha\int_{\{|x|>\eps-R\}}
\frac{1}{|x|^{n+\alpha}}\,
\bigg|\int_{\{|y|>\eps\}}f(y+x)\,dy\,\bigg|\,dx.
\end{align*}
Now, on the one hand, we have
\begin{equation}\label{eq:claim_change_x_proof_1}
\alpha\int_{\{\eps-R<|x|\le\eps+R\}}\frac{1}{|x|^{n+\alpha}}\,\abs*{\int_{\{|y|>\eps\}}f(y+x)\,dy\,}\,dx
\le\alpha n\omega_n\|f\|_{L^1(\R^n)}\int_{\eps-R}^{\eps+R}\frac{dr}{r^{\alpha+1}}
\end{equation}
for all $\alpha\in(0,1)$. On the other hand, since
\begin{equation*}
|x|>\eps+R
\implies
B_R\subset B_\eps(x)^c,
\end{equation*}
we have
\begin{equation}\label{eq:claim_change_x_proof_2}
\begin{split}
\alpha\int_{\{|x|>\eps+R\}}\frac{1}{|x|^{n+\alpha}}\,\bigg|\int_{\{|y|>\eps\}}f(y+x)\,dy\,\bigg|\,dx
&=\alpha\int_{\{|x|>\eps+R\}}\frac{1}{|x|^{n+\alpha}}\,\bigg|\int_{\R^n}f\,dz\,\bigg|\,dx\\
&=\frac{n\omega_n}{(\eps+R)^\alpha}\,\bigg|\int_{\R^n}f\,dz\,\bigg|
\end{split}
\end{equation}
for all $\alpha\in(0,1)$. Hence, claim~\eqref{eq:claim_change_x} follows by first combining~\eqref{eq:claim_change_x_proof_1} and~\eqref{eq:claim_change_x_proof_2} and then passing to the limit as~$\alpha\to0^+$.

\smallskip

\textit{Step~2}. We claim that
\begin{equation}\label{eq:claim_elia}
\bigg|\frac{y}{|y|^{n+\alpha+1}}+\frac{x}{|x|^{n+\alpha+1}}\bigg|
\le
(n+3)\,\frac{|x+y|}{|y|^{n+\alpha+1}}\bigg(\frac{\eps}{\eps-R}\bigg)^{n+\alpha+1}
\end{equation}
for all $x,y\in\R^n$ such that $|x|>\eps-R$, $|y|>\eps$ and $|y+x|<R$. Indeed, setting $F(z):=\frac{z}{|z|^{n+\alpha+1}}$ for all $z\in\R^n\setminus\{0\}$, we can estimate
\begin{align*}
\abs*{\frac{y}{|y|^{n+\alpha+1}}+\frac{x}{|x|^{n+\alpha+1}}}
&=|F(y)-F(-x)|
\le|y+x|\sup_{t\in[0,1]}|\nabla F|((1-t)y-tx)\\
&\le (n+\alpha+2)\,|y+x|\,\sup_{t\in[0,1]}\frac{1}{|(1-t)y-tx|^{n+\alpha+1}}.
\end{align*}
Since
\begin{align*}
\frac{1}{|(1-t)y-tx|^{n+\alpha+1}}
&\le\frac{1}{||y|-t|y+x||^{n+\alpha+1}}\\
&\le\frac{1}{(|y|-R)^{n+\alpha+1}}\\
&\le\frac{1}{|y|^{n+\alpha+1}}\bigg(\frac{|y|}{|y|-R}\bigg)^{n+\alpha+1}\\
&\le\frac{1}{|y|^{n+\alpha+1}}\bigg(\frac{\eps}{\eps-R}\bigg)^{n+\alpha+1}
\end{align*}
for all $t\in[0,1]$, claim~\eqref{eq:claim_elia} immediately follows. Now, recalling~\eqref{eq:balls_domain}, we can estimate
\begin{align*}
\bigg|\alpha & \int_{\R^n}\abs*{\int_{\{|y|>\eps\}}\frac{y f(y+x)}{|y|^{n+\alpha+1}}\,dy\,}\,dx
-
\alpha\int_{\R^n}\bigg|\int_{\{|y|>\eps\}}\frac{x f(y+x)}{|x|^{n+\alpha+1}}\,dy\,\bigg|\,dx\bigg|\\
&\le
\alpha\int_{\R^n}\int_{\{|y|>\eps\}}|f(y+x)|\,\bigg|\frac{y}{|y|^{n+\alpha+1}}+\frac{x}{|x|^{n+\alpha+1}}\bigg|\,dy\,dx\\
&=\alpha\int_{\{|x|>\eps-R\}}\int_{\{|y|>\eps\}}|f(y+x)|\,\abs*{\frac{y}{|y|^{n+\alpha+1}}+\frac{x}{|x|^{n+\alpha+1}}}\,dy\,dx\\
&\le\alpha(n+3)
\bigg(\frac{\eps}{\eps-R}\bigg)^{n+\alpha+1}\int_{\{|x|>\eps-R\}}\int_{\{|y|>\eps\}}|f(y+x)|\,\frac{|y+x|}{|y|^{n+\alpha+1}}\,dy\,dx
\end{align*}
for all $\alpha\in(0,1)$ thanks to~\eqref{eq:claim_elia}. Since
\begin{align*}
\alpha\int_{\{|y|>\eps\}}\frac{1}{|y|^{n+\alpha+1}}\int_{\{|x|>\eps-R\}}|f(y+x)|\,|y+x|\,dx\,dy
\le\alpha n\omega_n R\,\|f\|_{L^1(\R^n)}\int_{\eps}^{\infty}\frac{dr}{r^{\alpha+2}},
\end{align*}
we conclude that
\begin{equation}\label{eq:change_limit}
\limsup_{\alpha\to0^+}\bigg|\,\alpha\int_{\R^n}\bigg|\int_{\{|y|>\eps\}}\frac{y f(y+x)}{|y|^{n+\alpha+1}}\,dy\,\bigg|\,dx
-\alpha\int_{\R^n}\bigg|\int_{\{|y|>\eps\}}\frac{x f(y+x)}{|x|^{n+\alpha+1}}\,dy\,\bigg|\,dx\,\bigg|=0.
\end{equation}
Thus~\eqref{eq:target} follows by combining~\eqref{eq:claim_change_x} with~\eqref{eq:change_limit} and the proof is complete.
\end{proof}

Thanks to \cref{res:unif_lim_alpha_to_0}, we can prove the following result.

\begin{lemma}\label{res:eta-eps_lim_alpha_to_0}
Let $f\in L^1(\R^n)$ and $\eta>0$. There exists $\eps>0$ such that
\begin{equation*}
\limsup_{\alpha\to0^+}\,\bigg|\,\alpha\mu_{n,\alpha}\int_{\R^n}\bigg|\int_{\{|y|>\eps\}}\frac{y f(y+x)}{|y|^{n+\alpha+1}}\,dy\,\bigg|\,dx
-n\omega_n\mu_{n,0}\Big|\int_{\R^n} f\,dx\,\Big|\bigg|<\eta.
\end{equation*}
\end{lemma}

\begin{proof}
Let $\eta'>0$ be such that $\eta=2n\omega_n\mu_{n,0}\,\eta'$. Since $f\in L^1(\R^n)$, we can find $R>0$ such that $\int_{B_R^c}|f|\,dx<\eta'$. Let $\eps>R$ and $g:=f\chi_{B_R}$, which satisfies $g \in L^1(\R^n)$ and $\supp(g)\subset\closure[1]{B_R}$.
Then
\begin{align*}
\bigg|\int_{\R^n}&\bigg|\int_{\{|y|>\eps\}}\frac{y \, f(y+x)}{|y|^{n+\alpha+1}}\,dy\,\bigg|\,dx
-\int_{\R^n}\abs*{\int_{\{|y|>\eps\}}\frac{y \, g(y+x)}{|y|^{n+\alpha+1}}\,dy\,}\,dx\,\bigg|\\
&\le\int_{\{|y|>\eps\}}\frac{1}{|y|^{n+\alpha}}\int_{\R^n}|f(y+x)-g(y+x)|\,dx\,dy\\
&=\frac{n\omega_n\|f-g\|_{L^1(\R^n)}}{\alpha\eps^\alpha}
<\frac{n\omega_n}{\alpha\eps^\alpha}\,\eta'.
\end{align*} 
Since clearly
\begin{equation*}
\bigg|\Big|\int_{\R^n} f\,dx\,\Big|-\Big|\int_{\R^n} g\,dx\,\Big|\bigg|
\le\|f-g\|_{L^1(\R^n)}<\eta',
\end{equation*}
by \cref{res:unif_lim_alpha_to_0} we conclude that
\begin{align*}
\limsup_{\alpha\to0^+}\,&\bigg|\alpha\,\mu_{n,\alpha}\int_{\R^n}\bigg|\int_{\{|y|>\eps\}}\frac{y \, f(y+x)}{|y|^{n+\alpha+1}}\,dy\,\bigg|\,dx
-n\omega_n\,\mu_{n,0}\,\abs*{\int_{\R^n} f\,dx\,}\bigg|\\
&<\limsup_{\alpha\to0^+}\,\abs*{\alpha\,\mu_{n,\alpha}\int_{\R^n}\abs*{\int_{\{|y|>\eps\}}\frac{y \, g(y+x)}{|y|^{n+\alpha+1}}\,dy\,}\,dx
-n\omega_n\,\mu_{n,0}\,\abs*{\int_{\R^n} g\,dx\,}}\\
&\quad+\left(n\omega_n\mu_{n,0}+n\omega_n\lim_{\alpha\to0^+}\mu_{n,\alpha}\eps^{-\alpha}\right)\eta'\\
&=2n\omega_n\mu_{n,0}\,\eta'=\eta
\end{align*}
and the proof is complete.
\end{proof}

We are now ready to prove \cref{res:energy_conv_alpha_0}.

\begin{proof}[Proof of \cref{res:energy_conv_alpha_0}]
Assume $f\in W^{\beta,1}(\R^n)$ for some $\beta\in(0,1)$ and fix $\eta>0$. By \cref{res:eta-eps_lim_alpha_to_0}, there exists $\eps>0$ such that
\begin{equation}\label{eq:find_eps}
\limsup_{\alpha\to0^+}\,\abs*{\alpha\,\mu_{n,\alpha}\int_{\R^n}\abs*{\int_{\{|y|>\eps\}}\frac{y \, f(y+x)}{|y|^{n+\alpha+1}}\,dy\,}\,dx
-n\omega_n\mu_{n,0}\abs*{\int_{\R^n} f\,dx\,}}<\eta.
\end{equation}
Since for all $\alpha\in(0,\beta)$ we can estimate
\begin{align*}
\bigg|\alpha &\int_{\R^n}|\nabla^\alpha f|\,dx
-n\omega_n\mu_{n,0}\abs*{\int_{\R^n} f\,dx\,}\bigg|\\
&\le\abs*{\alpha\,\mu_{n,\alpha}\int_{\R^n}\abs*{\int_{\{|y|>\eps\}}\frac{y \, f(y+x)}{|y|^{n+\alpha+1}}\,dy\,}\,dx
-n\omega_n\mu_{n,0}\abs*{\int_{\R^n} f\,dx\,}}\\
&\quad+\alpha\,\mu_{n,\alpha}\int_{\R^n}\int_{\{|y|\le\eps\}}\frac{|f(y+x)-f(x)|}{|y|^{n+\alpha}}\,dy\,dx\\
&\le\abs*{\alpha\,\mu_{n,\alpha}\int_{\R^n}\abs*{\int_{\{|y|>\eps\}}\frac{y \, f(y+x)}{|y|^{n+\alpha+1}}\,dy\,}\,dx
-n\omega_n\mu_{n,0}\abs*{\int_{\R^n} f\,dx\,}}
+\alpha\,\mu_{n,\alpha}\,\eps^{\beta-\alpha}[f]_{W^{\beta,1}(\R^n)},
\end{align*}
by~\eqref{eq:find_eps} we have
\begin{equation*}
\limsup_{\alpha\to0^+}\,\abs*{\alpha\int_{\R^n}|\nabla^\alpha f|\,dx
-n\omega_n\mu_{n,0}\abs*{\int_{\R^n} f\,dx\,}}
<\eta
\end{equation*}
and the conclusion follows passing to the limit as $\eta\to0^+$.
\end{proof}

\appendix

\section{\texorpdfstring{$C^\infty_c(\R^n)$}{Cˆinfty-c} is dense in \texorpdfstring{$S^{\alpha,p}(\R^n)$}{Sˆ{alpha,p}(Rˆn)}}
\label{sec:identificaiton_Bessel}

In this section, we prove \cref{res:L_p_alpha_equal_S_p_alpha} below. This result completely answers a question left open in~\cite{CS19}*{Section~3.9}.

\begin{theorem}[Approximation by $C^\infty_c$ functions in $S^{\alpha,p}$]
\label{res:L_p_alpha_equal_S_p_alpha}
Let $\alpha\in(0,1)$ and $p\in[1,+\infty)$. 
The set $C^\infty_c(\R^n)$ is dense in $S^{\alpha,p}(\R^n)$. 
\end{theorem}

For the proof of \cref{res:L_p_alpha_equal_S_p_alpha}, we need some preliminary results. We begin with the following integration-by-parts formula. 

\begin{lemma}\label{res:int_by_parts_nabla_0}
Let $p,q\in(1,+\infty)$ be such that $\frac1p+\frac1q=1$. If $f\in L^p(\R^n)$ and $\phi\in L^q(\R^n;\R^n)$, then
\begin{equation}\label{eq:int_by_parts_nabla_0}
\int_{\R^n} f\,\div^0\phi\,dx
=-\int_{\R^n}\phi\cdot\nabla^0 f\,dx.
\end{equation}
\end{lemma}

\begin{proof}
Integrating by parts and applying Fubini's Theorem, formula~\eqref{eq:int_by_parts_nabla_0} is easily proved for all $f\in C^\infty_c(\R^n)$ and $\phi\in C^\infty_c(\R^n;\R^n)$. Since the real-valued bilinear functionals 
\begin{equation*}
(f,\phi)\mapsto\int_{\R^n} f\,\div^0\phi\,dx,
\quad
(f,\phi)\mapsto\int_{\R^n} \phi\cdot \nabla^0 f\,dx,
\end{equation*} 
are both continuous on $L^p(\R^n)\times L^q(\R^n;\R^n)$ by H\"older's inequality and the $L^p$-continuity of Riesz transform, the conclusion follows by a simple approximation argument. 
\end{proof}

\begin{remark}
As an immediate consequence of Lemma \ref{res:int_by_parts_nabla_0} and the $L^p$-continuity of the Riesz transform, we can conclude that the space
\begin{equation*}
S^{0, p}(\R^n) := \set*{f \in L^{p}(\R^n) : \nabla^{0} f \in L^{p}(\R^n; \R^n)}
\end{equation*}
actually coincides with $L^{p}(\R^n)$ for all $p \in (1, +\infty)$, with $\nabla^0 f = Rf$. In addition, Theorem~\ref{res:H_1=BV_0} easily yields the identity $BV^0(\R^n) = S^{0,1}(\R^n) = H^1(\R^n)$. Arguing in an analogous fashion, we can see that, for all $p \in (1, + \infty)$, 
\begin{equation*}
BV^{0,p}(\R^n) :=\set*{f \in L^{p}(\R^n) : D^{0} f \in \mathscr{M}(\R^n; \R^n)}
\end{equation*}
coincides with the space
\begin{equation*}
\set*{f \in L^{p}(\R^n) : R f \in L^1(\R^n; \R^n)},
\end{equation*}
and we have $D^0 f = (R f) \Leb{n}$.
\end{remark}

Adopting the notation introduced in~\cite{S18}*{Equation~(1.9)}, for $\alpha\in(0,1)$ and $f\in \mathcal{S}(\R^n)$, let
\begin{equation*}
\mathcal D^\alpha f(x):=
\int_{\R^n}\frac{|f(y+x)-f(y)|}{|y|^{n+\alpha}}\,dy
\end{equation*}
for all $x\in\R^n$. 
Note that $|(-\Delta)^{\frac{\alpha}{2}}f(x)| \le |\nu_{n, \alpha}| \, \mathcal D^\alpha f(x)$ for all $\alpha\in(0,1)$, $f\in \mathcal{S}(\R^n)$ and $x \in \R^n$. 
In the following result, we prove that the operator $\mathcal D^\alpha$ naturally extends to a continuous operator from $W^{1,p}(\R^n)$ to $L^p(\R^n)$.

\begin{lemma}\label{res:frac_box_Lp}
Let $\alpha\in(0,1)$ and $p\in[1,+\infty]$. The operator $\mathcal D^\alpha\colon W^{1,p}(\R^n)\to L^p(\R^n)$ is well defined and satisfies
\begin{equation}\label{eq:frac_box_Lp}
\left\|\mathcal D^\alpha f\right\|_{L^p(\R^n)}
\le \frac{2^{1 - \alpha} n\omega_n}{\alpha(1-\alpha)}\, \|f\|_{L^p(\R^n)}^{1 - \alpha}
\|\nabla f\|_{L^p(\R^n;\,\R^n)}^{\alpha}
\end{equation}
for all $f\in W^{1,p}(\R^n)$.
\end{lemma}

\begin{proof}
Let $f\in W^{1,p}(\R^n)$ and $r>0$. 
We can estimate
\begin{align*}
\mathcal D^\alpha f(x)
\le
\bigg(\int_{\{|y|<r\}}\frac{|f(y+x)-f(x)|}{|y|^{n+\alpha}}\,dy
+\int_{\{|y|\ge r\}}\frac{|f(y+x)-f(x)|}{|y|^{n+\alpha}}\,dy\bigg)
\end{align*}
for a.e.\ $x\in\R^n$. By Minkowski's integral inequality and well-known properties of Sobolev functions (see \cite{L09}*{Lemma 11.11} in the case $p \in [1, + \infty)$), on the one hand we have
\begin{align*}
\bigg\|\int_{\{|y|<r\}}\frac{|f(y+\cdot)-f(\cdot)|}{|y|^{n+\alpha}}\,dy\,\bigg\|_{L^p(\R^n)}
&\le\int_{\{|y|<r\}}\frac{\|f(y+\cdot)-f(\cdot)\|_{L^p(\R^n)}}{|y|^{n+\alpha}}\,dy\\
&\le\|\nabla f\|_{L^p(\R^n;\R^n)}\int_{\{|y|<r\}}\frac{dy}{|y|^{n+\alpha-1}}\\
&=\frac{n\omega_n r^{1-\alpha}}{1-\alpha}\,\|\nabla f\|_{L^p(\R^n;\R^n)}
\end{align*}
while, on the other hand, we have
\begin{align*}
\bigg\|\int_{\{|y|\ge r\}}\frac{|f(y+\cdot)-f(\cdot)|}{|y|^{n+\alpha}}\,dy\,\bigg\|_{L^p(\R^n)}
&\le\int_{\{|y|<r\}}\frac{\|f(y+\cdot)\|_{L^p(\R^n)}+\|f\|_{L^p(\R^n)}}{|y|^{n+\alpha}}\,dy\\
&=2\|f\|_{L^p(\R^n)}\int_{\{|y|\ge r\}}\frac{dy}{|y|^{n+\alpha}}\\
&=\frac{2n\omega_n r^{-\alpha}}{\alpha}\,\|f\|_{L^p(\R^n)}.
\end{align*}
Hence
\begin{equation*}
\left\|\mathcal D^\alpha f\right\|_{L^p(\R^n)}
\le n\omega_n
\,\bigg(\frac{r^{1-\alpha}}{1-\alpha}\,\|\nabla f\|_{L^p(\R^n;\,\R^n)}
+2 \frac{r^{-\alpha}}{\alpha}\,\|f\|_{L^p(\R^n)}\bigg)
\end{equation*}
for all~$r>0$. Thus~\eqref{eq:frac_box_Lp} follows by choosing $r=\frac{2 \|f\|_{L^p(\R^n)}}{\|\nabla f\|_{L^p(\R^n;\,\R^n)}}$ and the proof is complete. 
\end{proof}

In the following result, we recall the self-adjointness property the fractional Laplacian.

\begin{lemma}\label{res:frac_laplacian_self_adjoint}
Let $\alpha\in(0,1)$ and $p,q\in[1,+\infty]$ such that $\frac1p+\frac1q=1$. If $f\in W^{1,p}(\R^n)$ and $g\in W^{1,q}(\R^n)$, then
\begin{equation}\label{eq:frac_laplacian_self_adjoint}
\int_{\R^n}f\,(-\Delta)^{\frac\alpha2}g\,dx
=\int_{\R^n}g\,(-\Delta)^{\frac\alpha2}f\,dx.
\end{equation}
\end{lemma}

\begin{proof}
Formula~\eqref{eq:frac_laplacian_self_adjoint} is well known for $f,g\in \mathcal S(\R^n)$ and can be proved by exploiting Functional Calculus or by directly using the definition of $(-\Delta)^{\frac\alpha2}$ for instance.
Since the real-valued functional
\begin{equation*}
(f,g)\mapsto\int_{\R^n}f\,(-\Delta)^{\frac\alpha2}g\,dx
\end{equation*} 
is bilinear and continuous on $L^p(\R^n)\times W^{1,q}(\R^n;\R^n)$ by H\"older's inequality and \cref{res:frac_box_Lp} above, the conclusion follows by a simple approximation argument for $p,q\in(1,+\infty)$. The case $p,q\in\{1,+\infty\}$ follows by Fubini's theorem, thanks to the fact that the function
\begin{equation*}
(x, y) \to f(x) \frac{g(x + y) - g(x)}{|y|^{n + \alpha}}
\end{equation*}
belongs to $L^1(\R^n \times \R^n)$ if $(f, g) \in L^1(\R^n) \times W^{1, \infty}(\R^n)$ or $(f, g) \in L^{\infty}(\R^n) \times W^{1, 1}(\R^n)$. The details are left to the reader.
\end{proof}

We are now ready to prove the main result of this section. 

\begin{proof}[Proof of \cref{res:L_p_alpha_equal_S_p_alpha}]
The density of $C^\infty_c(\R^n)$ in $S^{\alpha,1}(\R^n)$ was already proved in~\cite{CS19}*{Theorem~3.23}, so we can restrict our attention to the case~$p>1$ without loss of generality.  
We divide the proof in two steps.

\smallskip

\textit{Step~1}.
Let $f\in S^{\alpha,p}(\R^n)$ and assume $f\in W^{1,p}(\R^n)\cap\Lip_b(\R^n)$. Given $\phi\in C^\infty_c(\R^n;\R^n)$, we can write $\div^\alpha\phi=(-\Delta)^{\frac\alpha2}\div^0\phi$ with $\div^0\phi\in\Lip_b(\R^n)\cap W^{1,q}(\R^n)$, so that
\begin{equation*}
\int_{\R^n}f\,(-\Delta)^{\frac\alpha2}\div^0\phi\,dx
=\int_{\R^n}(-\Delta)^{\frac\alpha2}f\,\div^0\phi\,dx
\end{equation*}
for all $\phi\in C^\infty_c(\R^n;\R^n)$ by \cref{res:frac_laplacian_self_adjoint}. Since $(-\Delta)^{\frac\alpha2} f\in L^p(\R^n)$ thanks to \cref{res:frac_box_Lp}, by \cref{res:int_by_parts_nabla_0} we have
\begin{equation*}
\int_{\R^n}(-\Delta)^{\frac\alpha2}f\,\div^0\phi\,dx
=-\int_{\R^n}\phi\cdot\nabla^0(-\Delta)^{\frac\alpha2}f\,dx
\end{equation*}
for all $\phi\in C^\infty_c(\R^n;\R^n)$. We thus get that $\nabla^\alpha f=\nabla^0(-\Delta)^{\frac\alpha2}f$ for all $f\in S^{\alpha,p}(\R^n)\cap W^{1,p}(\R^n)\cap\Lip_b(\R^n)$, so that
\begin{equation*}
c_1\|(-\Delta)^{\frac\alpha2}f\|_{L^p(\R^n)}\le[f]_{S^{\alpha,p}(\R^n)}
\le c_2\|(-\Delta)^{\frac\alpha2}f\|_{L^p(\R^n)}
\end{equation*}
for all $f\in S^{\alpha,p}(\R^n)\cap W^{1,p}(\R^n)\cap\Lip_b(\R^n)$, where $c_1,c_2>0$ are two constants depending only on~$p>1$. Thus, recalling the equivalent definition of the space $L^{\alpha,p}(\R^n)$ given in~\eqref{eq:def2_Bessel_space}, we conclude that 
\begin{equation*}
S^{\alpha,p}(\R^n)\cap W^{1,p}(\R^n)\cap\Lip_b(\R^n)
\subset 
L^{\alpha,p}(\R^n)
\end{equation*}
with continuous embedding. 

\smallskip

\textit{Step~2}. Now fix $f\in S^{\alpha,p}(\R^n)$ and let $(\rho_\eps)_{\eps>0}\subset C^\infty_c(\R^n)$ be a family of standard mollifiers (see~\cite{CS19}*{Section~3.3} for a definition). Setting $f_\eps:=f*\rho_\eps$ for all~$\eps>0$, arguing as in the proof of~\cite{CS19}*{Theorem~3.22} we have that $f_\eps\to f$ in~$S^{\alpha,p}(\R^n)$ as~$\eps\to0^+$. 
By Young's inequality, we have that $f_\eps\in S^{\alpha,p}(\R^n)\cap W^{1,p}(\R^n)\cap\Lip_b(\R^n)$ for all~$\eps>0$. 
Thus $S^{\alpha,p}(\R^n)\cap W^{1,p}(\R^n)\cap\Lip_b(\R^n)$ is a dense subset of $S^{\alpha,p}(\R^n)$. 
Hence, by Step~1, we get that also $L^{\alpha,p}(\R^n)$ is a dense subset of~$S^{\alpha,p}(\R^n)$. Since $L^{\alpha,p}(\R^n) = S^{\alpha,p}_0(\R^n) = \closure[-1]{C^\infty_c(\R^n)}^{\,\|\cdot\|_{S^{\alpha,p}(\R^n)}}$ (see \cite{SS15}*{Theorem 1.7}), the conclusion follows.
\end{proof}

\section{Some properties of \texorpdfstring{$S^{\alpha,p}(\R^n)$}{Sˆ{alpha,p}(Rˆn)}}
\label{sec:props_of_S_alpha_p}

In this section, we collect some additional properties of the space $S^{\alpha,p}(\R^n)$. We begin with the following result, whose proof is very similar to the one of~\cite{CS19}*{Proposition~3.3} and is left to the reader.

\begin{proposition}\label{res:lsc_norm_S_alpha_p}
Let $\alpha\in(0,1)$ and $p\in[1,+\infty)$. If $(f_k)_{k\in\N}\subset S^{\alpha,p}(\R^n)$ is such that 
\begin{equation*}
\liminf_{k\to+\infty}\|\nabla^\alpha f_k\|_{L^p(\R^n;\,\R^n)}
<+\infty	
\end{equation*}
and $f_k\to f$ in~$L^p(\R^n)$ as~$k\to+\infty$, then $f\in S^{\alpha,p}(\R^n)$ with
\begin{equation}\label{eq:lsc_norm_S_alpha_p}
\|\nabla^\alpha f\|_{L^p(U;\,\R^{n})}
\le\liminf_{k\to+\infty}\|\nabla^\alpha f_k\|_{L^p(U;\, \R^{n})}
\end{equation}
for any open set $U\subset\R^n$.
\end{proposition}

The following result provides an $L^p$-estimate on translations of functions in~$S^{\alpha,p}(\R^n)$. 
It can be stated by saying that the inclusion $S^{\alpha,p}(\R^n) \subset B_{p,\infty}^{\alpha}(\R^n)$ is continuous, where $B^\alpha_{p,q}(\R^n)$ is the Besov space, see~\cite{L09}*{Chapter~14}. 
For a similar result in the $W^{\alpha,p}(\R^n)$ space, we refer the reader to~\cite{TGV20}.

Thanks to \cref{res:S=L}, this result can be derived from the analogous result already known for functions in~$L^{\alpha,p}(\R^n)$. However, the estimate in~\eqref{eq:Lp_control_on_traslations} provides an explicit constant (independent of~$p$) that may be of some interest. The proof of \cref{res:Lp_control_on_traslations} below can be easily established following the one of~\cite{CS19}*{Proposition~3.14}(and exploiting Minkowski's integral inequality and Theorem \ref{res:L_p_alpha_equal_S_p_alpha}) and we leave it to the reader. 

\begin{proposition}\label{res:Lp_control_on_traslations}
Let $\alpha\in(0,1)$ and $p\in[1,+\infty)$. If $f\in S^{\alpha,p}(\R^n)$, then 
\begin{equation}\label{eq:Lp_control_on_traslations}
\|f(\cdot+y)-f(\cdot)\|_{L^p(\R^n)}
\le\gamma_{n,\alpha}\,|y|^\alpha\,
\|\nabla^\alpha f\|_{L^p(\R^n;\,\R^n)}
\end{equation}
for all $y\in\R^n$, where $\gamma_{n,\alpha}>0$ is as in~\cite{CS19}*{Proposition~3.14}. 
\end{proposition}

A similar result holds for spaces $BV^\alpha(\R^n)$, indeed from~\cite{CS19}*{Proposition~3.14}, one immediately deduces that the inclusion $BV^{\alpha}(\R^{n}) \subset B^{\alpha}_{1,\infty}(\R^{n})$ holds continuously for all $\alpha\in(0,1)$. 
The next result shows that this inclusion is actually strict whenever~$n\ge2$.

\begin{theorem}[$B^{\alpha}_{1,\infty}(\R^{n})\setminus BV^{\alpha}(\R^{n})\ne\varnothing$ for $n\ge2$]\label{thm:strict_inclusion}
	Let $\alpha\in(0,1)$ and $n\ge2$. The inclusion $BV^{\alpha}(\R^{n}) \subset B^{\alpha}_{1,\infty}(\R^{n})$ is strict.
\end{theorem}

\begin{proof}
	By~\cite{CS19}*{Theorem~3.9}, we just need to prove that $B^{\alpha}_{1,\infty}(\R^{n})\setminus L^{\frac n{n-\alpha}}(\R^{n})\ne\varnothing$. 
	Let $\eta_1\in C^\infty_c(\R^n)$ be as in~\eqref{eq:def_eta_function} and \eqref{eq:def_cut_off}, and let $f(x)=\eta_1(x)|x|^{\alpha-n}$ for all $x\in\R^n$. On the one side, we clearly have $f\notin L^{\frac n{n-\alpha}}(\R^{n})$. On the other side, for all $h\in\R^n$ with $|h|<1$, we can estimate
	\begin{align*}
		\int_{\R^n}|f(x+h)-f(x)|\,dx
		&\le
		\int_{\set*{|x|>2|h|}}\big|\eta_1(x+h)|x+h|^{\alpha-n}-\eta_1(x)|x|^{\alpha-n}\big|\,dx\\
		&\quad+2
		\int_{\set*{|x|<3|h|}}\eta_1(x)|x|^{\alpha-n}\,dx\\
		&\le C|h|\int_{\set*{|x|>2|h|}}|x|^{\alpha-n-1}\,dx
		+
		C\int_{\set*{|x|<3|h|}}|x|^{\alpha-n}\,dx\\
		&=
		C|h|
		\int_{2|h|}^{+\infty}r^{\alpha-2}\,dr
		+
		C\int_{0}^{3|h|} r^{\alpha-1}\,dr
		=C|h|^\alpha,
	\end{align*}
	where $C>0$ is a constant depending only on~$n$ and~$\alpha$ (that may vary from line to line). Thus $f\in B^\alpha_{1,\infty}(\R^n)$ and the conclusion follows.
\end{proof}

We conclude with the following result which, again, can be derived from the theory of Bessel potential spaces. We state it here since our distributional approach provides explicit constants (independent of~$p$) in the estimates that may be of some interest. The proof is very similar to the one of~\cite{CS19-2}*{Proposition~3.12} and we leave it to the interested reader. 

\begin{proposition}[$S^{\beta,p}(\R^n)\subset S^{\alpha,p}(\R^n)$ for $0<\beta<\alpha<1$]
\label{res:Davila_estimate_S_alpha_p}
Let $0<\beta<\alpha<1$ and $p\in(1,+\infty)$. If $f\in S^{\alpha,p}(\R^n)$, then $f\in S^{\beta,p}(\R^n)$ with
\begin{equation}\label{eq:Davila_estimate_S_alpha_p}
\|\nabla^\beta f\|_{L^p(A;\,\R^n)}
\le\frac{n\omega_n\mu_{n,1+\beta-\alpha}}{n+\beta-\alpha}
\left(\frac{r^{\alpha-\beta}}{\alpha-\beta}\,\|\nabla^\alpha f\|_{L^p(\overline{A_r};\,\R^n)}
+c_{n,\alpha}\,\frac{r^{-\beta}}{\beta}\,\|f\|_{L^p(\R^n)}\right)
\end{equation}
for any $r>0$ and any open set $A\subset\R^n$, where $A_r:=\set*{x\in\R^n : \dist(x,A)<r}$ and $c_{n,\alpha}>0$ is a constant depending only on~$n$ and~$\alpha$. In particular, we have
\begin{equation}\label{eq:Davila_estimate_S_alpha_p_right}
\|\nabla^\beta f\|_{L^p(\R^n;\,\R^n)}
\le c_{n,\alpha}\,\frac{\mu_{n,1+\beta-\alpha}}{\beta(\alpha-\beta)(n+\beta-\alpha)}\,
\|\nabla^\alpha f\|_{L^p(\R^n;\,\R^n)}^{\beta/\alpha} \|f\|_{L^p(\R^n)}^{(\alpha - \beta)/\alpha},
\end{equation}
where $c_{n,\alpha}>0$ is a constant depending only on~$n$ and~$\alpha$. In addition, if $p\in\left(1,\frac{n}{\alpha-\beta}\right)$ and $q=\frac{np}{n-(\alpha-\beta)p}$, then
\begin{equation}\label{eq:representation_formula_S_alpha_p}
\nabla^\beta f
=I_{\alpha-\beta}\nabla^\alpha f
\quad
\text{a.e.\ in~$\R^n$}
\end{equation}
and $\nabla^\beta f\in L^q(\R^n;\R^n)$.
\end{proposition}

\section{Continuity properties of the map \texorpdfstring{$\alpha\mapsto\nabla^\alpha$}{alpha -> nablaˆalpha}}
\label{sec:continuity_props_nabla_alpha}

Here we prove the following continuity properties of the fractional gradient operator.

\begin{theorem}[Continuity properties of $\alpha\mapsto\nabla^\alpha$]
Let $\alpha\in(0,1]$ and $p\in[1,+\infty)$. 

\begin{enumerate}[(i)]

\item \label{item:cont_nabla_frac_BV}
If $f\in BV^\alpha(\R^n)$, then the function
\begin{equation*}
(0,\alpha)\ni\beta\mapsto\nabla^\beta f\in L^1(\R^n;\R^n)
\end{equation*}
is continuous.
If $f\in BV^\alpha(\R^n)\cap H^1(\R^n)$, then we also have the continuity at $\beta=0$.

\item\label{item:cont_nabla_frac_S}
If $f\in S^{\alpha,p}(\R^n)$, then the function
\begin{equation*}
(0,\alpha]\ni\beta\mapsto\nabla^\beta f\in L^p(\R^n;\R^n)
\end{equation*}
is continuous.
If $p > 1$, then we also have the continuity at $\beta = 0$.	
\end{enumerate}
\end{theorem}

\begin{proof}
We prove the two statements separately.

\smallskip

\textit{Proof of~\eqref{item:cont_nabla_frac_BV}}.
Let $f\in BV^\alpha(\R^n)$ be fixed.
By~\cite{CS19}*{Theorem~3.32}, we know that $f\in W^{\gamma,1}(\R^n)$ for all $\gamma\in(0,\alpha)$. 
Hence the claimed continuity follows by combining~\cite{CS19-2}*{Lemma~5.1 and Remark~5.2}.
If $f\in BV^\alpha(\R^n)\cap H^1(\R^n)$ the claimed conclusion follows from \cref{rm:direct_proof_intro_eq:frac_limit_1}.

\smallskip

\textit{Proof of~\eqref{item:cont_nabla_frac_S}}.
The continuity at the boundary points $\alpha=0$ and $\alpha=1$ is already proved in~\cref{res:strong_conv_alpha_0}\eqref{item:strong_conv_alpha_0_Lp} and~\cite{CS19-2}*{Theorem~4.10} respectively, so we can assume $\alpha\in(0,1)$. 
We can further assume $p>1$ since, thanks to the continuous embedding $S^{\alpha,1}(\R^n)\subset BV^\alpha(\R^n)$ established in~\cite{CS19}*{Theorem~3.25}, the case $p=1$ is already proved in~\eqref{item:cont_nabla_frac_BV}.
If $f\in\Lip_c(\R^n)$, then one can prove that $\nabla^\beta f\to\nabla^\alpha f$ in $L^p(\R^n;\R^n)$ as $\beta\to\alpha$ with the strategy adopted in~\cite{CS19-2}*{Section~5.1} up to some minor modifications that we leave to the interested reader.
For a general $f\in S^{\alpha,p}(\R^n)$, the claimed continuity follows from \cref{res:L_p_alpha_equal_S_p_alpha} and \cref{res:interpolation_MH}\eqref{item:interpolation_MH_p} arguing as in the proof of \cref{res:interpolation_MH}\eqref{item:interpolation_MH_p}.
\end{proof}


\end{document}